\documentclass[reqno,11pt]{amsart}
\usepackage{amsmath, latexsym, amsfonts, stmaryrd, amssymb, amsthm, amscd, array, caption}
\usepackage[toc,page]{appendix}
\usepackage[utf8]{inputenc}
\usepackage{graphics,epsf,psfrag,epsfig,dsfont}
\usepackage[colorlinks=true, pdfstartview=FitV, linkcolor=blue, citecolor=blue, urlcolor=blue,pagebackref=false]{hyperref}

\numberwithin{equation}{section}
\newtheorem{theorem}{Theorem}[section]
\newtheorem{lemma}[theorem]{Lemma}
\newtheorem{proposition}[theorem]{Proposition}

\newtheorem{rem}[theorem]{Remark}
\newtheorem{claim}[theorem]{Claim}
\newtheorem{definition}[theorem]{Definition}

\newtheorem*{mainassumption*}{Main assumption}

\renewenvironment{proof}[1][Proof]{\begin{trivlist}
\item[\hskip \labelsep {\bfseries #1}]}{\qed\end{trivlist}}

\DeclareMathSymbol{\leqslant}{\mathalpha}{AMSa}{"36} 
\DeclareMathSymbol{\geqslant}{\mathalpha}{AMSa}{"3E} 
\DeclareMathSymbol{\eset}{\mathalpha}{AMSb}{"3F}     
\newcommand{\dd}{\,\text{\rm d}}             

\newcommand{\ind}{\mathbf{1}}

\newcommand{\e}{{\rm e}}
\newcommand{\Ceff}{C_{{\rm eff}}}

\newcommand{\cB}{{\ensuremath{\mathcal B}} }

\newcommand{\cP}{{\ensuremath{\mathcal P}} }

\newcommand{\cH}{{\ensuremath{\mathcal H}} }

\newcommand{\cT}{{\ensuremath{\mathcal T}} }
\newcommand{\cD}{{\ensuremath{\mathcal D}} }

\newcommand{\cJ}{{\ensuremath{\mathcal J}} }
\newcommand{\cG}{{\ensuremath{\mathcal G}} }
\newcommand{\cW}{{\ensuremath{\mathcal W}} }
\newcommand{\cK}{{\ensuremath{\mathcal K}} }
\newcommand{\cS}{{\ensuremath{\mathcal S}} }

\newcommand{\bbE}{{\ensuremath{\mathbb E}} }
\newcommand{\E}{{\ensuremath{\mathbb E}} }

\newcommand{\bbN}{{\ensuremath{\mathbb N}} }
\newcommand{\N}{{\ensuremath{\mathbb N}} }

\newcommand{\bbP}{{\ensuremath{\mathbb P}} }
\renewcommand{\P}{{\ensuremath{\mathbb P}} }

\newcommand{\bbR}{{\ensuremath{\mathbb R}} }

\newcommand{\bbZ}{{\ensuremath{\mathbb Z}} }
\newcommand{\Z}{{\ensuremath{\mathbb Z}} }

\newcommand{\ga}{\alpha}

\newcommand{\gd}{\delta}
\newcommand{\gep}{\varepsilon}       
\newcommand{\gp}{\varphi}

\newcommand{\gG}{\Gamma}

\newcommand{\go}{\omega}

\newcommand{\gl}{\lambda}

\newcommand{\gU}{\Upsilon}

\newcommand{\1}[1]{\mathbf 1_{\left\{#1\right\}}}

\makeatletter
\def\captionfont@{\footnotesize}
\def\captionheadfont@{\scshape}

\long\def\@makecaption#1#2{%
  \vspace{2mm}
  \setbox\@tempboxa\vbox{\color@setgroup
    \advance\hsize-6pc\noindent
    \captionfont@\captionheadfont@#1\@xp\@ifnotempty\@xp
        {\@cdr#2\@nil}{.\captionfont@\upshape\enspace#2}%
    \unskip\kern-6pc\par
    \global\setbox\@ne\lastbox\color@endgroup}%
  \ifhbox\@ne 
    \setbox\@ne\hbox{\unhbox\@ne\unskip\unskip\unpenalty\unkern}%
  \fi
  \ifdim\wd\@tempboxa=\z@ 
    \setbox\@ne\hbox to\columnwidth{\hss\kern-6pc\box\@ne\hss}%
  \else 
    \setbox\@ne\vbox{\unvbox\@tempboxa\parskip\z@skip
        \noindent\unhbox\@ne\advance\hsize-6pc\par}%
\fi
  \ifnum\@tempcnta<64 
    \addvspace\abovecaptionskip
    \moveright 3pc\box\@ne
  \else 
    \moveright 3pc\box\@ne
    \nobreak
    \vskip\belowcaptionskip
  \fi
\relax
}
\makeatother
\def\writefig#1 #2 #3 {\rlap{\kern #1 truecm
\raise #2 truecm \hbox{#3}}}

\newtheorem{innercustom}{Assumption}
\newenvironment{custom}[1]
  {\renewcommand\theinnercustom{#1}\innercustom}
  {\endinnercustom}

\author[Q.~Berger]{Quentin Berger}
\address{LPSM, Sorbonne Universit\'e,  UMR 8001\\
Campus Pierre et Marie Curie, Bo\^ite courrier 158, 4 Place Jussieu, 75252 Paris Cedex 05}
\email{quentin.berger@sorbonne-universite.fr}

\author[M.~Salvi]{Michele Salvi}
\address{CMAP, \'Ecole Polytechnique \& INRA\\
Route de Saclay, 91128 Palaiseau Cedex }
\email{michele.salvi@polytechnique.edu}

\begin{document}

\title[Scaling limit of sub-ballistic 1D RW BiRC]{Scaling limit of sub-ballistic
1D Random Walk among Biased  Conductances:
\textit{a story of wells and walls}}


\thanks{Q.~Berger acknowledges the support of ANR grant SWiWS (ANR-17-CE40-0032-0). \\
\indent	M.~Salvi acknowledges the support of ANR grant CADENCE (ANR-16-CE32-0007-01).}

\begin{abstract}
We consider a one-dimensional random walk among biased i.i.d.\ conductances, in the case where the random walk is transient but sub-ballistic: this occurs when the conductances have a heavy-tail  at $+\infty$ or at $0$.
We prove that the scaling limit of the process is the inverse of an $\alpha$-stable subordinator, which indicates an aging phenomenon, expressed in terms of the generalized arcsine law.
In analogy with the case of an i.i.d.\ random environment studied in details in \cite{ESZ09a,ESTZ13}, some ``traps'' are responsible for the slowdown of the random walk. However, the phenomenology is somehow different (and richer) here. In particular, three types of traps may occur, depending on the fine properties of the tails of the conductances: (i) a very large conductance (a \emph{well} in the potential); (ii) a very small conductance (a \emph{wall} in the potential); (iii) the combination of a  large conductance followed shortly after by a small conductance (a \emph{well-and-wall} in the potential). 
\end{abstract}

\maketitle
%

\section{Introduction}

Random walks in random environment have been  studied extensively over the past forty years, both from a physical and a mathematical perspective.
In dimension one, they were introduced in the mathematical literature by Solomon in~\cite{Sol75}, who gave a criterion for transience/recurrence, and identified three possible regimes: recurrent, transient with positive speed, transient with zero speed.

The behavior of recurrent random walks in (genuine) random environment has been further studied by Sinai in \cite{Sinai82}, who showed a strong slowdown of the walk, with an unusual scaling $(\log n)^2$. The limiting law has then been identified in \cite{Gol86,Kes86}, and the scaling limit of the process is an interesting singular diffusion in random environment, see~\cite{FIN02}.
The case of transient sub-ballistic random walks in i.i.d.\ random environment has been considered in \cite{KKS75}, and then in \cite{ESZ09a, ESZ09,ESTZ13}. In particular,  the scaling limit of the walk is the inverse of an $\ga$-stable subordinator, and the random walk is shown to exhibit an aging phenomenon, \textit{i.e.}\ the two-time correlation function converges when both times tend to infinity with a fixed ratio, the limit being a non-trivial function of the ratio (aging has been extensively studied in the physics literature, in the context of out-of-equilibrium glassy disordered systems, see \cite{BCKM98} for an overview).

In the present paper we consider the case of a one-dimensional random walk among i.i.d.\ random conductances: in this case, the random walk is recurrent. One central question in Physics (see e.g.~\cite{NJBODG10}) is to understand the effect of an external field of intensity $\lambda$ (producing a bias in the conductances) on the behavior of the walk.
In particular, as soon as $\lambda>0$, the walk is transient to the right. Under some integrability condition of the conductances, the walk is ballistic; this regime has been studied for instance in \cite{LD16,FS18} (see also references therein).
Here we investigate the regime where the random walk among biased random conductances is sub-ballistic: we show that  its scaling limit is the inverse of an $\ga$-stable subordinator, confirming the universality of the aging phenomenon, cf.~\cite{BAC08, Zin09}. 
	
As mentioned above, a similar behavior is observed for sub-ballistic transient random walks in i.i.d.\ environment. We show however that as soon as the independence hypothesis on the environment is dropped (conductances introduce automatically a strong dependence between neighbours), the trapping mechanism for the walk changes dramatically and might become more irregular.
We also point out that convergence to the inverse of an $\alpha$-stable subordinator has also been shown for walks among biased random conductances in dimension $d\geq 2$  in \cite{FK16}. The one-dimensional case, though, presents once more a wilder zoology of possible trapping mechanisms.
	In particular, we might have three different kind of traps, given either by a very large conductance, either by a very small conductance or by a combination of the two. The kind of traps that contributes the most to the slowdown of the walk  is given by the fine properties of the tails of the conductances at $+\infty$ and at $0$. The distribution of the depth of such traps (this is roughly the amount of time that the walk will spend on the trap) has to be studied in terms of the product of random variables with regularly varying tails, which might exhibit unexpected behaviors.

\subsection{General setting of the paper}

Let $(X_n)_{n\in \N}$ be a discrete-time random walk on an environment $\go$ given by a sequence $\{c_{x}\}_{x\in \bbZ}$ of random conductances that are i.i.d.\  under  measure $\P$. 
For a fixed realization of $\go$, we call $P^\go$ the law of the random walk
 $(X_n)_{n\in\N}$ which  starts at the origin and has transition probabilities
\begin{align*}
P^{\go} \big( X_{n+1} =x+1  \mid X_n= x\big) 
	= \frac{c_x}{c_{x-1} + c_x} \,, 
	\quad  P^{\go} \big( X_{n+1} =x-1  \mid X_n= x\big) = \frac{c_{x-1}}{c_{x-1}+c_x} \, .
\end{align*}
In the usual language of random walks in random environment, where the probability of jumping from $x$ to $x+1$ is called $\go_x$, we have $\go_x:= \frac{c_x}{c_{x-1}+c_x}$.
For $x\in \bbZ$ define $\rho_x := \frac{1-\go_x}{\go_x} = \frac{c_{x-1}}{c_x}$: since the $\{c_x\}_{x\in \bbZ}$ are independent, we get that $\bbE[\log \rho_0] =0$ (provided that $\log c_0$ is integrable). As a consequence, Theorem 2.1.2 in \cite{TZ04} tells us that $(X_n)_{n \geq 0}$ is  $P^{\go}$-a.s.~recurrent for $\bbP$-a.e.\ $\go$.

We stress that while in the seminal paper by Solomon \cite{Sol75} (and in many other works) the $\{\go_x\}_{x\in \bbZ}$ are taken i.i.d., this is clearly not true anymore in the case of conductances. On the other hand, the conductance model boasts the important feature of being reversible with respect to the measure $\pi(x) = c_{x-1}+c_x$.
We refer to \cite{B11} for an extensive account on the random conductance model, which has been widely studied in the literature.

Now, we introduce an external field of intensity $\gl>0$, which corresponds to a tilt of the conductances:
\begin{equation}
\label{eq:biased_cond}
c_x^{\gl} = \e^{\gl x } c_x  \, .
\end{equation} 
The process $(X_n)_{n\in\N}$ on these tilted conductances is called random walk among Biased Random Conductances (which we abbreviate as BiRC). 
In the BiRC setting, for $x\in \bbZ$ we have $\rho_x =\rho_x(\lambda) := \e^{- \lambda} c_{x-1}/c_x$ (here and in the rest of the paper, we drop the dependence of $\rho_x$ on $\lambda$). We get that $\bbE[\log \rho_0] = -\lambda <0$ (provided that $\log c_0$ is integrable),  so that $(X_n)_{n\in \bbN}$ is $P^\go$-a.s.\ transient to the right for $\P$-a.e.~$\go$, see \cite[Thm~2.1.2]{TZ04}.
Additionally, the asymptotic velocity $\mathbf{v}(\gl) = \lim_{n\to+\infty} {X_n}/{n}$  exists and is $\bbP\otimes P^\go$-a.s.~equal to
\[
\mathbf{v}(\gl)   =  \frac{1}{\bbE[\bar S]} \qquad \text{with } \quad\bbE[\bar S] = 1 + 2 \bbE[c_0]\bbE[1/c_0] \frac{\e^{-\gl}}{1-\e^{-\gl}} \, ,
\]
cf.~\cite[Thm.~2.1.9]{TZ04}.
Hence, for $\lambda>0$, $\mathbf{v}(\gl)>0$ if and only if $\bbE[c_0] <+\infty$ and $\bbE[1/c_0] <+\infty$.
The zero velocity regime can therefore occur for two different reasons:  
the conductances  have some heavy tail at $+\infty$ or they have some heavy tail at $0$.

In our previous work \cite{BS17}, we consider the sub-ballistic regime, where  $\bbE[c_0] = +\infty$ or $\bbE[1/c_0] = +\infty$, and we find the correct order for the scaling of $X_n$. More precisely, we assume that there are some $\ga_\infty ,\ga_0\in[0,+\infty]$ with $\alpha := \min(\alpha_0,\alpha_{\infty}) \le 1$, such that
\begin{equation}
\label{hyp:tail1}
\lim_{t\to+\infty} \frac{\log \bbP(c_0 >t)}{\log t} 
	= -\alpha_\infty \,, \qquad 
\lim_{\gep \to 0} \frac{\log \bbP(c_0<\gep)}{\log \gep} 
	= \alpha_0 \, .
\end{equation}
In Theorem~1.1 of \cite{BS17} we prove that for each $\gl>0$
\[
\lim_{n\to\infty} \frac{\log X_n}{\log n} = \alpha  \qquad \bbP\otimes P^\go-a.s.
\]

The main goal of the present paper is to make this result much more precise, and prove the convergence of $X_n$ rescaled by $n^\ga$ (properly corrected by a slowly varying function) to the inverse of an $\alpha$-stable subordinator. 
In  \cite{ESZ09, ESTZ13} the authors prove this type of result for  one-dimensional random walks in random environment, where the transition probabilities $\{\go_x\}_{x\in\Z}$ are i.i.d. Their study is based on the analysis of the so-called potential of the environment: the walk is slowed-down by large regions with an atypical value of the potential. As we shall see, this is in sharp contrast with our setting, where the trapping parts of the environment are determined by one or at most two abnormal values of the conductances, see the discussion in Section \ref{commentsontraps}.

\subsection{Main assumption and main result}
Having only \eqref{hyp:tail1} is not sufficient for describing a functional limit theorem for the BiRC. In the present paper we will make the following stronger assumption. 

For two functions $f$ and $g$, we write that $f(t)\sim g(t)$ if $\lim_{t\to\infty} f(t)/g(t)=1$.
\begin{custom}{``Traps''}
\label{mainassumption}
There exist some slowly varying functions $L_\infty(\cdot), L_0(\cdot) $ and some $\ga_\infty, \ga_0>0 $ such that, for $t>1$ and $\varepsilon<1$,
\begin{equation}
\label{hyp:tail2}
\bbP(c_0 >t) = L_{\infty}(t)\,  t^{-\alpha_{\infty}} \,  \qquad \text{and} \, \qquad
 \bbP(c_0 <\gep) = L_0 (1/\gep)\,  \gep^{\alpha_0}\, ,
\end{equation}
with $\alpha:=\min\{\ga_0,\ga_\infty\}<1$.
If $\ga_0=\ga_{\infty}$,  we additionally assume that $v\mapsto L_{\infty}(\e^{v})$, $ v\mapsto L_0(\e^{v})$ are regularly varying with respective indices $\gamma_{\infty},\gamma_0$, 
\textit{i.e.}\ there are slowly varying functions $\gp_0, \gp_{\infty}$ such that $L_{\infty}(x) = \gp_{\infty}( \log x) (\log x)^{\gamma_{\infty}}$,  $L_{0}(x) = \gp_{0}( \log x) (\log x)^{\gamma_0}$.
We suppose that $\gamma_{\infty}\neq -1,\gamma_0\neq -1$.
\end{custom}


As it will be clear below, large values of $\rho_x$ will be associated to ``traps'' responsible for the slowdown of the random walk $(X_n)_{n\in\N}$: it is therefore crucial to obtain the tail behavior of $\rho_x$, in order to be able to quantify the depth of the traps.
Assumption~\ref{mainassumption} serves that purpose, and is the condition that enables one to obtain the behavior of $\bbP(\rho_0 > t)$ as $t\to+\infty$. Indeed, Corollary~5 in \cite{C86} gives the the following proposition:
\begin{proposition}
\label{prop:rho0}
Assume that Assumption~\ref{mainassumption} holds. Then, there is a slowly varying function $\psi(\cdot)$ such that 
\begin{equation}
\bbP(\rho_0 >t) = \bbP\big( \e^{-\lambda} c_{-1}/c_0 >  t\big) \sim \psi(t) t^{-\ga} \, \qquad \text{ as } t\to+\infty \, .
\label{tail:rho}
\end{equation}
The function $\psi(t)$ is asymptotic to $\e^{\lambda \ga}$ times:
\begin{align*}
\arraycolsep=2pt\def\arraystretch{1.7}
\begin{array}{ll}
\bbE[c_0^{\ga}] L_{0}(t) \ind_{\{\bbE[c_0^{\ga}] <+\infty \}}  + \bbE[1/c_0^{\ga}] L_{\infty}(t)\ind_{\{\bbE[1/c_0^{\ga}] <+\infty \}}
	&\mbox{if $\,\bbE[c_0^{\ga}]<+\infty$ or $\bbE[1/c_0^{\ga}]  <+\infty$;}\\
\ga   \tfrac{\Gamma(1+\gamma_0) \Gamma(1+\gamma_{\infty})}{\Gamma(2+\gamma_0+\gamma_{\infty})} (\log t) L_0(t) L_{\infty}(t)
	&\mbox{if $\,\bbE[c_0^{\ga}]=+\infty$ and $\bbE[1/c_0^{\ga}]  =+\infty$.} 
\end{array}
\end{align*}
 \noindent
In view of Assumption~\ref{mainassumption}, we have that $v\mapsto \psi(\e^v)$ is regularly varying if  $\ga_0=\ga_{\infty}$.
\end{proposition}
We stress that in the case $\ga_0=\ga_{\infty}$, the fact that  $L_{\infty}(\e^{t}),L_0(\e^{t})$ are regularly varying is crucial to obtain Proposition~\ref{prop:rho0}---however, the case $L_{\infty/0}(t)=\exp( \pm (\log t)^a)$ with $a\in (0,1)$ has to be excluded.
Some more comments on the tail of $\rho_0$ are made in Appendix~\ref{sec:rho}.

We are now ready to state our main theorem:
\begin{theorem}\label{thm:Biased2}
Suppose that Assumption~\ref{mainassumption} holds, and recall \eqref{tail:rho}.
Then, on the space $\mathbf{D}([0,1])$ of c\`adl\`ag functions $[0,1]\to \bbR$ with the uniform topology, we have the following convergence in distribution, as $n\to+\infty$: under $\P\otimes P^\go$ 
\[ \Big( \frac{X_{\lfloor un\rfloor}}{ n^{\ga}/ \psi(n) } \Big)_{u \in [0,1]} 
\	\Longrightarrow  \ \frac{\sin(\pi \ga)}{\pi\ga\bbE[\zeta^{\ga}] }  \big( (\mathcal{S}_{\ga})^{-1}(u) \big)_{u\in[0,1]} \,. \]
Here $(\mathcal{S}_{\ga}(u))_{u\geq 0}$ is an $\ga$-stable subordinator satisfying $\mathbf{E}[\e^{- t \mathcal{S}_{\ga}(u)}] = \e^{- u t^{\ga}}$, and $\zeta =\zeta_{\lambda}(\go)$ is the random variable defined in \eqref{def:zeta} (if $\bbE[c_0^{\ga}],\bbE[1/c_{0}^{\ga}]=+\infty$, we get that $\zeta=2$).
\end{theorem}

This scaling limit is similar to that found in  \cite{ESTZ13} (and \cite{Zin09}), but we stress that the trapping mechanism is different. 
Note also that the inverse of an $\alpha$-stable subordinator is easily shown to exhibit an aging phenomenon expressed in terms of generalized arcsine laws: for any fixed $h>1$, we have
\begin{equation}
\label{aging}
\lim_{t\to+\infty} \bbP\Big( (S_{\ga})^{-1}(th) = (S_{\ga})^{-1} (t) \Big)  = \frac{\sin(\pi \ga)}{ \pi} \int_{0}^{1/h} y^{-(1-\ga)} (1-y)^{-\ga} \dd y \, ,
\end{equation}
see e.g.~\cite{Ber99} (arcsine laws are shown for last-passage times, but the proofs easily adapt to the above statement).
An analogous aging result for the random walk $(X_n)_{n\in \N}$ should hold:
using the same techniques as in \cite{ESZ09a}, one could prove that, for any sequence $j_n\to+\infty$ and any $h>1$, $\bbP(  | X_{\lfloor h n \rfloor} - X_n| \leq j_n )$ converges as $n\to+\infty$ to the right-hand side of \eqref{aging}.
Our proof actually shows some \emph{localization} property of the walk, which is typically ``stuck'' at a trap, with a value of $\rho_x$ of order $d_n$ (see~\eqref{def:dn}) which is regularly varying with index $1/\ga$; this is the key idea to prove aging result.

\begin{rem}
In \cite{BS17}, we also considered the biased range-one Mott random walk in the sub-diffusive regime. This process is the simplified version of the Variable--Range  Hopping model used in Physics (cfr.~\cite{FGS16}), and essentially corresponds  to a random walk among random conductances with a different type of bias. In \cite{BS17} we identify the correct scaling exponent for the walk as a specific function of the field intensity $\lambda$ (in contrast with the BiRC where the scaling $\alpha$ does not depend on $\lambda$). The techniques developed in the present work would allow us to push this result further and obtain the scaling limit of the range-one Mott walk, but the result would be less interesting since in that model the conductances are bounded from above and hence only one kind of trap is present.
\end{rem}

In dimension $d\geq 2$, the rescaled sub-ballistic BiRC also converges to the inverse of a stable subordinator, as proven in  \cite{FK16}. In that case, the sub-ballistic behavior is due to very large conductances, the walk being trapped between the two endpoints of the edge.
On the other hand, small conductances do not represent a problem for the walk, which will typically go around them. In dimension $d=1$ however, the lattice geometry forces the walk to pass through each edge of the positive half line, so the walk can be slowed down both by large conductances (as in the higher dimensional case) and by very small, hard-to-overjump, conductances. Furthermore, in a particular range of the parameters, a new kind of trap arises, as the combination of a small conductance followed shortly after by a large conductance. The fine properties of the tails of $c_0$ at $+\infty$ and at $0$ dictate which kind of traps represent the biggest contribution to the slow-down.

\subsection{The different types of traps}\label{commentsontraps}

In his seminal work, Sinai \cite{Sinai82} described the notion of traps in the environment via ``valleys'' in the potential function $V$, defined by $V(x) = \sum_{i=1}^x \log \rho_i$ for $x\in \bbZ_+$ ($V(0)=0$ and a minus sign is added for $x\in \bbZ_-$).
In the case of i.i.d.\ $\go_i$, the potential $V$ is itself a random walk, and valleys in the potential $V$ are due to large regions where the sum of $\log \rho_i$'s is abnormally large, cf.~\cite{KKS75} (see also \cite{ESTZ13}). In that setting, if $H$ denotes the height of a valley in the potential (for a precise definition, see \cite[Sec.~3]{ESTZ13}), then a key result by Iglehart \cite{Igle72} gives that $\bbP(H>v) \sim C_{I} \e^{- \kappa v}$, for some specific $\kappa$ and some explicit constant $C_I$.

This is in sharp contrast with what happens in the BiRC setting: as already noticed in \cite[\S~1.3]{BS17},  we have here  $V(x) = \log c_0 - \lambda x - \log c_x $ (for $x\geq 1$), and valleys are caused by isolated large values of $V(x)-V(x-1) = \log \rho_x$.
Hence, Proposition~\ref{prop:rho0} is crucial in the understanding of the depth of traps, and the tail behavior  of $c_0$ and $1/c_0$ plays a key role in the deviations for $\log \rho_x$, and it gives rise to a much richer phenomenology than in the case of i.i.d.\ $\go_i$'s. In particular, the probability of observing a large valley in the potential behaves as $\bbP(\log \rho_0 > v) \sim \psi(\e^v) \e^{- \ga v}$.
The extra slowly varying function $\psi(\cdot)$ depends on the fine asymptotics of the tails \eqref{hyp:tail2} of $c_0$ and $1/c_0$, cf.~Proposition~\ref{prop:rho0}, and may be ruled by different types of mechanisms that need to be treated separately.
Propositions~\ref{prop:conditional1}-\ref{prop:conditional2} below consider the distribution of $c_{-1}, c_0$ conditionally on having $\rho_0 = \e^{-\lambda} c_{-1}/c_0>t$, as $t\to+\infty$: this enables us to understand whether large values of $\rho_x$ are typically due to large values of $c_{x-1}$ (wells), small values of $c_x$ (walls), or a combination of a large value of $c_{x-1}$ and of a small value of $c_{x}$ (well-and-walls), see Figure~\ref{figure1} for an illustration.

We will distinguish two cases.
\begin{custom}{``Simple Traps''}
\label{assumptionA}
The distribution of the conductances satisfy Assumption \ref{mainassumption}, and 
	$\E[c_0^\alpha]<\infty$ or $\E[1/c_{0}^\alpha]<\infty$.
\end{custom}

\begin{custom}{``Well-and-walls''}
\label{assumptionB}
The distribution of the conductances satisfy Assumption \ref{mainassumption}, and 
	$\E[c_0^\alpha]=\infty$ and $\E[1/c_{0}^\alpha]=\infty$.
\end{custom}

We state here some results that will be proven in Appendix~\ref{sec:rho}.
Let us introduce some notation. If $\bbE[c_{0}^{\ga}]<+\infty$, then  $\bar c_{-1}$ is a r.v.\ with c.d.f.\ $F_{\bar c_{-1}} (u) = \frac{1}{\bbE[c_{-1}^{\ga}]} \bbE[c_{-1}^{\ga} \ind_{\{c_{-1} \leq u\}}]$. If $\bbE[1/c_0^{\ga}]<+\infty$, then  $1/\bar c_0$ is a r.v.\ with c.d.f.\ $F_{1/\bar c_0} (u) = \frac{1}{\bbE[1/c_0^{\ga}]} \bbE[1/c_0^{\ga} \, \ind_{\{1/c_0 \leq u\}}]$.

\begin{proposition}
\label{prop:conditional1}
Suppose that Assumption~\ref{assumptionA} holds.
Then the distribution of $(c_{-1},1/c_0)$, conditionally on $\rho_0>t$, converges as $t \to+\infty$ to the r.v.
\begin{equation}
\label{limitconditional}
(1-B) \cdot (\bar c_{-1} , +\infty) + B \cdot (+\infty, 1/\bar c_0) \, ,
\end{equation}
where $B$ is a Bernoulli r.v.\ independent of $\bar c_{-1}, 1/\bar c_0$, with parameter $q\in [0,1]$. If $\bbE[1/c_{0}^{\ga}]=+\infty$, then $q=0$; if $\bbE[c_0^{\ga}] = +\infty$ then $q=1$; if $\bbE[c_{0}^{\infty}], \bbE[1/c_0^{\ga}] < +\infty$ ($\ga=\ga_{\infty}=\ga_0$) then $q= \lim_{t\to+\infty} \frac{\bbE[c_0^{\ga}] L_0(t)  }{\bbE[c_0^{\ga}] L_0(t) + \bbE[1/c_0^{\ga}] L_{\infty}(t)} $ (we assume that this limit exists, to avoid working with subsequences).
\end{proposition}

Proposition \ref{prop:conditional1} tells that under Assumption~\ref{assumptionA}, only two types of traps can occur: conditionally on having a large trap (we postpone the precise definition of traps to Definitions \ref{simpletraps} and \ref{ktraps}, but one can just think of having $\rho_x$ large), then either

\begin{itemize}
\item[(i)] $B=0$, $c_{x-1} \asymp 1$, $c_x \ll 1$, corresponding to a \emph{wall} in the potential~$V$;

\item[(ii)] $B=1$, $c_{x-1} \gg 1$, $c_{x} \asymp 1$, corresponding to a \emph{well}  in the potential~$V$.
\end{itemize}

\noindent
If $q=0$ (e.g.\ if $\bbE[1/c_0^{\ga}] =+\infty$), then $B=0$ and only walls can occur. If $q=1$ (e.g.\ if $\bbE[c_0^{\ga}] =+\infty$) then only wells can occur. If $q\in(0,1)$ (e.g.\ if $\ga =\ga_{\infty} = \ga_0$ with $L_0(t) \sim c L_\infty(t)$ for some constant $c>0$), then both walls or wells may occur, with respective probability $1-q$ and $q$, but not simultaneously.

\begin{proposition}
\label{prop:conditional2}
Suppose that Assumption~\ref{assumptionB} holds.
Then the distribution of $(c_{-1},1/c_0)$, conditionally on $\rho_0>t$, converges as $t \to+\infty$ to the r.v.\ $(+\infty,+\infty)$.
\end{proposition}

Proposition \ref{prop:conditional2} tells that under Assumption~\ref{assumptionB}, only one type of traps can occur: conditionally on having a large trap (say $\rho_x$ large), then we necessarily have

\begin{itemize}
\item[(iii)] $c_{x-1} \gg 1$, $c_x\ll 1$, corresponding to a \emph{well-and-wall} in the potential~$V$.
\end{itemize}

\noindent
It turns out that under Assumption~\ref{assumptionB}, also ``$k$-distant'' well-and-walls traps may occur, with $k\geq 1$:
they consist of the combination of a large conductance $c_{x-1}$ followed shortly after by a small conductance $c_{x+k}$ 
(this makes $\rho_x^{(k)}:= \rho_x \cdots \rho_{x+k} = \e^{-\lambda(k+1)} c_{x-1}/c_{x+k}$ large).

\begin{figure}[h!]
\begin{center}
	\setlength{\unitlength}{0.12\textwidth}  
	\begin{picture}(7.5,2.1)(0,0)
		\put(0,1.5){\includegraphics[scale=0.75]{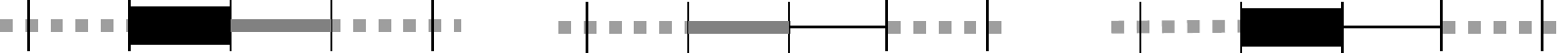}}
		\put(.4,1.36){\small $x\!-\!1 \ \ \;\;x \ \ \ \,\, x\!+\!1$}
		\put(3,1.36){\small $x\!-\!1 \ \ \;\;x \ \ \ \,\, x\!+\!1$}
		\put(5.6,1.36){\small $x\!-\!1 \ \ \;\;x \ \ \ \,\, x\!+\!1$}
		\put(0.58,2.05){ $\gg \! 1 \ \; O(1) \ $}
 		\put(0.69,1.8){ $ \downarrow\quad \ \ \; \downarrow$}
 		\put(3.05,2.05){\;\;$O(1) \ \ll \! 1$}
 		\put(3.27,1.8){ $\downarrow\quad \ \  \downarrow$}
		\put(5.76,2.05){ $\gg\! 1 \ \; \ll\! 1 $}
 		\put(5.88,1.8){ $\downarrow\quad \ \ \; \downarrow$}
 		\put(0.2,0){\includegraphics[scale=1.1]{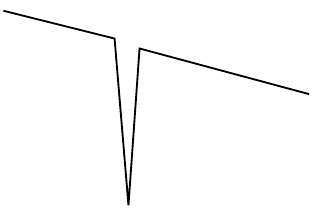}}
 		\put(0.25,.6){\footnotesize a \emph{well}}
 		\put(2.8,-.1){\includegraphics[scale=1.1]{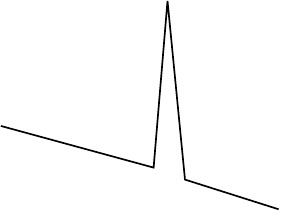}}
 		\put(2.85,.6){\footnotesize a \emph{wall}}
 		\put(5.2,-.2){\includegraphics[scale=1.1]{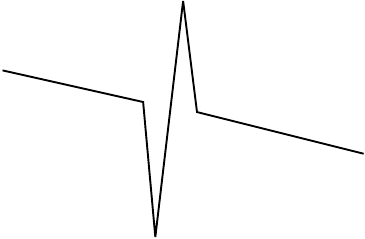}}
 		\put(4.8,.4){\footnotesize a \emph{well-and-wall}}
	\end{picture}
\end{center}
\caption{
\textit{\footnotesize
The three types of traps are represented above, together with the shape of the potential $x\mapsto V(x)$ associated to each type of trap (recall $V(x) = \log c_0 -\lambda x - \log c_x$). A trap, which roughly speaking is a large value of $\log \rho_x$, can occur in three different ways (from left to right): under Assumption~\ref{assumptionA}, we  have either (i)  $c_{x-1}\! \gg\! 1$, $c_x \!\asymp\! 1$ well-shaped trap, or (ii) $c_{x-1}\! \asymp \! 1$, $c_x \! \ll \! 1$ wall-shaped trap; under Assumption~\ref{assumptionB}, we have only (iii) $c_{x-1} \! \gg \! 1$, $c_x  \! \ll \! 1$ well-and-wall-shaped trap. 
}
}
 \label{figure1} 
\end{figure}
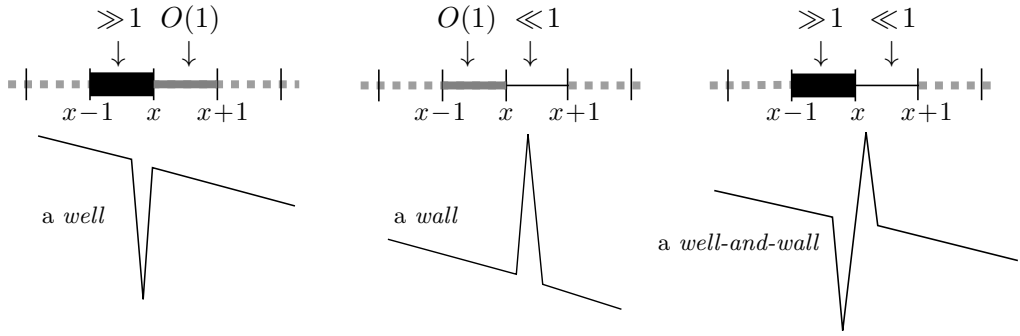

\subsection{Organization of the paper and overview of the proof}

Let us present briefly how the rest of the paper is organized.
In Section~\ref{sec:prelim}, we give some preliminary results: 
we define first-passage times $T_n$, and then state the convergence of the passage-time process towards an $\alpha$-stable subordinator, see  Theorem~\ref{thm:passagetime}. We prove then that the random walk cannot backtrack more than $C \log n$ and conclude the section by showing how Theorem~\ref{thm:Biased2} follows from Theorem~\ref{thm:passagetime}.

The rest of the paper is devoted to the proof of Theorem~\ref{thm:passagetime}: in Section~\ref{sec:def} we give the  precise definition of simple traps and list some important properties. In Section~\ref{sec:reduc1} to~\ref{sec:conv}, we prove Theorem~\ref{thm:passagetime} under Assumption~\ref{assumptionA} (we detail the scheme of the proof below). In Section~\ref{sec:wellandwalls}, we define well-and-wall traps, list their properties and adapt the proof of Theorem \ref{thm:passagetime} to Assumption~\ref{assumptionB}: on the one hand we need to additionally deal with $k$-distant traps, but on the other hand some simplifications occur (due to the fact that $\zeta=2$ in Theorems~\ref{thm:Biased2}-\ref{thm:passagetime}).

Some technical results are collected in the Appendix: in Appendix~\ref{app:formulas} we recall useful formulas for resistor networks; in Appendix~\ref{sec:rho} we deal with the tail of $\rho_0$ and prove some consequences, such as Propositions~\ref{prop:conditional1}-\ref{prop:conditional2}; finally  in Appendix~\ref{dimostrazionitrappole} we collect the proofs of the properties of the traps presented in Sections~\ref{sec:def} and ~\ref{sec:wellandwalls}.

\smallskip

Let us now sketch how the proof of Theorem~\ref{thm:passagetime} works  in the case of Assumption~\ref{assumptionA}, \textit{i.e.}\ we give a summary of Sections~\ref{sec:reduc1}--\ref{sec:conv}.

\begin{description}
\item[\textbf{Part 1}] Section~\ref{sec:reduc1}.
First, we divide the interval $[0,n]$ into blocks of size $C\log n$ and show that the main contribution to the first-passage time $T_n$ comes from so-called \textit{trapping} blocks, which contain a (unique) large value of $\rho_{x}$ (see Section~\ref{sec:def} for a definition).
In Proposition~\ref{blocchetti}, we show that with high $\bbP\otimes P^{\go}$-probability, $T_n \simeq  \sum_{j \in \cJ_n} T({\cB_j})$ where $\cJ_n$ is the set of indices of trapping blocks, and $T(\cB_j)$ is the crossing time of such a block.

\item[\textbf{Part 2}] Section~\ref{sec:crossingontraps}.
Our second step consists in identifying the trapping mechanism on trapping blocks.
We show that, if $\cB$ is a (good) trapping block, then the time to cross the block $\cB$ is dominated by the time to overcome the deep trap where $\rho_x$ is very large: we prove in Proposition~\ref{trappolone} that $T(\cB)/\rho_x \simeq \tau_{\cB}$ with high probability.
Here, $\tau_{\cB}$ is a random variable associated with the environment in the block $\cB$ around $x$.
In particular, to overcome a trap, one needs to account for

- the number of attempts to cross the edge $(x,x+1)$;

- the time between two attempts at crossing the edge $(x,x+1)$;

- the number of times the random walk falls anew in the trap.

\noindent
This is why we write $\tau_{\cB}$ as  $\theta_{\cB}\sum_{j=1}^G E_j$ (cfr.~ \eqref{semplificare}), where $\theta_{\cB} = \theta_{\cB}(\go)$ is the average time between two attempts to cross $(x,x+1)$; $G$ is the geometric random variable 
of parameter $p_{\cB}=p_{\cB}(\go)$ counting the number of times we fall anew in the trap ($p_{\cB}$ is the probability of never returning to $x$, starting from $x+1$); $(E_j)_{j\geq 1}$ are i.i.d.\ exponential random variables of parameter $1$, which approximate the geometric number of attempts to cross $(x,x+1)$, renormalized by $\rho_x$.
Let us stress here that if $c_{x-1} \gg 1$ (well), then $\theta_{\cB}\approx 2$, while if $c_x\ll 1$ (wall), then $p_{\cB} \approx 1$. In the case \ref{assumptionB} of Section \ref{sec:wellandwalls}, we will therefore show that $T(\cB)/\rho_x \simeq 2 \mathbf{e}_{\cB}$, where $\mathbf{e}_{\cB}$ is just one exponential random variable of parameter $1$.

\item[\textbf{Part 3}] Section~\ref{sec:reduc2}.
We prove that the crossing time of all good blocks is dominated by the crossing times of the blocks with values of $\rho_{x} \geq  \gep d_n$,  cf.~Proposition~\ref{prop:fewblocks}, up to an error $\eta$ (that can be made arbitrarily small by taking $\gep$ small). Here $d_n$ is roughly the order of the deepest trap between $0$ and $n$, cf.~\eqref{def:dn}.
The main technical difficulty is that one needs to exclude the possibility that many smaller traps with $\rho_x < \gep d_n$ would slowdown a lot the random walk: in particular, we need to have a (uniform) control on the tail of $\tau_{\cB}$ conditionally on having $\rho_{x}$ large (cf.~Lemma~\ref{lem:tauB}), in order to be sure that the sum of $\tau_{\cB}$ on blocks with $\rho_x < \gep d_n$ cannot be large.

\item[\textbf{Part 4}] Section~\ref{sec:conv}.
Finally, we prove the convergence for the first-passage time reduced to blocks with $\rho_{x} \geq  \gep d_n$.
We show that the positions, depths and crossing times $\tau_{\cB}$ of blocks with $\rho_{x} \geq  \gep d_n$ converge to a Poisson Point Process, with an explicit intensity, cf.~Proposition~\ref{prop:convergence}:
this allows us to show that $T_n/d_n$ converges to a stable distribution. We then conclude proof of Theorem \ref{thm:passagetime} by proving the convergence of the whole process to an $\ga$-stable subordinator.

\end{description}

\section{Preliminaries: relation between $(X_n)_{n\in \N}$ and first-passage times}
\label{sec:prelim}

A classical way of analyzing the properties of the random walk $(X_n)_{n\in \N}$ is to study its first-passage times.
Let
\begin{equation}
T_n :=\inf\{j\geq 1, X_j=n\}
\end{equation}
be the first-passage time to $n$ of the walk.
Recall Proposition~\ref{prop:rho0} and  define the sequence $d_n$, up to asymptotic equivalence, by
\begin{equation}
\label{def:dn}
\bbP\big( \rho_0 > d_n \big) \sim  \psi(d_n) d_n^{-\ga} \sim  1/n \quad \text{ as } n\to \infty ,
\end{equation}
so that $d_n$ corresponds to the order of $\max_{1\leq x \leq n} \rho_x$.
We see that $d_n$ is a regularly varying function with index $\gamma=1/\ga$.
The core of the paper is the proof of the following result.

\begin{theorem}
	\label{thm:passagetime}
	Suppose that Assumption~\ref{mainassumption} holds. Then, on the space $\mathbf{D}([0,1])$ of c\`adl\`ag functions $[0,1]\to \bbR$ with Skorohod $M_1$-topology, we have the following convergence in distribution, as $n\to+\infty$: under $\bbP\otimes P^{\go}$ we have
	\begin{align*}
	\Big(\frac{T_{\lfloor un\rfloor}}{d_n}\Big)_{u \in [0,1]}  \  \Longrightarrow \ \Big( \frac{\pi \ga \bbE[\zeta^{\ga}]}{\sin(\pi \ga)} \Big)^{1/\ga} \big( \mathcal{S}_{\ga}(u) \big)_{u\in[0,1]}  \, .
	\end{align*}
	where $( \mathcal{S}_{\ga}(u))_{u\geq 0}$ is an $\alpha$-stable subordinator which statisfies $\mathbf{E}[e^{-t \cS_{\ga}(u)}] = \e^{- u t^{\ga}}$, and $\zeta$~is a r.v.\ defined in \eqref{def:zeta}.
\end{theorem}

We use here Skorohod $M_1$-topology, which is weaker than the $J_1$-topology (roughly speaking, it allows for \emph{intermediate jumps}), since it will be sufficient for our purpose. We refer to \cite{Whitt02} for a detailed account on the $M_1$- and $J_1$-topologies on $\mathbf{D}$.

Before showing how to derive Theorem~\ref{thm:Biased2} from Theorem~\ref{thm:passagetime}, we need the following proposition, which slightly improves Proposition~2.1 in \cite{BS17}: we prove that the walk $(X_j)_{j \geq 0}$ cannot backtrack more than $C \log n$ before reaching distance $n$. We will use this property several times in the proof of the main theorem, but it  already tells us that the map $n \mapsto T_n$ is the inverse of the map $j \mapsto X_j$, up to an error of at most a constant times $\log n$.

\begin{proposition}\label{kingofthenorth}
	There exists $C>0$ such that the following holds:
define the events $\bar A_x^n :=\bigcap_{j\geq 0}\{ X_{T_{x} +j} \geq x - C \log n \}$  that the random walk does not backtrack more than $C \log n$ after having reached $x$.
Then $\P\otimes P^{\go}$--almost surely there exists $n_0\in\N$ such that, for all $n\geq n_0$, $\bar A_x^n$ holds for all $0\leq x\leq n$.
\end{proposition}

\begin{proof}
We have that $(\bar  A_x^n)^c :=\{X_m< x - C \log n \,,\;\mbox{for some } m \geq T_x\}$, and then
\begin{align}\label{ponchielli}
P^\go \big( (\bar  A_x^n)^c \big)
	\leq P^\go_x(T_{x- C \log n}<+\infty)
	& =\lim_{M\to\infty} \Big({\sum_{j=x}^M \frac{1}{c_j^\gl}}\Big)\Big({	\sum_{j=x-C\log n}^M \frac{1}{c_j^\gl}}\Big)^{-1}
\end{align}
where $P^\go_x$ is the law of the random walk in random environment $\go$, starting from $x$. For the equality we have used \eqref{classicone1}.
Keeping only the conductance at position $x-C\log n$ in the second sum and recalling \eqref{eq:biased_cond}, \eqref{ponchielli} can be bounded by
\begin{align*}
P^\go \big( (\bar  A_x^n)^c \big)
	\leq \e^{ - \lambda C \log n }  c_{x - C\log n} \sum_{j=x}^\infty \e^{-\lambda (j-x)}\frac{1}{c_{j}}
	=:\e^{-\lambda C \log n} K_x^n(\omega)\, .
\end{align*}
Then, we may use that for any $x, n$,  $\bbP(K_x^n >t ) = t^{-\alpha+o(1)}$ (see \cite[Lem.~2.5]{BS17}), to get that  $\bbP$-a.s., for $n$ large enough, $\max_{0 \leq x \leq n} K_x^n \leq n^{3/\ga}$.
As a consequence, for $n$ large enough we have that $P^\go( \bigcup_{0\leq x\leq n} (\bar  A_x^n)^c)\leq n^{1+3/\ga} \e^{-\lambda C \log n}$. 
For $C$ large enough this  is summable, so Borel-Cantelli gives that $\bbP\otimes P^{\go}$-a.s., for $n$ large enough, $\bar  A_x^n$ holds for all $0\leq x \leq n$.
\end{proof}

\smallskip

\begin{proof}[Proof of Theorem~\ref{thm:Biased2}]
From Theorem~\ref{thm:passagetime} and Proposition~\ref{kingofthenorth}, we are able to prove Theorem~\ref{thm:Biased2}.
Let us denote $Y_n := \inf\{ j \colon T_j >n\}$, \textit{i.e.}\ the unique integer such that $T_{Y_n-1} \leq n < T_{Y_n}$. By construction, $X_n \leq Y_n$ for all $n\in \bbN$, and Proposition~\ref{kingofthenorth} tells that a.s., for $n$ large enough, $ X_j \geq Y_j -1 - C \log n$ for all $0\leq j\leq n$.
This shows that we only have to prove the convergence of $(\frac{\psi(n)}{n^{\ga}} Y_{\lfloor un\rfloor} )_{u \in [0,1]}$ as $n\to +\infty$.

Now, define for any $u \geq 0$
\begin{equation}
\label{def:cTcT-1}
\cT_n(u) 
	:= \frac{T_{\lfloor un \rfloor}}{ d_n} \qquad \text{ and } \qquad \cT^{-1}_n(u)
	:= \inf\{ s \geq 0 \, :\, \cT_n(s) >u\}
\end{equation}
the inverse map of $\cT_n$.
Let $f_n$ be an inverse of $d_n$, \textit{i.e.} a sequence of real numbers such that $d_{f_n} =n$ (we may assume that $n\mapsto d_n$ is defined on $\bbR_+$): note that by \eqref{def:dn}, we have that $n^{\ga} \psi(n)^{-1} \sim f_n$ as $n\to+\infty$.
Then, we can write $Y_{\lfloor u n \rfloor}$ in terms of $\cT_{f_n}^{-1}$:
\begin{align*}
\frac{Y_{\lfloor un \rfloor}}{f_n} 
	=\inf \{ j/f_n \, :\, T_{j} > un  \} 
	= \inf\{s \geq 0 \, : \,  \tfrac{1}{n} T_{ \lfloor s f_n \rfloor } >u\}
 = \cT_{f_n}^{-1} (u)
\end{align*}
Now, we may apply Theorem~\ref{thm:passagetime} with $f_n$ in place of $n$, together with Corollary~13.6.4 of \cite{Whitt02} which says that the inverse map from non-decreasing functions of $(\mathbf{D},M_1)$ to non-decreasing functions of $(\mathbf{D},\|\cdot\|_{\infty})$ is continuous at strictly increasing functions. Since $(\cS_{\ga}(u))_{u\geq 0}$ is a.s.\ strictly increasing, see~\cite{Ber99}, we obtain the following convergence in distribution as $n\to+\infty$, under $\bbP\otimes P^{\go}$
\[
( \cT_{f_n}^{-1} (u) )_{u\in[0,1]} \ \Longrightarrow \  \frac{\sin(\pi \ga)}{\pi\ga\bbE[\zeta^{\ga}] }\,   \big( (\mathcal{S}_{\ga})^{-1}(u) \big)_{u\in[0,1]} \, .
\]
Note that we also used the scaling relation $( \mathtt{c} \cS_\ga (u ))_{u\geq 0} \stackrel{(d)}{=} (\cS_{\ga}( \mathtt{c}^{\ga} u ) )_{u\geq 0}$ to express in a simpler way the inverse of $( \mathtt{c} \cS_{\ga}(u) )_{u\geq 0}$ with $\mathtt{c} = ( \frac{\pi \ga}{\sin(\pi \ga)} \bbE[\zeta^{\ga}])^{1/\ga}$.

This concludes the proof since $\cT_{f_n}^{-1} (u) =Y_{\lfloor un \rfloor} /f_n$, with $f_n \sim n^{\ga}/\psi(n)$ as $n\to+\infty$.
\end{proof}


\section{Definition of simple traps and their properties}
\label{sec:def}

In this section (and up to Section~\ref{sec:conv}), we suppose that Assumption \ref{assumptionA} holds, that is $\bbE[c_0^{\ga}]<+\infty$ or $\bbE[(1/c_{-1})^{\ga}] <+\infty$. We recall that under these hypothesis there will be only well-traps or wall-traps, but never well-and-wall traps, cf.\ Proposition~\ref{prop:conditional1}.

\subsection{Definition of simple traps}

By definition \eqref{def:dn} of $d_n$, the maximal depth of a trap between $0$ and $n$ will be of order $d_n$.
We give here the definition of simple $n$-traps, that is traps with near-maximal depth that are either well-shaped or wall-shaped.  We will drop the ``$n$-'' in the name whenever possible.
\begin{definition}\label{simpletraps}
	Let $q_n := (\log n)^{1/4}$. A site $x$ is a \emph{simple $n$-trap}   if $\rho_x > d_n \e^{- q_n}$. It is
		\begin{itemize}
			\item[$\blacklozenge$] of well type if $c_{x-1} > d_n \e^{- q_n^2}$,
			\item[$\blacklozenge$] of wall type if $\frac{1}{c_{x}} > d_n \e^{- q_n^2}$.
		\end{itemize}
	We call
	$\cW_x := \{\rho_x > d_n \e^{-{q_n}}\}$ the event that there is a trap in $x$.  
\end{definition} 

For $n\in\N$ we let  $C_n:=\lceil C\log n\rceil$, with $C>0$  the  constant in Proposition~\ref{kingofthenorth}. We divide $\Z$ in disjoint $n$-\textit{blocks}, where an $n$-block is a box of the form $\{j C_n, jC_n+1,\dots,(j+1)C_n-1\}$ for some $j\in\Z$.
A $n$-\textit{triblock} is a sequence of three consecutive $n$-blocks: we denote $\cB_j := \{ (j-1) C_n , \ldots, (j+2)C_n -1\} $ the $j$-th triblock, and we let $k_n := \lceil n /C_n \rceil$ be the number of $n$-triblocks between $0$ and $n$.

\begin{definition}
	\label{def:trapping}
	For a given $n\in\N$, we define a trapping environment as
	$
	\gG_n:= \bigcup_{x=0}^{C_n -1} \cW_x\,,
	$	that is, an environment with a simple $n$-trap in the first $n$-block.
	A $n$-triblock $\cB_j$ is called a \emph{trapping $n$-triblock} if $\theta^{jC_n} \go \in \gG_n$ (with $\theta^a \go:= (\go_{x+a})_{x\in \bbZ}$).
	We  define $\cJ_n  := \{j\in\{0,\dots,k_n\}:\,\theta^{jC_n} \go \in \Gamma_n\}$ the set of indices of the trapping triblocks.
\end{definition}

\begin{definition}
	\label{def:goodtrapping}
	A trapping $n$-triblock with a $n$-trap at site $x$ is said to be \emph{good}  if
	\begin{itemize}
		\item[(i)] $c_{x} \leq  \e^{4q_n}$, $c_{x-1} \geq e^{- 4 q_n}$;
		\item[(ii)] for all $ y \neq x-1,x$ in the $n$-triblock,  $c_y \leq \e^{4q_n}$ and $c_y\geq \e^{-4q_n}$.
	\end{itemize}
	We denote 
	\[
	\bar \gG_n := \bigcup_{x=0}^{C_n-1}  
	\Big( \cW_x\cap \{ c_{x}, \tfrac{1}{c_{x-1}} \leq  \e^{4q_n} \} \bigcap_{\substack{ -C_n \leq y  < 2C_n \\ y \neq x-1, x } } \{ c_y \in [\e^{-4q_n},\e^{4q_n}] \}  \Big)
	\]
\end{definition}

\subsection{Properties of simple traps}

In this section we collect some important properties of simple traps and of trapping triblocks:
\begin{itemize}
	\item[1.] Traps are isolated, in fact distant by at least $n\e^{- 5 q_n}$ (Lemma \ref{disgiunti});
	\item[2.] Traps are not too deep: they cannot be deeper than $d_n \e^{q_n}$ (Lemma \ref{lem:maxrho});
	\item[3.]  All trapping triblocks are good (Lemma \ref{sogood});
	\item[4.] Traps are the only annoying parts of the environment (Lemma \ref{annoying}).
\end{itemize}

We now state the lemmas corresponding to these four properties, and postpone their proofs to Appendix \ref{dimostrazionitrappole}---they will be treated together with the analogous statements in the case \ref{assumptionB}.

\begin{lemma}\label{disgiunti}
	Let $\cD_n$ be the event
	\[
	\cD_n 
	:= \bigcup_{ \substack{-n \leq x<y\leq n \\ |x-y|\leq n \e^{- 5 q_n } } }  \cW_x \cap \cW_y \, .
	\]
	Then $\bbP$-a.s.\ $\cD_n$ occurs for only finitely many $n$.
\end{lemma}

\noindent The proof of Lemma \ref{disgiunti} can be found in Section \ref{proofdisgiunti} of the Appendix.
	As a consequence of Lemma \ref{disgiunti}, we have a.s.\ that $|\cJ_n| \leq \e^{5 q_n}$ for $n$ large enough.

\begin{lemma}\label{lem:maxrho}
	$\bbP$-a.s., there is some $n_0$ such that for all $n\geq n_0$
	\[ 
	M_n:=\sup_{ -n \leq x \leq n}\rho_x
	\leq  d_n \e^{ q_n} \, .
	\]
\end{lemma}

\noindent The proof of Lemma \ref{lem:maxrho} can be found in Section \ref{prooflem:maxrho} of the Appendix.

\begin{lemma}\label{sogood}
	Let $\cG_n$ be the event
	\[
	\cG_n := \bigcup_{ - n\leq x\leq n} \mathcal W_x \cap  \Big( \bigcup_{ \substack{ y = x- 2C_n \\ y\neq x-1,x } }^{x+2 C_n}  \{ c_{x}, \tfrac{1}{c_{x-1}} > \e^{ 4q_n} \} \cup\{ c_y \notin [\e^{-4 q_n},\e^{4 q_n}] \} \Big) \, .
	\]
	Then $\bbP$-a.s.\ $\cG_n$ occurs only for finitely many $n$.
\end{lemma}

\noindent The proof of Lemma \ref{sogood} can be found in Section \ref{proofsogood} of the Appendix.

\smallskip
The last property mentioned above may be more mysterious.
Let us define  a generalization of $\rho_x$:  for all $x\in \bbZ$ and $k\in\N$, set
\begin{equation}
\label{def:rhok}
\rho_x^{(k)} := \e^{-\lambda(k+1)} \frac{c_{x-1}}{c_{x+k}} =   \rho_{x} \cdots \rho_{x+k}\, . 
\end{equation}
As it will be clear in the next sections, high values of $\rho_x^{(k)}$ are also responsible for the trapping of the walk. What we mean by property 4 is that, under Assumption \ref{assumptionA}, high values of $\rho_x^{(k)}$ are only observed in the presence of a well or of a wall (that is, they cannot come from the combination of an almost-well and an almost-wall). This is to ensure that triblocks that are not trapping triblocks (according to Definition \ref{def:trapping}) will not slowdown the walk for a long time.

\begin{lemma}
\label{annoying}
 Let $\gep_n =q_n^{-\gd}$ with some $\gd>0$, and let
\begin{align}\label{giupiter}
\mathcal{H}_n 
:= \bigcup_{-n\leq x \leq n\,}  \bigcup_{k\ge 1}   \{\rho_{x}^{(k)} > \gep_n \e^{-\lambda  k/2} d_n \}\cap \cW_x^{c} \cap \cW_{x+k}^c\,.
\end{align}
Then, if $\gd$ is  sufficiently small, we have that $\bbP(\cH_n)  \to 0$ as $n\to+\infty$.
\end{lemma}

\noindent The proof of Lemma \ref{annoying} can be found in Section \ref{proofannoying} of the Appendix.


\section{Reduction to large traps}
\label{sec:reduc1}

For a triblock $\cB_j$, we define its crossing time:
\begin{equation}
\label{crossB}
T(\cB_j) :=  T_{(j+2)C_n} - T_{j C_n}   \, ,
\end{equation}
which is the time it takes for the random walk, started at the beginning of the middle $n$-block of $\cB_j$, to exit $\cB_j$ from the right. In view of Proposition~\ref{kingofthenorth}, provided that $n$ is sufficiently large, the walk $X_i$ remains in $\cB_j$  for all $i$ between times $T_{j C_n}$ and $T_{(j+2)C_n}$.
In this section we prove that the time $T_n$ to reach distance $n$ is close to the crossing time of all trapping $n$-triblocks, so that the other blocks can be neglected.

\begin{proposition}\label{blocchetti}
Let $\gep_n:= q_n^{-\gd}$ with $\gd$ small enough, as defined in Lemma~\ref{annoying}: we have that
\begin{align*}
 P^\go\Big( \frac{1}{d_n} \Big|T_n-\sum_{j\in \cJ_n} T(\cB_j)\Big|  >   \gep_n^{(1-\ga)/2} \Big) 
 	\leq  \gep_n^{(1-\ga)/2} \, .
\end{align*}
with $\bbP$-probability going to $1$.
\end{proposition}

Before we start the proof, recall the definition \eqref{def:rhok} of $\rho_x^{(k)}:= \e^{-\lambda (k+1)} \frac{c_{x-1}}{c_{x+k}}$.  
Recall also Proposition~\ref{prop:rho0}: using Potter's bound (cf.~\cite[Thm.~1.5.6]{BGT89}), we easily have that there is a constant $c>0$ such that for any $k\geq 0$ 
\begin{equation}
\label{asymprhok}
\bbP ( \rho_x^{(k)} > t \big) \leq c\, \e^{-\lambda \ga k/2} \psi(t) t^{-\ga}\quad \text{ for all } t>1\,  .
\end{equation}

\begin{proof}[Proof of Proposition~\ref{blocchetti}]
We write that a.s., for $n$ large enough
\begin{equation}
\label{decompTn}
T_n-\sum_{j\in \cJ_n} T(\cB_j)  =  \sum_{j=0}^{k_n} (T_{ (j+1) C_n} - T_{ j C_n} ) \ind_{ \{ j,j-1\notin \cJ_n\} } \ind_{A_{j}^n}\, .
\end{equation}
Here, we let $A_{j}^n:=\{ X_{T_{jC_n}+i} > (j-1) C_n \text{ for all }i\geq 0\}$ be the event  that the random walk does not backtrack from the $j$-th $n$-block to position $(j-1) C_n$: Proposition~\ref{kingofthenorth} tells that $\bbP\otimes P^\omega$-a.s., for $n$ large enough, $\ind_{A_{j}^n} =1$ for all $0\leq j \leq k_n$.
We also used that there cannot be two consecutive good blocks, see Lemma~\ref{disgiunti}.
Let us now remark that thanks to  formula \eqref{realmagic}
\begin{align*}
E^{\go} \big[ (T_{(j+1) C_n} - T_{ j C_n} ) \ind_{A_{j}^n} \big] & = E_{j C_n}^{\go}\big[ T_{(j+1) C_n} \ind_{A_{j}^n} \big] \\
 & \leq C_n +  \sum_{z =(j-1) C_n }^{jC_n-1}   \sum_{k=j C_n -z}^{(j+1)C_n-z}  \rho_{z}^{(k)} + \sum_{z=jC_n}^{(j+1)C_n -1}  \sum_{k=0}^{(j+1)C_n-z} \rho_z^{(k)}  \, .
\end{align*}
Here we have first set all conductances to the left of $(j-1)C_n$ equal to $0$ thanks to the indicator of $A_j^n$  (so the $\rho_z^{(k)}$ are equal to $0$ for $z< (j-1)C_n$), and then we have dropped the indicator.
Thanks to Lemma \ref{annoying}, we have that with high $\bbP$-probability all the $\rho_z^{(k)}$ involved in the sum (\textit{i.e.} for $j, j-1\notin\cJ_n$) are smaller than $\gep_n \e^{-\lambda k/2} d_n$.
Going back to \eqref{decompTn}, we therefore get that  on the event $\cH_n^c$ (recall $\cH_n$ is defined in \eqref{giupiter}),
\begin{align}
E^{\go} &\Big[ T_n-\sum_{j\in \cJ_n} T(\cB_j)   \Big] 
	\leq n + 2\sum_{z=- C_n}^n  \sum_{k=0}^{2C_n} \rho_z^{(k)} 
		\ind_{\{ \rho_z^{(k)} \leq \gep_n \e^{-\lambda k/2} d_n \}} \notag \\
	&\leq  n 
		+  2 \sum_{z=- C_n}^n  \sum_{k=0}^{\ell_n } \rho_z^{(k)} \ind_{\{ \rho_z^{(k)} \leq \gep_n \e^{-\lambda k/2} d_n \}}  
		+ 2 \sum_{z=- C_n}^n  \sum_{k=\ell_n }^{C_n} \rho_z^{(k)} \ind_{\{ \rho_z^{(k)} \leq   \e^{-\lambda \ell_n  /2} d_n \}}   \, ,
\label{eq:expectT}
\end{align}
where we set $\ell_n := (\log \log n)^2$ and take the integer part when dealing with non-integer indices.

We first deal with the first sum in \eqref{eq:expectT}.
Let us set $R_z^{n}:=  \sum_{k=0}^{\ell_n} \rho_z^{(k)}$. We have that 
\[
\sum_{k=0}^{\ell_n} \rho_z^{(k)}  \ind_{\{ \rho_z^{(k)} \leq \gep_n \e^{-\lambda k/2} d_n  \}}
	\leq R_z^{n} \ind_{\{ R_z^{n} \leq \ell_n \gep_n d_n \}} 
	=: \bar  R_z^{n}.
\]
For any fixed $1-\ga<\gd'< 1$, we can bound
\begin{align}
\bbP\Big( \sum_{z=-C_n}^{n} \bar  R_z^{n}  > (\gep_n)^{\gd'} d_n \Big) 
	\leq  3 \ell_n \bbP \Big( \sum_{j=- C_n/3\ell_n }^{n/3\ell_n} \bar R_{3 j \ell_n}^n > \frac{(\gep_n)^{\gd'} }{3 \ell_n}  d_n \Big)
\label{splitindependent}
\end{align}
where we used a union bound for the last inequality, splitting the sum $\sum_{z=-C_n}^{n} \bar R_z^n $ as {$\sum_{x=0}^{3\ell_n-1} \sum_{j=-C_n/3\ell_n}^{n/3\ell_n-1} \bar R_{3j \ell_n+x}$}. 
Now, we notice that the $\bar R_{3j \ell_n }^{n}$ for $j\in \{-C_n/3\ell_n, \ldots, n/3\ell_n\}$ are independent. Also, thanks to \eqref{asymprhok} and a union bound, we have that $\bbP(R_{0}^n>t) \leq cst.\, \psi(t) t^{-\ga}$. Them, using a Fuk-Nagaev inequality (see \cite[Thm.~1.1]{Nag79}), we have that there is a constant $c$ such that for any $m\geq 1$ and any $y \leq x$
\[
\bbP\Big( \sum_{j=1}^{m} R_{3j\ell_n}^n > x \, ,\,  \max_{1\leq j\le m}  R_{3j \ell_n}^n \leq y \Big) \leq  \Big( c m \frac{y}{x} \, \psi(y) y^{-\ga} \Big)^{x/y}\,.
\]
We therefore get that there are constants $c,c',c''>0$ such that \textit{}
\begin{align*}
 \bbP \Big( \sum_{j=- C_n/3\ell_n }^{n/3\ell_n} \bar R_{3 j \ell_n}^n > &\frac{(\gep_n)^{\gd'} }{3 \ell_n}  d_n \Big)
	\leq  \Big( c \frac{n}{\ell_n} \frac{\ell_n \gep_n}{ (\gep_n)^{\gd'} /\ell_n }  \psi\big( \ell_n \gep_n d_n \big)  \big(  \ell_n \gep_n d_n \big)^{-\ga}\Big)^{  (\gep_n)^{\gd'-1} /(3 \ell_n^2) } \\
	&\leq  \Big( c' n \ell_n  (\gep_n)^{ (1-\ga - \gd')/2} \psi(d_n) d_n^{-\ga} \Big)^{ (\gep_n)^{-(1-\gd')/2} } 
	\leq  c''  \e^{- (\gep_n)^{-(1-\gd')/2} } \, .
\end{align*}
For the second inequality, we used Potter's bound to get that $\psi (\ell_n \gep_n d_n  ) $ is bounded by a constant times $ \ell_n^{\ga} (\gep_n)^{ (\gd' +\ga-1) /2 } \psi(d_n)$. We also used that $(3 \ell_n)^{-2}  \leq (\gep_n)^{-(1-\gd')/2}$ for $n$ large enough (recall $\gep_n= q_n^{-\gd}$ with $q_n = (\log n)^{1/4}$). For the last inequality, we used the definition~\eqref{def:dn} of $d_n$ to get that for $n$ large enough the term inside the parenthesis is smaller than  a constant times $\ell_n (\gep_n)^{(1-\ga-\gd)/2}$, which can be made arbitrarily small by taking $n$ large.

Going back to \eqref{splitindependent}, and using that $\ell_n =o(\e^{ \frac12(\gep_n)^{-(1-\gd')/2}})$ we get that
\begin{equation}
\bbP\Big( \sum_{z=-C_n}^{n} \bar  R_z^{n}  > (\gep_n)^{\gd'} d_n \Big)
	 \leq  c'' \e^{- \frac12(\gep_n)^{-(1-\gd')/2}}\, .
\end{equation}

We can treat the second term in \eqref{eq:expectT} similarly.
Setting $V_z^n := \sum_{k=\ell_n}^{C_n} \rho_z^{(k)}$, we get that 
\[
\sum_{k=\ell_n }^{C_n} \rho_z^{(k)} \ind_{\{ \rho_z^{(k)} 
	\leq  \e^{-\lambda \ell_n/2 } d_n \}}  \leq V_z^n \ind_{\{V_z^n \leq C_n \e^{- \lambda \ell_n /2}d_n\}} =: \bar V_z^n \, .
\]
Similarly as above, with $C_n$ playing the role of $\ell_n$ and $\e^{- \lambda \ell_n /2}$ playing the role of $\gep_n$ (note that  $C_n = o(\e^{\lambda \ell_n /2})$), we obtain that for $n$ large enough
\begin{equation}
\bbP\Big( \sum_{z=-C_n}^{n} \bar  V_z^{n}  > \e^{- \gd' \lambda \ell_n /2} d_n \Big) 
	\leq  c'' \e^{- \frac12\e^{(1-\gd')\lambda \ell_n /4} }
	\leq   c'' \e^{- \frac12(\gep_n)^{-(1-\gd')/2}}\, .
\end{equation}

Going back to \eqref{eq:expectT}, by a union bound we get that for any $1-\ga< \gd' <1$,
\begin{align}
\label{eq:something}
\bbP\Big( E^{\go} \big[ T_n-\sum_{j\in \cJ_n} T(\cB_j)   \big] >  3 (\gep_n)^{\gd'} d_n \, ; \cH_n^c\Big) 
	\leq  2 c'' \e^{- \frac12(\gep_n)^{-(1-\gd')/2}}  \to 0\, .
\end{align}
Using Markov's inequality, we get that on $\cH_n^{c}$
\[
P^{\go} \Big( \frac{1}{d_n} \big|T_n-\sum_{j\in \cJ_n} T(\cB_j)\big|  > \gep_n^{(1-\ga)/2} \Big) \leq  \gep_n^{-(1-\ga)/2} \frac{1}{d_n}E^{\go} \big[ T_n-\sum_{j\in \cJ_n} T(\cB_j)   \big] \, ,
\]
which thanks to \eqref{eq:something} is smaller than $\gep_n^{(1-\ga)/2}$ for $n$ large enough (recall $\gd'>1-\ga$), with $\bbP$-probability going to $1$.
Since $\bbP(\cH_n)\to 0$ by Lemma~\ref{annoying}, we get the conclusion of Proposition~\ref{blocchetti}.
\end{proof}


\section{Crossing times of trapping triblocks}
\label{sec:crossingontraps}

Notice that Lemma~\ref{disgiunti} gives that to each trapping block $\cB$ can be associated a unique site $x_{\cB}$ such that  $x_{\cB}$ is a  $n$-trap.
We denote $\rho_{\cB}:= \rho_{x_{\cB}}$ the depth of the trap associated with~$x_\cB$.

In this section we show that, for a trapping block, the random variable $T(\cB)/\rho_{\cB}$ can be well approximated by another random variable $\tau_\cB$, given by
\begin{align}\label{semplificare}
\tau_\cB=\tau_\cB(\go):=\theta_{\cB} \sum_{j=1}^{G(p_{\cB})}   E_j \, .
\end{align}
Let us explain each term appearing in \eqref{semplificare}:
\begin{itemize}
\item[$p_{\cB}$]$= P_{x_{\cB}+1}^\go( T_{x_{\cB}} > T_{x_{\cB}+ C_n})$ is the probability, starting from $x_\cB+1$, of never returning to $x_\cB$ (this is also  the escape probability from $x_\cB+1$, recall the non-backtracking property of Proposition~\ref{kingofthenorth}). Notice that this only depends on the environment to the right of $x_\cB+1$ that is in $\cB$.

\item[$G(p_\cB)$] is a geometric random variable of parameter $p_\cB$.
This will be coupled with the random variable that counts the number of times we cross the edge $(x_\cB+1,x_\cB)$ from the right to the left, that is, the number of times we fall anew in the trap. 

\item[$\theta_\cB$]$=\theta_\cB(\go)$ is equal to $E^{\bar \go}[\Theta_x]$,  where $\Theta_x$ is $1$ plus the time that it takes for $X_j$ to go from $x_\cB-1$ to $x_\cB$, and $\bar\go$ is $\go$ with all conductances to the left of $x_{\cB}-C_n$ replaced by $0$  (recall again the non-backtracking property of Proposition~\ref{kingofthenorth}). Notice that $\theta_\cB$ depends only on the environment to the left of $x_\cB$ that is in $\cB$.

\item[$(E_j)_{j\geq 1}$] is a sequence of i.i.d.~exponential random variables of parameter $1$, independent of $G(p_\cB)$ and $\theta_\cB$. 
\end{itemize}

We have written $\tau_{\cB}$ as \eqref{semplificare} in order to make explicit where the different terms come from, but one realizes (for instance computing the characteristic function) that a geometric sum of independent exponential r.v.s is itself an exponential random variable. We can therefore rewrite $\tau_{\cB}$ as
\begin{equation}\label{app:tau}
\tau_{\cB} = \frac{\theta_{\cB}}{p_{\cB}}\ \mathbf{e}_{\cB} \, ,
\end{equation}
where $\mathbf{e}_\cB \sim \mathrm{Exp}(1)$ is independent of $p_\cB$ and $\theta_\cB$.

\begin{proposition}\label{trappolone}
Let $\cB=\cB_j$ be a good trapping $n$-triblock (recall  the Definition~\ref{def:goodtrapping} of~$\bar \Gamma_n$). There exists a coupling between $(X_j)_{j\in\N}$ and $\tau_{\cB}$ such that for any $\gd\in(0,1)$
\begin{align}\label{paolo}
\sup_{\go\in \bar \gG_n} 
	P^\go\Big(\Big|\frac{T(\cB)}{\rho_\cB}-\tau_{\cB}\Big|>\gd  , A_n\Big)
		\leq  2 \e^{- \e^{q_n}} +  c \,\gd^{-2} \e^{50 q_n} \rho_{\cB}^{-1} 
\end{align}
where  $A_n$ is the event that $(X_j)_{j\in\N}$ never backtracks more than $C_n$.
\end{proposition}

\noindent
We proceed in three steps:

\smallskip

Step 1. We decompose $T(\cB)$ into several pieces, see \eqref{decomposino}: a) the time it takes to arrive at position $x_\cB$, b) a geometric number $G=G(p_{\cB})$ of i.i.d.\ times to go from $x_\cB$ to $x_\cB+1$ ($G-1$ corresponds to the geometric number of times we fall anew in the trap), c) the time it takes from $x_\cB+1$ to exit the triblock to the right.

\smallskip

Step 2. We control easily the terms a) and c), which bring a small contribution to $T({\cB})$.

\smallskip

Step 3. We control the main term b), by saying that the time it takes to exit the trap is another geometric r.v.\ with parameter roughly $1/\rho_{\cB}$ (the number of trials one needs to cross the edge $(x_\cB, x_\cB+1)$ from left to right), multiplied by the time it takes between two trials (which is roughly $1+ E^{\go}_{x_\cB-1}[T_{x_\cB}] = \theta_\cB$ by the law of large numbers).

\subsection{Step 1}\label{sectioncrossinglarge}

For simplicity we suppose that $\cB=\{-C_n,\dots , 2C_n-1\}$, that the random walk starts in $0$ (we write $P^{\go}:=P_{0}^{\go}$) and we simply call $x=x_\cB$.
We notice that under the event $A_n$ the walk never visits points to the left of $-C_n$, so that we can replace each $\go$ in \eqref{paolo} by the environment $\bar\omega$ where all conductances $c_j$ with $j<-C_n$ are set equal to $0$ and the other conductances stay as before. For each $\go \in \bar \Gamma_n$ we will therefore try to bound 
\begin{align*}
P^{\go}\Big(\Big|\frac{T(\cB)}{\rho_\cB}-\tau_{\cB}\Big|>\gd  , A_n\Big)
	\leq P^{\bar \go}\Big(\Big|\frac{T(\cB)}{\rho_\cB}-\tau_{\cB}\Big|>\gd\Big)\,.
\end{align*}

Let us introduce some further notation. 
We call $G:=1+\#\{\ell :\,X_{\ell-1}=x+1,\,X_\ell=x\}$. We notice that $G$ is a geometric random variable of parameter $p_\cB$ appearing in \eqref{semplificare}.
Let $T^{(1)}:=T_{x+1}-T_{x}$ and $t_1=T_{x+1}$.
When $G\geq 2$, we iteratively define, for $2\leq i\leq G$,
\begin{align*}
t_i
	&:=\inf\{\ell>t_{i-1}:\,X_{\ell-1}=x,\,X_\ell=x+1\}\\
T^{(i)}
	&:=t_i-t_{i-1}\, .
\end{align*}
We interpret $T^{(i)}$ as the time it takes for the walk to escape from the trap for the $i$-th time.
We can now decompose $T(\cB)$ in the following way:
\begin{align}\label{decomposino}
T(\cB)=T_{x} +\sum_{i=1}^{G}T^{(i)} +\bar T_{x+1,2C_n}\,,
\end{align} 
where $\bar T_{x+1,2C_n}$ is the time it takes for $(X_j)_{j\in\N}$ to go from $x+1$ to $2C_n$ conditioned on never returning to $x$. In light of \eqref{decomposino}, for any $\go\in\bar\gG_n$ we can bound
\begin{align}\label{ddd}
P^{\bar\go}\Big(\Big|\frac{T(\cB)}{\rho_\cB}-\tau_{\cB}\Big|>\gd \Big)
	\leq D_1+D_2+D_3\,,
\end{align}
where
\begin{align*}
D_1 = P^{\bar\go}\Big(\frac{T_{x}}{\rho_\cB}>\frac{\gd}{3} \Big),\;
D_2 = P^{\bar\go}\Big(\Big|\frac{\sum_{i=1}^{G}T^{(i)}}{\rho_\cB}-\tau_{\cB}\Big|>\frac{\gd}{3}\Big) ,\;
D_3 = P^{\bar\go}\Big(\frac{\bar T_{x+1,2C_n}}{\rho_\cB}>\frac{\gd}{3} \Big)\,.
\end{align*}

\subsection{Step 2}

The first and last term are fairly easy and are treated in the same way.
For $D_1$, we use Markov's inequality to get that
\[
D_1 \leq  3 \gd^{-1} \rho_{\cB}^{-1} E^{\bar\go}[T_{x} ].
\]
Formula \eqref{natale} and \eqref{Sij} gives an explicit expression for $E^{\bar\go}[T_{x}]$. In particular, by Definition~\ref{def:goodtrapping} of a good block (in particular since $c_y \in [\e^{-4q_n}, \e^{4q_n}]$ for all $y<x-1$ and $c_{x-1}>\e^{4q_n}$), we get  that  $E^{\bar\go}[T_{x}]\leq c' C_n \e^{8q_n}$. Hence,  for $n$ large enough,
\begin{align*}
D_1
	\leq 
	c\, \gd^{-1} \e^{10q_n}\rho_{\cB}^{-1} \,.
\end{align*}

For the term $D_3$, we use the same idea:
\begin{align}
D_3 
	&\leq 3 \gd^{-1} \rho_\cB^{-1} E^{\bar\go}[ \bar T_{x+1,2C_n} ] =  3 \gd^{-1} \rho_\cB^{-1} \frac{1}{p_\cB} E^{ \bar \go}_{x+1}[ T _{2C_n}  \ind_{\{T_{2C_n}<T_{x} \}} ]	
\end{align}
Then, we can use that 
Here we have used that $E^{ \bar \go}_{x+1}[ T _{2C_n}  \ind_{\{T_{2C_n}<T_{x} \}} ]\leq  E^{ \tilde \go}_{x+1}[ T _{2C_n}   ]$, where thanks to the indicator function we have replaced $\bar \go$ by $\tilde \go$ (in which the conductances to the left of $x+1$ have been replaced by $0$), and then dropped the indicator function.
Then, formula~\eqref{natale} together with Definition~\ref{def:goodtrapping} of a good block gives that $E^{ \tilde \go}_{x+1}[ T _{2C_n} ] \leq c' C_n \e^{8 q_n}$.
We finally get that
\begin{equation}\label{termd3}
D_3 \leq c \gd^{-1} \rho_{\cB}^{-1} C_n \e^{8 q_n} /p_\cB \leq  c \gd^{-1} \rho_{\cB}^{-1} \e^{20 q_n} \, .
\end{equation}
For the last inequality, we used the formula \eqref{formulap} for $p_\cB$, together with the definition of a good block to get that $p_\cB \geq c \e^{- 8 q_n}$ (and took $n$ large enrough so that $C_n \leq \e^{4q_n}$).

\subsection{Step 3}
It remains to deal with the term $D_2$ in \eqref{ddd}, which is the most technical part of this section.
For  $i=1,\dots, G$,  we decompose $T^{(i)}$ as follows
\begin{align}\label{coniglio}
T^{(i)}=\Upsilon_i+ \Big(\sum_{j=1}^{\cG_i-1}\Theta_{x}(j,i) \Big)+1\,,
\end{align}
where $\gU_1=0$ and $\Upsilon_i:=\inf\{\ell>t_{i-1}:\,X_\ell=x+1\}-t_{i-1}$ for $i\ge 2$; $\cG_i$ is a geometric random variable with parameter $1/(1+\rho_{x})$ and $\{\Theta_{x}(j,i)\}_{j,i\in\N}$ are i.i.d.~copies of $\Theta_{x}$. In fact, we can describe $T^{(i)}$, with $i=2,\dots, G$, as follows (the discussion is similar for $i=1$): at time $t_{i-1}$ the random walk is in $x+1$ and in a time $\Upsilon_i$ it reaches $x$ for the first time after $t_{i-1}$. 
Then it jumps to the right with probability $1/(1+\rho_{x})$ (in which case  we either go to iteration $i+1$ or never return to $x$) or to the left with probability $\rho_{x}/(1+\rho_{x})$. In this second case, $\Theta_{x}(1,i)$ represents this one step plus the time it takes for $(X_\ell)$ to go from $x-1$ back to $x$. Then $(X_\ell)$ ``tries'' again to jump to $x+1$ with probability $1/(1+\rho_{x})$ (in which case  we either go to iteration $i+1$ or never return to $x$) or goes back to $x-1$ with probability $\rho_{x}/(1+\rho_{x})$, etc... This continues until the random walk manages to reach $x+1$. Since each attempt of jumping from $x$ to $x+1$ is independent from the others, the number of attempts is geometric. 
The $\Theta_{x}(j,i)$ are i.i.d.~by the Markov property. 

Thanks to the representation \eqref{coniglio}, and using i.i.d.\ exponential r.v.\  $(E_i)_{i\geq 1}$ of parameter~$1$ coupled with $\cG_i$ (by $\cG_i = \lceil E_1 \log \tfrac{\rho_{\cB}}{1+\rho_\cB} \rceil$), we can bound 

\begin{align}\label{denver}
&  D_2 
 	\leq P^{\bar\go}\Big(\sum_{i=1}^G\Big|\frac{ \Upsilon_i+\sum_{j=1}^{\cG_i-1}\Theta_{x}(j,i)+1}{\rho_\cB}- \theta_\cB E_i\Big|>\frac{\gd}{3} \Big)\nonumber\\
	&\leq P^{\bar\go}\Big(\sum_{i=1}^G \Big|\frac{ \Upsilon_i+\sum_{j=1}^{\cG_i-1}\Theta_{x}(j,i)+1}{\rho_\cB} - \theta_\cB E_i\Big|>\frac{\gd}{3} , G \leq \e^{10q_n}\Big)+P^{\bar\go}(G>\e^{10 q_n})\,.
\end{align}

For the second summand of \eqref{denver}, recall that $1+G$ is a geometric variable with parameter~$p_\cB$: we have that
$P^{\bar\go}(G>\e^{10q_n })  \leq \exp ( - p_\cB \e^{10 q_n} )$.
Now, we have already seen that $p_\cB \geq c \e^{8 q_n}$, so that for $n$ large enough we get
\begin{align}
\label{eq:boundG}
P^{\bar\go}(G>\e^{10 q_n}) \leq \exp \big( -  \e^{q_n} \big) \, .
\end{align}

We control now the first term in \eqref{denver} via the triangular inequality and  a union bound (recall that $\Upsilon_1=0$ by definition): we get that it is bounded by
\begin{align}\label{ragnarock}
\e^{10q_n} P^{\bar\go}\Big( \frac{ \Upsilon_2+1}{\rho_\cB}>\frac{\gd}{6}  \e^{-10q_n} \Big)
		+\e^{10q_n} P^{\bar\go}\Big(\Big|\frac{ 1}{\rho_\cB}\sum_{j=1}^{\cG_1-1}\Theta_{x}(j,1)- \theta_\cB \, E_1\Big|>\frac{\gd}{6} \e^{-10q_n}\Big)\,.
\end{align}
We deal with the two summands of \eqref{ragnarock} in the following claims.
\begin{claim}\label{sentenzioso1}
There is a constant $c>0$ such that,  for all  $\go \in \bar\Gamma_n$,  for all $v>0$, we have
\begin{align}\label{libellula}
P^{\bar\go}\Big(\frac{ \Upsilon_2+1}{\rho_\cB}> v \Big)
	<  c \, v^{-1}\rho_\cB^{-1} \, \e^{10 q_n} \, .
\end{align}
\end{claim}
\begin{claim}\label{sentenzioso2}
	There is  a constant $c>0$ such that, for all  $\go \in \bar\Gamma_n$,  for all $v>0$, we have
\begin{align}\label{ritorno}
P^{\bar\go}\Big(\Big|\frac{ 1}{\rho_\cB}\sum_{j=1}^{\cG_1-1}\Theta_{x}(j,1) -  \theta_\cB E_1 \Big|> v \Big)
	\leq
\e^{-\e^{q_n}}+c\,\e^{20 q_n}v^{-2}\rho_\cB^{-1}\,.
\end{align}
\end{claim}
We prove these two claims just below. 
Thanks to Claims \ref{sentenzioso1} and \ref{sentenzioso2}, we obtain, together with estimates \eqref{denver} and \eqref{eq:boundG}-\eqref{ragnarock}, that (using also that $\gd^{-1}\leq \gd^{-2}$)
\begin{align*}
D_2 
	\leq  2\e^{-  \e^{q_n}}+c'\,\delta^{-2}\e^{50 q_n}
	\rho_\cB^{-1} \, ,
\end{align*}
for some constant $c'>0$. This concludes the proof of Proposition \ref{trappolone}.

\begin{proof}[Proof of Claim \ref{sentenzioso1}]
We will use Markov inequality, and therefore proceed  to the calculation of the expectation of $\gU_2$. By definition,
$E^{\bar\go}[\gU_2]
	=E^{\bar\go}_{x+1}[T_{x}\,|\,T_{x}<T_{2C_n}]
	=E^{\bar\go}_{x+1}[T_{x} \ind_{\{T_{x}<T_{2C_n}\}}]/P^{\bar\go}_{x+1}(T_{x}<T_{2C_n})$. 
We use formulas \eqref{magicformula2} and  \eqref{classicone}, which yield 
\begin{align*}
E^{\bar\go}[\gU_2]
	&=\frac{1}{c_{x}^\gl}\sum_{y=x+1}^{2C_n-1}
		\pi(y)P_y(T_{x+1}<T_{2C_n})P_y(T_{x}<T_{2C_n})
	\leq \sum_{y=x+1}^{2C_n-1}\frac{c^{{\gl}}_y}{c_{x}^\gl}
	\,P_y(T_{x}<T_{2C_n})
\end{align*}
For all $y=x+1,\dots,2C_n-1$, we can use again \eqref{classicone} and bound 
\begin{align*}
\frac{c_y^\gl}{c_{x}^\gl} P_y(T_{x}<T_{2C_n})
	&= \frac{c_y}{c_{x}}\e^{\gl (y-x)}\frac{\sum_{j=y}^{2C_n-1}c_j^{-1}\e^{-\gl j}}{\sum_{j=x}^{2C_n-1}c_j^{-1}\e^{-\gl j}}
	\leq \frac{1}{c_{x}}\e^{8q_n} \frac{\sum_{j=0}^{2C_n-1-y}\e^{-\gl j}}{c_{x}^{-1}}
	= c' \,\e^{8q_n}\,,
\end{align*}
for some $c'>0$, where for the first inequality we have used the fact that in good blocks $\e^{-4q_n}<c_j<\e^{4q_n}$ for all $x<j<2C_n$ and in the denominator we simply kept the term $j=x$. We therefore obtain
\begin{align*}
E^\go[\gU_2+1]\leq 2c'C_n\e^{8q_n} \leq \e^{10 q_n}
\end{align*}
for $n$ large enough.
This gives the claim by Markov inequality.
\end{proof}

\begin{proof}[Proof of Claim \ref{sentenzioso2}]
We abbreviate $\Theta_{x}^{(j)}:=\Theta_{x}(j,1)$, and recall that $\theta_\cB = E^{\bar \go}[\Theta_x^{(1)}]$. 
By the triangular inequality, we bound the left hand side of \eqref{ritorno} by
\begin{align}\label{pony}
	P^{\bar\go}\Big(\Big|\sum_{j=1}^{\cG_1-1}\Big(\Theta_{x}^{(j)}\rho_\cB^{-1}-  \theta_\cB\rho_{\cB}^{-1} \Big)\Big|>\frac v 2\Big) 
		+ P^{\bar\go}\Big( \theta_\cB \Big|E_1-(\cG_1-1)\rho_{\cB}^{-1}\Big|>\frac v 2\Big)\,.
\end{align}

For the second term in \eqref{pony}, recall that we coupled $E_1$ with $\cG_1$ by $\cG_1 = \lceil E_1 \log \frac{\rho_{\cB}}{1+\rho_{\cB}} \rceil$: this gives that $|E_1 - (\cG_1-1)/\rho_{\cB} | \leq (2+E_1)/\rho_{\cB}$, so that the second term in \eqref{pony} is bounded by 
$
P^{\bar\go}( E_1 > c v \rho_{\cB} \theta^{-1}_\cB)
$
for some constant $c>0$.
Thanks to Lemma~\ref{meanvarlemma1} below, this is smaller than $\exp(-c''v\rho_\cB \e^{-8q_n}) \leq c''' v^{-1} \rho_{\cB}^{-1} \e^{8 q_n}$.

We are left with the first term in \eqref{pony}. We study the mean and variance of the i.i.d.~random variables $\Theta_{x}^{(j)}$, which by definition are distributed as the random variable $\Theta_{x}$, in the following lemma.
\begin{lemma}\label{meanvarlemma1}
We have, for all $\go \in \bar \Gamma_n$
\begin{align}
\theta_\cB = E^{\bar\go}[\Theta_{x}]
	&\leq c' \e^{8 q_n} \, ,\label{expectationk0}\\
\mathrm{Var}\,^{\bar\go}(\Theta_{x})
	&\leq c' C_n^2 \e^{16q_n}\label{variancek0}\,.
\end{align}
\end{lemma}

Let us first conclude the proof of Claim~\ref{sentenzioso2} before we turn to the proof of Lemma~\ref{meanvarlemma1}.
We have that the first term in \eqref{pony} is
\begin{align}\label{pony1}
P^{\bar\go}\Big(\Big|\sum_{j=1}^{\cG_1-1}\Big(\Theta_{x}^{(j)}-  \theta_\cB \Big)\Big|>\frac v 2\rho_{\cB}\Big)
	&\leq	P^{\bar\go}\big(\cG_1>\e^{2q_n}\rho_{\cB}\big)+\e^{2q_n}\rho_\cB\frac{\mathrm{Var}\,^{\bar\go}(\Theta_{x})}{v^2\rho_\cB^2/4}\nonumber\\
	&\leq 	\e^{-\e^{q_n}}	+c''\e^{20q_n}v^{-2}\rho_\cB^{-1}\,.
\end{align}
In the first line we have used Chebychev inequality and the independence of the $\Theta_{x}^{(j)}$, while in the second line we used the fact that $\cG_1$ is a geometric of parameter $1/(1+\rho_\cB)$ for the first term and the bound \eqref{variancek0} for the second term.
\end{proof}

\begin{proof}[Proof of Lemma \ref{meanvarlemma1}]  Recall that $\Theta_{x}$ is $1$ plus the time it takes for the random walk, starting in $x-1$, to reach $x$. 
By formula \eqref{formulatheta} and by the definition of a good triblock
we easily obtain bound \eqref{expectationk0}.

For studying the second moment of $\Theta_x$, we call $N_j$, with $j\in\{-C_n,\dots,x-1\}$, the random variable counting the number of visits to point $j$ before touching $x$, so that $\Theta_x=1+\sum_{j=-C_n}^{x-1} N_j$ and
\begin{align}\label{polipo}
(\Theta_x)^2\leq 2 C_n\sum_{j=-C_n}^{x-1}N_j^2\,.
\end{align}
Let us describe  the distribution of the random variables $N_j$: 
(i) $N_j$  is equal to $0$ with probability $P_{x-1}^{\bar\go}(T_{x}<T_j)$ (when the random walk reaches $x$ before ever visiting $j$);
(ii) $N_j$ is equal to a random variable $Y(j)$ with the remaining probability, 
where $Y(j)$ is a geometric random variable with parameter $q(j)=P_{j}^{\bar \go}(T_{x}<T_j)
$ (once $j$ is reached for the first time, the random walk returns to $j$ before touching $x$ with probability $1-q(j)$).
Hence, we get that $E^{\bar \go}[N_j^2] \leq E^{\bar \go}[Y(j)^2] \leq 1/q(j)^2$.

We notice that, on good blocks, because of formula~\eqref{classicone1},  $q(j)\geq c \e^{-8 q_n}$ for any $j\leq x-1$.
Therefore the second moment of $N_j$ is bounded by $1/q(j)^2 \leq c' \e^{16 q_n}$, 
so that taking the expectation in  \eqref{polipo} we get 
$\text{Var}\,^{\bar\go}(\Theta_{x})
	\leq c'' C_n^2 \e^{16q_n}$ as desired.
\end{proof}


\section{Reduction to a finite number of traps}
\label{sec:reduc2}

First of all, by a union bound and then Proposition \ref{trappolone}  with $\gd=  \frac{d_n \e^{-q_n}}{|\cJ_n| \rho_{\cB_j}} \leq |\cJ_n|^{-1} <1$, we get that for a.e.\ $\go$, for $n$ large enough (so that all trapping blocks are good, by Lemma~\ref{sogood})
\begin{align*}
P^{\go} \Big(  \Big| \sum_{j\in \cJ_n}T(\cB_j) -\sum_{j\in \cJ_n} \rho_{\cB_j} \tau_{\cB_j}  \Big| &\geq d_n \e^{- q_n} \Big)
	 \leq \sum_{j\in \cJ_n} P^{\go} \Big( \big| T(\cB_j) -\rho_{\cB_j} \tau_{\cB_j}  \big| \geq \frac{ d_n \e^{-q_n}  }{ |\cJ_n| } \Big) \\
	& \leq 2 |\cJ_n|  \e^{-  \e^{-q_n}} +  c  |\cJ_n| \e^{50 q_n} \Big(\frac{d_n \e^{-q_n}}{|\cJ_n|}\Big)^{-2}  \max_{j\in \cJ_n}  \rho_{\cB_j}  \, .
\end{align*}
Then, by Lemmas~\ref{disgiunti} and \ref{lem:maxrho}, we get that  $|\cJ_n|\leq \e^{5 q_n}$ and $ \max_{j\in \cJ_n}  \rho_{\cB_j} \leq d_n \e^{q_n}$ for $n$ large enough,  so that
\begin{equation}
P^{\go} \Big(  \Big| \sum_{j\in \cJ_n}T(\cB_j) -\sum_{j\in \cJ_n} \rho_{\cB_j} \tau_{\cB_j}  \Big| \geq  d_n \e^{- q_n }  \Big) \leq c' d_n^{-1}  \e^{ 70 q_n} \, .
\end{equation}
Together with Proposition~\ref{blocchetti}, we conclude that with $\bbP\otimes P^{\go}$-probability going to $1$,
\begin{equation}
\label{reductiontolargetraps}
\frac{1}{d_n} \big| T_n - \sum_{j\in \cJ_n} \rho_{\cB_j} \tau_{\cB_j} \big| \leq   \gep_n^{(1-\ga)/2} \xrightarrow{n\to\infty} 0 \, ,
\end{equation}
where we recall that $\gep_n = q_n^{-\gd}$ for some (small) fixed $\gd>0$.

Now, we show that the main contribution to $\sum_{j\in \cJ_n} \rho_{\cB_j} \tau_{\cB_j} $ comes from the blocks with the largest $\rho_{\cB}$'s (of order $d_n$).

\begin{proposition}
\label{prop:fewblocks}
For any $\eta,\eta'>0$, there exists $\gep = \gep(\eta, \eta')>0$ such that for $n$ sufficiently large
\begin{align}\label{eq:propreduction}
\bbP \otimes P^{\go} \Big(  \frac{1}{d_n} \sum_{j\in \cJ_n} \rho_{\cB_j} \tau_{\cB_j} \ind_{\{\rho_{\cB_j} \leq \gep d_n\}}   > \eta\,     \Big) \leq \eta' \, .
\end{align}
\end{proposition}

This proposition, together with \eqref{reductiontolargetraps}, shows that for any fixed $\eta,\eta'>0$, we have for $n$ large enough (recall $k_n:= n/C_n$)
\begin{equation}
\label{reductiontofew}
\frac{1}{d_n} \Big| T_n -\sum_{i =1}^{k_n} \rho_{\cB_i} \tau_{\cB_i} \ind_{\{\rho_{\cB_i} > \gep d_n\}}  \Big| \leq 2\eta
\end{equation}
with $\bbP\otimes P^{\go}$-probability at least $1-\eta'$.
We dropped the condition $j\in \cJ_n$ in the sum since the condition $\rho_{\cB_i} > \gep d_n$ ensures that $i \in \cJ_n$.

\smallskip

The difficulty here is that we are not able to obtain directly the tail of $\rho_{\cB} \tau_{\cB}$ (they both have heavy-tail and are not independent). 
Proposition~\ref{prop:fewblocks} is proved in two steps:

Step 1: We decompose the sum according to some ranges of values for $\rho_{\cB}$, and  dominate each term  by a sum of a constant number $N_{\ell}$ of i.i.d.\ random variables.

Step 2: By a union bound, we reduce the proof to controlling each term separately. In particular, an important estimate is Lemma~\ref{lem:limittauB}, which controls the tail of $\tau_{\cB}$ conditionally on having $\rho_{\cB}$ large (uniformly on the large value of $\rho_{\cB}$).

\subsection{Step 1}
We divide the contribution according to ranges of value of $\rho_{\cB_j}$:
\begin{align}
\label{eq:firstreduction}
\bbP \otimes P^{\go}  \Big(  \sum_{j\in \cJ_n} \rho_{\cB_j} \tau_{\cB_j} \ind_{\{\rho_{\cB_j}  < \gep  d_n \}}  > \eta\,   d_n     \Big)
	\leq  \bbP \otimes P^{\go}  \Big(   \sum_{\ell = \log_2( 1/\gep) }^{  2 q_n } 2^{-\ell +1}\sum_{j\in \cJ_n^{(\ell)}} \tau_{\cB_j}  > \eta \Big)
\end{align}
where we set $\cJ_n^{(\ell)} = \{ j \in \cJ_n , \,  2^{-\ell}  \leq \rho_{\cB_j} /d_n \leq 2^{-\ell +1}    \}$ (we also used that $q_n/\log 2 \leq 2q_n$).

Then, we control the number of terms in each sum, \textit{i.e.}\ we show that $|\cJ_n^{(\ell)}| \leq N_{\ell}$, for some well chosen $N_{\ell}$.
This is essentially due to the following lemma.
\begin{lemma}
\label{lem:jth}
For any $\gd'>0$, there is a constant $C_{\gd'}$ such that, for all $n\in \bbN$ and all $u \in(0,1)$, setting $i_u := \lceil u^{-\ga(1+\gd')} \rceil$,
\[
\bbP\Big(  \big| \{ 1\leq  j\leq n \, \colon \, \rho_j \geq  u d_n\} \big|   \geq  i_u  \Big) \leq  \big( C_{\gd} u \big)^{ - \gd' i_u \ga/4 } \leq 2^{- i_u} \, ,
\]
the last inequality being valid for $u$ small enough.
In the case $\ga_0=\ga_{\infty}$ (recall Assumption~\ref{mainassumption}),
we have that there is some $A>0$ such that for any $u \in (d_n^{-1/2}, 1)$
\[
\bbP\big( \big| \{ 1\leq  j\leq n \, \colon \, \rho_j \geq  u d_n\} \big|   \geq  A u^{-\ga}   \big) 
	\leq 2^{- u^{-\ga}}\, .
\]
\end{lemma}

\begin{proof} 
The probability we need to control is  
\begin{align*}
 \bbP\Big( \exists &\, x_1, \ldots, x_{i_u} \in \{1,\ldots, n\}:\; \rho_{x_j} \geq  u  d_n \text{ for all } j\leq i  \Big)\\
	&\leq \bbP\Big( \exists\,   x_1, \ldots, x_{i_u/2} \leq n  \text{ all odd or all even}, \, \text{s.t. }\rho_{x_j} \geq u d_n \text{ for all } j\leq i_u/2  \Big) \\
	& \leq  \binom{n}{i_u/2} \bbP\big( \rho_0 \geq u d_n \big)^{i_u/2} \, .
\end{align*}
The second inequality comes from a union bound, and from the fact that the $\rho_{x_j}$ are independent.
Then we can use that $\bbP(\rho_0>t)$ is regularly varying with exponent $\ga$, so that by Potter's bound (cf.~\cite{BGT89}) we have a constant $C'_{\gd}$ such that for any $u\in(0,1)$, $ \bbP\big( \rho_0 \geq u d_n \big) \leq C'_{\gd} u^{- \ga(1+\gd' /2)} \bbP(\rho_0 \geq d_n)$.
Then, using that $d_n$ has been chosen such that $\bbP( \rho_0 \geq d_n ) \sim 1/n$, together with the fact that $u^{- \ga(1+ \gd' /2)} \leq i_u u^{\gd' \ga /2}$, we obtain
\begin{equation}
\bbP\big(  \big| \{ 1\leq  j\leq n \, \colon \, \rho_j \geq  u d_n\} \big|   \geq  i_u  \big) \leq \frac{n^{i_u/2}}{ (i_u/2)!}  \Big( \frac{ C'_{\gd} i_u u^{\gd' \ga /2} }{n} \Big)^{i_u/2} \leq \big( C_{\gd}'' u \big)^{ - \gd' i_u \ga/4 } \, .
\end{equation}
The last inequality comes from the fact that there is a constant $c$ such that  $(i/2)! \geq (c i)^{i/2}$ for all $i\geq 1$.

In the case where $\ga_0=\ga_{\infty}$, we have that $g(u):=\psi(\e^{u})$ is regularly varying, recall Assumption~\ref{mainassumption}. Then we use the same union bound, setting $i_u = \lceil A u^{-\ga} \rceil$: using that $\bbP\big(\rho_0 \geq u d_n\big) \leq c u^{-\ga} \psi(u d_n) d_n^{-\ga}$, and that $\psi(d_n)d_n^{-\ga} \leq c/n$, we get that 
\[
\bbP\big( \big| \{ 1\leq  j\leq n \, \colon \, \rho_j \geq  u d_n\} \big|   \geq  i_u \big)
\leq \frac{n^{i_u/2}}{ (i_u/2)!} \Big( \frac{i_u /A}{n} \frac{\psi(u d_n)}{\psi(d_n)}\Big)^{i_u/2} 
\leq \Big( c' A^{-1} \Big)^{i_u/2}\,.
\]
For the last inequality, we use that  $\log d_n \geq \log (ud_n) \geq \frac12 \log d_n$ for any$u \in (d_n^{-1/2}, 1)$, so that $\psi(u d_n) \leq c \psi(d_n)$ since $t \mapsto \psi(\e^{t})$ is regularly varying.
Having fixed $A$ sufficiently large in the above display gives the conclusion of the Lemma.
\end{proof}

\smallskip

Now, let $N_{\ell}:= 2^{ \ga \ell (1+\gd')}$ if $\ga_0\neq \ga_{\infty}$ and $N_{\ell} := A 2^{\ga \ell}$ if $\ga_0 =\ga_{\infty}$ (in any case $N_{\ell} \geq 2^{\ga\ell}$):  thanks to Lemma~\ref{lem:jth}, we have that for $\ell$ large
\[
\bbP \big( |\cJ_n^{(\ell)}| \geq N_{\ell} \big)
\leq \bbP\Big( \big| \{j \colon  \rho_{\cB_j} \geq 2^{-\ell} d_n\}\big| \geq N_{\ell} \Big)
\leq 2^{-N_{\ell}} \, .
\]
Hence, we get that
\begin{equation}
\label{def:Hell}
 \bbP \otimes P^{\go}\Big(  \sum_{j\in \cJ_n^{(\ell)}} \tau_{\cB_j}   > \cH_{\ell}    \Big) \leq 2^{-2^{\ga \ell}}  \, \qquad \text{with } \cH_{\ell} :=  \sum_{i=1}^{N_{\ell}}  \tilde\tau_{\cB_i}  \, ,
\end{equation}
where  $\tilde \tau_{\cB_i} := \tilde \tau_{\cB_i}^{(\ell)}$ are i.i.d.\ random variables, distributed as $\tau_{\cB_j}$ conditioned on $j \in \cJ_n^{(\ell)}$.

Going back to \eqref{eq:firstreduction}, we therefore get that
\begin{equation}
\label{endofstep1}
 \bbP \otimes P^{\go}  \Big(\sum_{\ell = \log_2( 1/\gep) }^{2 q_n}2^{-\ell+1}\sum_{j\in \cJ_n^{(\ell)}} \tau_{\cB_j}  > \eta \Big)
 	\leq \bbP \otimes P^{\go}  \Big(   \sum_{\ell = \log_2( 1/\gep) }^{2 q_n } 2^{-\ell+1}  \cH_{\ell}  > \eta  \Big)
 		+ \sum_{\ell \geq \log_2(1/\gep)} 2^{-2^{\ga \ell}} \, ,
\end{equation}
the last sum being bounded by a constant time $2^{-1/\gep^{\ga}}$, which goes to $0$ as $\gep\to 0$.

\subsection{Step 2.}
Using a union bound and Markov's inequality, we have
\begin{align}
\bbP \otimes P^{\go}  \Big(   \sum_{\ell = \log_2( 1/\gep) }^{  2 q_n }  2^{-\ell+1}  \cH_{\ell}    > \eta  \Big)
	\le & \sum_{\ell = \log_2(1/\gep)}^{ 2 q_n} \bbP \otimes P^{\go}\big(   \cH_{\ell}    >    2^{\ell}  \big) \notag\\
 	& + \frac{2}{\eta} \sum_{\ell = \log_2(1/\gep)}^{2 q_n }2^{-\ell} \bbE \otimes E^{\go}\big[  \cH_{\ell} \ind_{\{ \cH_{\ell}    \leq     2^{\ell}   \}}  \big] \, .
\label{unionbound}
\end{align}
We treat the two terms separately.

\smallskip
\noindent
{\bf  First term in \eqref{unionbound}.}
We need to bound the probability $\bbP \otimes P^{\go}\big( \cH_{\ell} >    2^{\ell}  \big)$, with $\cH_{\ell} = \sum_{i=1}^{N_{\ell}} \tilde \tau_{\cB_i}$.
We can use that $\bbP(j \in \cJ_n^{(\ell)})$ is bounded below by a constant times $\bbP(\rho_{\cB_j} > 2^{-\ell} d_n)$, so that we get that $\bbP\otimes P^{\go} \big( \tilde \tau_{\cB} >t \big)  \leq c\,  \bbP\otimes P^{\go} \big( \tau_{\cB} >t  \mid \rho_{\cB} >2^{-\ell} d_n\big)$.
We now need the following lemma (that we prove in Appendix~\ref{sec:appb2}), to control the distribution on $\tau_{\cB}$ conditionally on having $\rho_{\cB}$ large. 
\begin{lemma}
\label{lem:tauB}
Define $\xi_{\cB}:= \frac{\theta_{\cB}}{ p_{\cB}}$, so that $\tau_{\cB} = \xi_{\cB} \, \mathbf{e}_{\cB}$.
For any $\tilde \gd , \bar \gd >0$ fixed small enough,  there is some $n_0$ such that for any $n\geq n_0$ and uniformly for $1\leq t \leq n^{2\bar \gd}$, $A_n \geq n^{\frac1\ga-\bar \gd}$,
\begin{align}
\bbP \big(  \xi_{\cB} >t &\mid \rho_{\cB} \geq  A_n \big) 
\leq  c t^{-2 \ga+\tilde\gd} +  c L_0(t) t^{-\ga_0}  f_{\infty}(A_n) + c L_{\infty}(t) t^{-\ga_{\infty}}  f_0(A_n) \, ,
\end{align}
with $f_{\infty}(A_n):=  \bbP \big( c_{-1} >A_n \big)  /  \bbP(\rho_0 >A_n) $ and $f_0(A_n):=\bbP \big( 1/c_0 >A_n \big)  /  \bbP(\rho_0 >A_n) $.
We also have that
\begin{align}
\bbP\otimes P^{\go} \big(  \tau_{\cB} >t &\mid \rho_{\cB} \geq  A_n \big) 
\leq  c t^{-2 \ga+\tilde\gd} +  c L_0(t) t^{-\ga_0}  f_{\infty}(A_n) + c L_{\infty}(t) t^{-\ga_{\infty}}  f_0(A_n) \, ,
\end{align}
\end{lemma}
Let us stress that, in view of Proposition~\ref{prop:rho0}, if $\bbE[1/c_0^{\ga}]<+\infty$ (resp.~$\bbE[c_{-1}^{\ga}]<+\infty$), then $ f_\infty(A_n)$ (resp.~$f_0(A_n)$) is bounded by a constant, whereas if $\bbE[1/c_0^{\ga}] = +\infty$ (resp.~$\bbE[c_{-1}^{\ga}]=+\infty$), then $ f_\infty(A_n) \to 0$ (resp.~$f_0(A_n)\to 0$).
Therefore, as a corollary, we get that for any fixed $t>1$,
\begin{align}
\limsup_{n\to+\infty} \bbP & \big(  \xi_{\cB} >t \mid \rho_{\cB} \geq  A_n \big) \notag\\
& \leq c t^{- 2\ga+\tilde \gd} + c L_0(t) t^{-\ga_0} \ind_{\{\bbE[1/c_{0}^{\ga}]<+\infty \}} +  c L_\infty(t) t^{-\ga_{\infty}} \ind_{\{\bbE[c_{-1}^{\ga}]<+\infty \}}  \,.
\label{limitingtaub}
\end{align}

\smallskip
With Lemma~\ref{lem:tauB} at hand, we decompose the probability $\bbP\otimes P^{\go} (\cH_{\ell}>z)$ for some $z>1$, according to whether there is some $1\leq i\leq N_{\ell}$ with $\tilde\tau_{\cB_i} > z$ (we will choose $z=2^{\ell}$, but the general bound will be useful). With a union bound and Markov's inequality, we get
\begin{align}
\bbP \otimes P^{\go}\Big( \sum_{i=1}^{N_{\ell}}  \tilde\tau_{\cB_i}  > z \Big)  \leq  N_{\ell}
& \bbP\otimes P^{\go} \big( \tilde \tau_{\cB} >   z\big)  + z^{-1}\,  N_{\ell}\, \bbE\otimes E^{\go} \big[ \tilde\tau_{\cB} \ind_{\{ \tilde \tau_{\cB} \leq z  \}} \big] \, .
 \label{truncated+Markov}
\end{align}

For the second term, we write
\begin{align*}
\bbE\otimes E^{\go} \big[ \tilde\tau_{\cB}  \ind_{\{ \tilde \tau_{\cB} \leq   z  \}} \big] 
	&\leq 1+  \int_{1}^{  z}   \bbP\otimes P^{\go} ( \tilde\tau_{\cB} >t ) \dd t \\
	& \leq  c   z^{ 1 - (\ga+\tilde \gd)} + c L_0( z)  z^{(1-\ga_0)} f_{\infty}(2^{-\ell} d_n)  + c L_{\infty}(  z)  z^{1-\ga_\infty} f_0(2^{-\ell} d_n)
\end{align*}
where  for the second inequality we used Lemma~\ref{lem:tauB} and integrated over $t$ using Karamata's Theorem \cite[Prop.~1.5.8]{BGT89}. We have exploited the fact that $\ga_0,\ga_{\infty}<1$ and, for the  integral $\int_1^{z} t^{- 2\ga+\tilde \gd}$, we have used that $-2\ga +\tilde \gd\leq -(\ga+\tilde \gd)$ with $\ga+\tilde \gd<1$,  for $\tilde \gd$ small enough.
Using Lemma~\ref{lem:tauB} for the first part of \eqref{truncated+Markov}, we get that  for any $z>1$
\begin{align}
\bbP  \otimes P^{\go}  \big( \cH_{\ell}  >  z \big) \leq  c' N_{\ell}\,  \Big( z^{ - (\ga+\tilde \gd)} 
+   L_0(  z)  z^{-\ga_0} f_{\infty}(2^{-\ell} d_n)  +    L_{\infty}( z)  z^{-\ga_\infty } f_0(2^{-\ell} d_n)  \Big) \, .
\label{lastgeneralbound}
\end{align}

\noindent Now, we set $z=2^{\ell}$, and we treat the first term in \eqref{endofstep1} separately the different cases.

\smallskip 

\textbullet\ If $\ga_0 > \ga_{\infty}$, then $\ga=\ga_{\infty}$ (the case where $\ga_{\infty} >\ga_0$ is symmetric) and $\bbP(\rho_0> t) \sim c L_{\infty}(t) t^{-\ga_{\infty}}$ since $\bbE[1/c_0^{\ga}] <+\infty$ (cf.~Proposition~\ref{prop:rho0}). 
We therefore get that  $f_{\infty}(A_n)$ is bounded by a constant (uniformly for $A_n \geq 1$), and that $f_{0}(A_n) \leq c A_n^{\ga_{\infty}- \ga_0 + \tilde \gd}$ (using Potter's bound).
Recalling that $N_{\ell} = 2^{\ga\ell(1+\gd')}$, and setting $z=2^{\ell}$, the left-hand side in \eqref{lastgeneralbound} is therefore bounded by a constant times 
\[
 2^{\ell ( \ga(1+\gd') - (\ga + \tilde \gd))} +  2^{\ell ( \ga(1+\gd') - \ga_0 + \tilde \gd)} +  2^{\ell (\gd' +\tilde \gd)} (2^{-\ell} d_n)^{\ga_{\infty}- \ga_0 +\tilde \gd }
 \leq 2^{- \tilde\gd \ga \ell/2 } + 2^{\ell} d_n^{-(\ga_{0}-\ga_{\infty})/2} \, ,
\]
where the last bound holds provided that $\tilde\gd$ then $\gd'$ have been fixed sufficiently small.
Going back to the first term in \eqref{unionbound}, we have that in the case $\ga_0>\ga_{\infty}$,
\begin{align}
\sum_{\ell = \log_2(1/\gep)}^{ 2 q_n }  \bbP \otimes P^{\go}\big( \cH_{\ell}>2^{\ell}  \big) 
	& \leq  c \sum_{\ell = \log_2( 1/\gep) }^{ +\infty}   2^{- \tilde\gd \ga \ell/2 } +
c d_n^{-(\ga_{0}-\ga_{\infty})/2} \sum_{\ell = 1 }^{ 2 q_n} 2^{\ell}  \notag \\
	&\leq c' \gep^{\tilde\gd \ga /2} +c' 4^{q_n} d_n^{-(\ga_{0}-\ga_{\infty})/2} \, .
\label{casega0>gainfty}
\end{align}
This can be made arbitrarily small by taking $\gep$ small enough and $n$ large.

\textbullet\ In the case $\ga_0 = \ga_{\infty} = \ga$ we have $\bbP(\rho_0>t) \sim \psi(t) t^{-\ga}$ (see~Proposition~\ref{prop:rho0}), so that $f_0(t) \sim L_0(t)/\psi(t)$ and $f_{\infty}(t) \sim L_{\infty}(t) /\psi(t)$.
Now we can use that $L_0(\e^{t})$, $L_{\infty}(\e^t)$ are regularly varying by Assumption \ref{mainassumption}, so that $\psi(\e^t)$ is regularly varying, and therefore so are $f_0(\e^{t}), f_{\infty}(\e^t)$.
Then, using that $d_n 2^{-\ell} \geq d_n^{1/2}$ for all $0\le \ell \le   q_n$, we get that there is a constant $c$ such that $f_0(2^{-\ell}d_n) \leq  c f_0(d_n)$ and $f_{\infty}(2^{-\ell}d_n) \leq c f_{\infty}(d_n)$.
Recalling that $N_{\ell}= A 2^{\ga \ell}$, we therefore get from \eqref{lastgeneralbound} (setting $z=2^{\ell}$) that there is a constant $c_A$ depending on $A$ such that
\begin{align*}
\bbP \otimes  P^{\go}\Big( \cH_{\ell} >   2^{\ell}  \Big) \leq c_{A} \big( 2^{-\tilde \gd \ell }
+  L_0(2^{\ell} ) f_{\infty}(d_n) +  L_{\infty} (2^{\ell} )  f_0 (d_n)   \big)\, .
\end{align*}
Going back to the first term in \eqref{unionbound}, we get that 
\begin{align}
\sum_{\ell = \log_2(1/\gep)}^{2 q_n } & \bbP \otimes P^{\go} \big(   \cH_{\ell}    >    2^{\ell}  \big) 
	\leq  c_{A} \sum_{\ell = \log_2( 1/\gep) }^{ +\infty}   2^{-\tilde \gd\ell  }  \notag\\
	& + c_{A} f_{\infty}(d_n) \sum_{\ell =  \log_2( 1/\gep) }^{2q_n }  L_0(2^{\ell} )  +  c_{A} f_{0}(d_n) \sum_{\ell =  \log_2( 1/\gep) }^{ 2q_n }  L_{\infty}(2^{\ell} ) \, .  
\label{casega0=gainfty}
\end{align}
The first term is bounded by a constant times $\gep^{\tilde \gd}$, and goes to $0$ as $\gep \downarrow 0$. The second and third term can be treated in an analogous way, so we focus on the second one.

\smallskip
\textbf{a)} If $\bbE[ 1/c_0^{\ga} ] <+\infty$, then  we already mentioned that $f_{\infty}(d_n)$ is bounded by a constant (thanks to Proposition~\ref{prop:rho0}). We also have that 
\begin{equation}
\sum_{\ell \geq   \log_2( 1/\gep) }  L_0(2^{\ell} ) 
	\leq c \sum_{\ell \geq   \log_2( 1/\gep) } (2^{\ell})^{\ga} \bbP(1/c_0 >2^{\ell}) 
	\leq  c' \bbE\big[  (1/c_0)^{\ga}  \ind_{\{  1/c_0 > 1/\gep  \}} \big] \, ,
\end{equation}
and hence it goes to $0$ as $\gep\to 0$.

\smallskip

\textbf{b)} If $\bbE[ 1/c_0^{\ga}] = +\infty$ but $ \bbE\big[ c_{-1}^{\ga} \big]<+\infty$, then we get from Proposition~\ref{prop:rho0} that $f_{\infty}(d_n)$ is bounded by a constant times $L_{\infty}(d_n)/L_0(d_n) \leq c (\log n)^{\gamma_{\infty} -\gamma_0} \gp_{\infty}( \log n) /\gp_0(\log n)$ with $\gamma_0 > -1 >\gamma_{\infty}$ (recall Assumption \ref{mainassumption}), so 
\begin{equation}
\label{pfffff}
f_{\infty}(d_n) \sum_{\ell =  \log_2( 1/\gep) }^{ q_n }  L_0(2^{\ell} )  
	\leq c \frac{\gp_{\infty}( \log n)  }{ \gp_0(\log n)} (\log n)^{\gamma_{\infty} - \gamma_0} \sum_{\ell =1}^{ 2q_n } \gp_0(\ell) \ell^{\gamma_0}.
\end{equation}
Using that $2q_n\leq \log n$, the last sum is bounded by a constant times $\gp_0(\log n) (\log n)^{1+\gamma_0}$, so that the second term in  \eqref{casega0=gainfty} goes to $0$ as $n\to\infty$, since $\gamma_{\infty}<-1$.

\smallskip
All together, we have obtained that, provided  $n$ is large enough and $\gep$ is small enough (how small depends on $\eta,\eta'$), \eqref{casega0=gainfty} is bounded by $\eta'/4$.

\medskip
\noindent
{\bf Second term in \eqref{unionbound}.}
We just adapt the previous argument.
Using the bound \eqref{lastgeneralbound} and integrating it over $z$ (recall $\ga+\tilde\gd ,\ga_0,\ga_{\infty}<1$), we have that
\begin{align*}
 \bbE & \otimes E^{\go}   \big[\cH_{\ell}  \ind_{\{ \cH_{\ell}    \leq     2^{\ell}   \}}  \big]
\leq 1+  \int_{1}^{2^{\ell}} \bbP  \otimes P^{\go}\big( \cH_{\ell}  >  z \big) \dd z \\
& \leq 1+ c N_{\ell} \Big( 2^{\ell (1- \ga -\tilde\gd)}
  + L_0(2^{\ell}) 2^{\ell(1-\ga_0)} f_{\infty}(2^{-\ell} d_n) 
  + L_{\infty}(2^{\ell}) 2^{\ell(1-\ga_{\infty})} f_{0}(2^{-\ell} d_n) \Big) \, .
\end{align*}
Hence we get that
\begin{align*}
2^{- \ell}  \bbE  \otimes E^{\go}&   \big[ \cH_{\ell}   \ind_{\{ \cH_{\ell}    \leq     2^{\ell}   \}}  \big] \\
& \leq  c' N_{\ell} \Big(  2^{-\ell(\ga+ \tilde \gd )} + N_{\ell}    L_0(  2^{\ell})  2^{-\ga_0 \ell } f_{\infty}(2^{-\ell} d_n)  + L_{\infty}( 2^{\ell} )  2^{-\ga_\infty \ell } f_0(2^{-\ell} d_n)  \Big)\,.
\end{align*}
This is the analogous of \eqref{lastgeneralbound}, and the proof concludes identically to what is done above.

\smallskip
All together, going back to \eqref{unionbound}, we have shown that there is a constant $C$ such that
\begin{align*}
\bbP \otimes P^{\go}  \Big(   \sum_{\ell = \log_2( 1/\gep) }^{ 2 q_n}   2^{-\ell}  \cH_{\ell}    > \eta  \Big)
\leq  \frac{C}{\eta} \big( \gep' + \gep_n  \big) \, ,
\end{align*}
with $\gep' \to 0$ as $\gep\to 0$, and $\gep_n \to 0$ as $n\to+\infty$.
This concludes the proof of Proposition~\ref{prop:fewblocks}.
\qed


\section{Conclusion of the proof of Theorem~\ref{thm:passagetime}}
\label{sec:conv}

\subsection{Convergence of the depths, positions and block-crossing times}

Recall we defined $\xi_{\cB} := \frac{\theta_{\cB}}{p_{\cB}}$, so  $\tau_{\cB}= \xi_{\cB} \, \mathbf{e}_{\cB}$ with $\mathbf{e}_{\cB}$ an exponential r.v.\ independent of the environment.
In Lemma~\ref{lem:limittauB} below, we state that under Assumption \ref{assumptionA}, conditionally on having $\rho_{\cB} \to +\infty$, $\xi_{\cB}$ converges in distribution to a r.v.\ $\zeta$, that we define as
\begin{equation}
\label{def:zeta}
\zeta:= 2  \Big( 1+  B\, { \bar c_{0}} V  +  (1-B) { \tfrac{1}{ \bar c_{-1}}} W \Big) \, ,
\end{equation}  
where:

$\ast$\ $B$ is a Bernoulli r.v.\ with parameter $q=1$ if $ \bbE[c_{-1}^{\ga}] =+\infty$; $q=0$ if $\bbE[1/c_{0}^{\ga}] = +\infty$; and  $q= \lim_{t\to+\infty} \frac{\bbE[c_{-1}^{\ga}] L_0(t) }{ \bbE[c_{-1}^{\ga}] L_0(t) + \bbE[1/c_{0}^{\ga}] L_{\infty}(t)}$ if  $ \bbE[c_{-1}^{\ga}] <+\infty $, $\bbE[1/c_{0}^{\ga}] < +\infty$.

$\ast$\ $V,W$ are defined by (see the analogous quantities for $n$ fixed in \eqref{p0}-\eqref{theta0})
\begin{equation}
V = V^{\lambda}:= \sum_{j\geq 1} \frac{1}{c_j} \e^{- \lambda (j+1)} \, , 
\qquad 
W = W^{\lambda}:=\sum_{j\geq 2} c_{-j} \e^{- \lambda(j+1)} \, .
\end{equation}

$\ast$\ If $\bbE[c_{-1}^{\ga}]<+\infty$, $\bar c_{-1}$ is a r.v.\ whose c.d.f.\ is given by $F_{\bar c_{-1}} (u) := \frac{1}{ \bbE[c_{-1}^{\ga}] } \bbE[c_{-1}^{\ga} \1{c_{-1}\leq u} ]$. If $\bbE[1/c_0^{\ga}]<+\infty$, then $1/\bar c_0$ is a r.v.\ with c.d.f.\  $F_{1/\bar c_0} (u) := \frac{1}{\bbE[1/c_0^{\ga}]}\bbE[1/c_0^{\ga} \1{1/c_0\leq u} ]$.

\begin{lemma}
\label{lem:limittauB}
Under  Assumption \ref{assumptionA}, conditionally on having $\rho_{\cB}> \gep d_n$, $\xi_{\cB}:=\frac{\theta_{\cB}}{p_{\cB}}$ converges in distribution as $n \to+\infty$ to the random variable $\zeta$ defined above in \eqref{def:zeta}.
More precisely, we have
\[
\lim_{n\to+\infty} \bbP  \big( \xi_{\cB} \in  \cdot  \mid \rho_{\cB}>\gep d_n \big) = \bbP ( \zeta\in \cdot)\, .
\]
Additionally, we have that $\bbE[\zeta^{\ga} ]<+\infty$. 
\end{lemma}

If $B=1$ it means that the trap is a well (roughly, $\theta_{\cB} \to 2$), while if $B=0$ it means that the trap is a wall (roughly, $p_{\cB}\to 1$).
Hence, if $q=1$ only wells can occur, while if $q=0$ only walls can occur.
The proof of Lemma~\ref{lem:limittauB} is provided in Appendix~\ref{sec:appb2}.

\smallskip
Then, we rely on Lemma~\ref{lem:limittauB} to show the following (recall $k_n = \lfloor n/C_n \rfloor$).

\begin{proposition}
\label{prop:convergence}
For any $\gep>0$, we have the following convergence in distribution:
\[
\Big\{  \Big( \frac{i}{k_n} , \frac{\rho_{\cB_i}}{d_n} ,  \xi_{\cB_i} \Big)  \colon    1 \leq i \leq k_n,  \rho_{\cB_i} > \gep d_n \Big\}
\ \Longrightarrow \
 \cP_{\gep} \, ,
\]
where $\cP_{\gep}$ is a PPP on $[0,1]\times \bbR_+ \times \bbR_+$ with intensity $\dd x \, \ga w^{-(1+\ga)} \ind_{\{w>\gep\}} \dd w \mu(\dd  z)$, and $\mu$ is the law of a random variable $\zeta$ defined in \eqref{def:zeta}.
\end{proposition}

\begin{proof}
The first remark we make is that the $(\rho_{\cB_i}, \xi_{\cB_i})_{1\leq i\leq k_n}$ are not independent: however $(\rho_{\cB_i}, \xi_{\cB_i})$ and $(\rho_{\cB_j}, \xi_{\cB_j})$ are independent as soon as $|j-i|>2$.
But because of Lemma~\ref{disgiunti}, the indices $i$ such that $\rho_{\cB_i} > \gep d_n$ are a.s.\ well separated, and we can do as if the $\rho_{\cB_i}, \xi_{\cB_i}$ appearing in $\{  ( \frac{i}{k_n} , \frac{\rho_{\cB_i}}{d_n} ,  \xi_{\cB_i} )  \colon    1 \leq i \leq k_n,  \rho_{\cB_i} > \gep d_n \} $ were independent.

We therefore only have to consider the joint distribution of $(\rho_{\cB}, \xi_{\cB})$ on the event $\rho_{\cB}>\gep d_n$.
Thanks to Lemma~\ref{lem:limittauB}, conditionally on $\rho_{\cB}>\gep d_n$,  $\xi_{\cB}$ converges in distribution to $\zeta$ defined in \eqref{def:zeta}, which does not depend on the value of $\rho_{\cB}$: on the event $\rho_{\cB}>\gep d_n$, $\xi_{\cB}$ is asymptotically independent of $\rho_{\cB}$.
It remains to see that, for any $t>0$, as $n\to+\infty$
\begin{align*}
\bbP(\rho_{\cB} \geq  t d_n) &= \bbP\big( \exists\ x\in \{1, \ldots, C_n\}, \rho_x > t d_n \big) \sim C_n \bbP(\rho_0 > t d_n) \\
&  \sim  C_n \psi(t d_n) (t d_n)^{-\ga}  \sim \frac{1}{k_n} t^{-\ga}\, .
\end{align*}
Here, we used that there is at most one trap larger than $d_n \e^{-q_n}$ in a block $\cB$, and then that $\psi(d_n )d_n^{-\ga} \sim 1/n$  by definition of $d_n$ (recall also that $k_n = n/C_n$ is the number of blocks).

All together, for any $t>\gep,v>0$, and for any $i\geq 1$, $\bbP(  \frac{\rho_{\cB_i}}{d_n} >t , \xi_{\cB_i} > v )$ is asymptotically equivalent to $\frac{1}{k_n} t^{-\ga} \bbP(\zeta >v)$. 
By the Poisson approximation (as mentioned above, we can do as if the $(\rho_{\cB_i}, \xi_{\cB_i})_{i\geq 0}$ were independent), this is enough to conclude.
\end{proof}

\subsection{Convergence in distribution of $(T_n/d_n)_{n\geq 1}$}

With Proposition \ref{prop:convergence} at hand,  and recalling that $\tau_{\cB} = \xi_{\cB} \mathbf{e}_{\cB}$ where $\mathbf{e}_{\cB}$ is an exponential random variable independent of $\go$, we easily have the convergence in distribution, under $\bbP\otimes P^{\go}$,
\begin{equation}
\label{convergencesumK}
\frac{1}{d_n}\sum_{i =1}^{k_n} \rho_{\cB_i}\tau_\cB  \ind_{\{\rho_{\cB_i} > \gep d_n\}}   \ \Longrightarrow \ \sum_{(x, w, z,r) \in \widetilde \cP} w\, z \, r \, \ind_{\{x\leq 1\}}\ind_{\{ w > \gep \}} \, ,
\end{equation}
where $\widetilde \cP$ is a PPP on $(\bbR_+)^{4}$ with intensity $\dd x \, \ga w^{-(1+\ga)}  \dd w  \, \mu(\dd z)\, \e^{-r }\dd r$.

All together, thanks to \eqref{reductiontolargetraps} and Proposition~\ref{prop:fewblocks}, and letting $\gep\downarrow 0$ (the right-hand-side of \eqref{convergencesumK} is monotone in $\gep$), we obtain the following convergence in distribution
\begin{equation}
\label{eq:compoundPoisson}
\frac{1}{d_n} T_n  \ \Longrightarrow \ \sum_{(x, w, z,r) \in \widetilde \cP} w\, z \, r \, \ind_{\{x\leq 1\}} \stackrel{(d)}{=} 
  \Big( \frac{\pi \ga}{ \sin(\pi \ga)} \bbE[\zeta^{\ga}] \Big)^{1/\ga}\,   \sum_{(x, v) \in  \cP}v \ind_{\{x\leq 1\}}
\, ,
\end{equation}
with $\cP$ a PPP on $(\bbR_+)^2$ of intensity $\dd x\, \frac{\ga}{\Gamma(1-\ga)} v^{-(1+\ga)} \, \dd v$.
Note that a crucial point, shown in Lemma~\ref{lem:limittauB}, is that $\bbE[\zeta^{\ga}]<+\infty$.

For the last identity in distribution in \eqref{eq:compoundPoisson}, we have used that $\{(x,w z r) \}_{(x,w,z, r) \in \tilde \cP}$  is a PPP on $(\bbR_+)^2$ with intensity $\dd x \ga c_{\ga, \mu}  u^{-(1+\ga)} \dd u$, where $c_{\ga,\mu} = \int_{0}^{\infty} z^{\ga} \mu(\dd z) \int_0^{\infty} r^{\ga} \e^{-r }\dd r = \bbE[\zeta^{\ga}]  \Gamma(1+\ga)$. In turn, we get that $\{(x, (c_{\ga,\mu} \Gamma(1-\ga))^{-1/\ga}w z r) \}_{(x,w,z, r) \in \tilde \cP}$  is a PPP on $(\bbR_+)^2$ with intensity $\dd x\,  \frac{\ga}{\Gamma(1-\ga)}  v^{-(1+\ga)} \dd v$ (\textit{i.e.}\ with the same law as $\cP$): this gives \eqref{eq:compoundPoisson}, using also that $\Gamma(1+\ga) \Gamma(1-\ga) = \pi \ga/ \sin(\pi \ga)$.

Now, the term $ \sum_{ (x,v)\in \cP} v \ind_{\{x\leq 1\}}$ on the right-hand-side of \eqref{eq:compoundPoisson} is a standard $\alpha$-stable subordinator at time $1$, with Laplace exponent $t^{\alpha}$ (by straighforward calculation, recall $t^{\ga} = \frac{\ga}{\Gamma(1-\ga)} \int_{0}^{\infty} (1- \e^{- t s}) s^{-(1+\ga)} \dd s$, see \cite{Ber99}).

\smallskip
As a matter of fact, our proof shows that for any fixed $u>0$,  we have the convergence in distribution, as $n\to+\infty$: under $\bbP\otimes P^{\go}$
\begin{equation}
\label{conv:finitedim}
\cT_n(u):=\frac{1}{d_n} T_{\lfloor u n \rfloor}   \Longrightarrow \Big( \frac{\pi \ga}{ \sin(\pi \ga)} \bbE[\zeta^{\ga}] \Big)^{1/\ga}\,   \,   \mathcal{S}_\ga (u) \, ,
\end{equation}
where $(\mathcal{S}_{\ga}(u))_{u\geq 0}$ is the $\ga$-stable subordinator with Laplace exponent $\log \mathbf{E}[\e^{-t\cS_{\ga}(u)}] =  u t^{\ga}$; note that $\cS_\ga(u)$  can be written, analogously to~\eqref{eq:compoundPoisson}, as
$\sum_{(x,w) \in \cP} w \ind_{\{x\leq u\}}$.

\begin{rem}
We believe that one could also obtain a quenched version of Theorem~\ref{thm:passagetime}, in analogy with Corollary~1 in \cite{ESTZ13}.
More precisely, in view of \eqref{convergencesumK}, we expect that with high $\bbP$-probability, the law of $T_n/d_n$ conditionally on $\go$ is close (for instance in Wasserstein distance) to the law of $\sum_{p=1}^{\infty}  \varpi_p \,  \mathbf{e}_p$ conditionally on $(\varpi_{p})_{p\geq 1}$, where $(\varpi_p)_{p\geq 1}$ is a Poisson Point Process of intensity $\ga \bbE[\zeta^{\ga}]  \varpi^{-(1+\ga)}$ coupled with the environment $\go$,  and $(\mathbf{e}_p)_{p\geq 1}$ are i.i.d.\ exponential r.v.\ of parameter $1$, independent of $(\varpi_p)_{p\geq 0}$. We think, though, that the convergence in distribution of Theorem \ref{thm:Biased2} already describes the essence of the limiting behavior of the walk. We prefer not to develop the technical details for the quenched result, avoiding  to make the paper heavier.
\end{rem}

\subsection{Process convergence, conclusion of the proof of Theorem~\ref{thm:passagetime}}
\label{sec:convproc}

First of all, we reduce ourselves to the study of
\begin{equation}
\label{deftildeT}
\widetilde \cT_n(u) := \frac{1}{d_n} \sum_{j=1}^{ \lfloor u k_n \rfloor} \rho_{\cB_j } \tau_{\cB_j}  \ind_{\{\cB_j \in \bar \Gamma_n\}} \, ,
\end{equation}
the \emph{coarse} version of $\cT_u^{(n)}$ (where we only keep the crossing times of good trapping blocks). Thanks to Propositions~\ref{blocchetti} and~\ref{trappolone} we know that, with probability going to $1$, the $M_1$-distance between $(\cT_u^{(n)})_{u\in [0,1]}$ and $(\widetilde \cT_{u}^{(n)})_{u \in [0,1]}$ goes to $0$ as $n\to+\infty$.
In view of Lemma~\ref{disgiunti}, $\bbP\otimes P^{\go}$-a.s., for $n$ large enough the trapping blocks are disjoint, and the non-zero terms $\rho_{\cB_{j}} \tau_{\cB_j} \ind_{\{\cB_j \in \bar \Gamma_n\}}$ are independent.
Hence, we may reduce to studying
\begin{equation}
\label{defhatT}
\widehat \cT_{n}(u) := \frac{1}{d_n} \sum_{i=1}^{\lfloor u k_n \rfloor} Y_i^{(n)} \, ,
\end{equation}
with $(Y_i^{(n)})_{i\geq 1}$ i.i.d.\ copies of $\rho_{\cB} \tau_{\cB} \ind_{\{\cB \in \bar\Gamma_n\}}$.
It is then clear that the proof of \eqref{conv:finitedim} shows that all finite-dimensional marginals of $(\widehat \cT_{n}(u))_{u\in [0,1]}$ converge to those of the $\ga$-stable subordinator $(\cS_{\ga}(u))_{u\in [0,1]}$, since for $v>u$, $(\widehat \cT_{n}(v) - \widehat \cT_n (u))$ is independent of $\widehat \cT_n(u)$ for all~$n$.

We can easily upgrade this to a process convergence: thanks to Prohorov's theorem, and since we already have the convergence of the finite-dimensional distributions, we only need to show that $\{ \widehat \cT_n , n \in \bbN\}$ is tight,  see e.g.\ \cite[Thm.~15.1]{Bill}. 
Showing the tightness is relatively standard, since we are working here with increasing processes $\widehat \cT_n$ (and actually follows from the finite-dimensional convergence): the proof is identical to that in \cite[Sec.~5]{BABC08} (or to that of \cite[Sec.~11]{FK16}). 
Tightness can also be seen as a consequence of Proposition~\ref{prop:fewblocks}, which says that with high probability (uniform in $n$) the main contribution to $\widehat \cT_{n}$ comes from jumps larger than $\gep$, \textit{i.e.}\ $\widehat \cT_{n}(u)- \frac{1}{d_n} \sum_{i=1}^{\lfloor u k_n \rfloor} Y_i^{(n)} \ind_{\{Y_i^{(n)} >\gep d_n\}}$ is uniformly smaller than $\eta$ with probability at leat $1-\eta'$ (this correspond to the (15.11) tightness criterion of \cite[Thm.~15.4]{Bill}, the part (15.10) following from finite-dimensional convergence).
This concludes the proof of Theorem~\ref{thm:passagetime}.


\section{Case of well-and-wall traps}
\label{sec:wellandwalls}

\noindent Throughout this section we will assume Assumption~\ref{assumptionB}. 

\subsection{Definition of  $k$-distant traps and their properties}

Recall that the maximal depth of a trap between $0$ and $n$ will be of order $d_n$, see  the definition~\eqref{def:dn} of $d_n$.
We introduce here the definition of $k$-traps, or well-and-wall traps, that is, traps formed by a very big conductance $c_{x-1}$ followed by a very small conductance $c_{x+k}$, making $\rho_x^{(k)} := \e^{-\lambda(k+1)} c_{x-1}/c_{x+k}$ large (we include the case $k=0$). As we shall see, the only $k$-traps that will appear are such that neither $c_{x-1}$ nor $1/c_{x+k}$ are big enough to give rise to a simple-trap (in the sense of Definition \ref{simpletraps}). Recall that we let $q_n := (\log n)^{1/4}$.

\begin{definition}\label{ktraps}
	For $k \geq 0$, a site $x$ is called a \emph{$k$-distant} (well-and-wall) $n$-trap if 
	\[
	\rho_x^{(k)} > d_n \e^{-q_n},\quad  c_{x-1}>\e^{q_n^2}\quad\,  and\, \quad  \frac{1}{c_{x+k}} >\e^{q_n^2} \,.
	\]
	We call  $\cW_x^{(k)}:= \{\rho_x^{(k)} > d_n \e^{-q_n}\} \cap \{ c_{x-1}, \frac{1}{c_{x+k}}>\e^{q_n^2} \}$ the event that $x$ is a $k$-distant trap.
\end{definition}

We now extend Definitions~\ref{def:trapping}-\ref{def:goodtrapping} to the case of $k$-distant traps. We  use the same notations as in the case of simple traps for simplicity.

\begin{definition}
	\label{def:trapping2}
	For a given $n\in\N$, we define a trapping environment as
	\[\gG_n:= \bigcup_{x=0}^{C_n -1}   \bigcup_{k=0}^{C_n-1} \cW_x^{(k)},\]
	that is, an environment with a $k$-distant $n$-trap in the first block.
	A $n$-triblock $\cB_j$ is called a trapping block if $\theta^{jC_n} \go \in \gG_n$. Let
	$\cJ_n  := \{j\in\{0,\dots,k_n\}:\,\theta^{jC_n} \go \in \Gamma_n\}$ be the set of indices of the trapping blocks.
\end{definition}

\begin{definition}\label{def:goodtrapping2}
	A trapping $n$-triblock with a $k$-distant $n$-trap at site $x$ is said to be \emph{good}  if for all $ y \neq x-1,x+k$ in the $n$-triblock,  $c_y \leq \e^{4q_n}$ and $1/c_y\leq \e^{4q_n}$.
	We denote 
	\[
	\bar \gG_n 
	:= \bigcup_{x=0}^{C_n-1}  \bigcup_{k=0}^{ C_n-1}
	\Big( \cW_x^{(k)} 
	\bigcap_{\substack{ -C_n \leq y  < 2C_n \\ y \neq x-1, x+k } } \{ c_y \in [\e^{-4q_n},\e^{4q_n}] \}  \Big) \, .
	\]
\end{definition}
\noindent Note that in the definition $\cW_x^{(k)}$, we  have that $c_{x+k} < \e^{-q_n^2} \leq \e^{4 q_n}$ and $c_{x-1}>\e^{q_n^2} > \e^{-4 q_n}$.

\subsection{Properties of $k$-distant traps}
In this section we collect some important properties of $k$-distant traps and of the new trapping blocks. Besides updating to $k$-distant traps some of the properties of simple traps, we add one new feature (property 4).
\begin{itemize}
	\item[1.]  $k$-distant traps are isolated, in fact distant by at least $n\e^{- 5 q_n}$ (Lemma \ref{disgiunti2});
	\item[2.]  $k$-distant traps are not too deep, they cannot be deeper than $d_n \e^{q_n}$ (Lemma \ref{lem:maxrho2});
	\item[3.] All trapping triblocks are good (Lemma \ref{sogood2});
	\item[4.] $k$-distant traps cannot have $k$ too large, in fact, $k< \frac6\lambda q_n$ (Lemma \ref{k-is-small});
	\item[5.] $k$-distant traps are the only annoying parts of the environment (Lemma~\ref{annoying2}).
\end{itemize}

\begin{lemma}\label{disgiunti2}
	Let $\cD_n$ be the event
	\[
	\cD_n 
	:= \bigcup_{ \substack{-n \leq x<y\leq n \\ |x-y|\leq n \e^{- 5 q_n } } }  \bigcup_{k\geq 0}  \bigcup_{k'\geq 0}  \cW^{(k)}_x \cap \cW_y^{(k')} \, .
	\]
	Then, $\bbP$-a.s.\ $\cD_n$ occurs for only finitely many $n$.
\end{lemma}

\noindent The proof of Lemma \ref{disgiunti2} can be found in Section \ref{proofdisgiunti} of the Appendix.
	As a consequence of Lemma \ref{disgiunti2}, we have that $|\cJ_n| \leq \e^{5 q_n}$ also in the case of $k$-distant traps.

\begin{lemma}\label{lem:maxrho2}
	$\bbP$-a.s., there is some $n_0$ such that for all $n\geq n_0$
	\[ 
	M_n:=\sup_{ -n \leq x \leq n} \sup_{k\geq 0}\rho_x^{(k)} 
	\leq  d_n \e^{ q_n} \, .
	\]
\end{lemma}

\noindent The proof of Lemma \ref{lem:maxrho2} can be found in Section \ref{prooflem:maxrho} of the Appendix.

\begin{lemma}\label{sogood2}
	Let $\cG_n$ be the event
	\[
	\cG_n := \bigcup_{ - n\leq x\leq n} \bigcup_{k\geq 0}  \big\{ \rho_x^{(k)} > d_n\, \e^{-q_n} \big\} \cap  \Big( \bigcup_{ \substack{ y = x- 2C_n \\ y\neq x-1,x+k } }^{x+2 C_n}  \{ c_y \notin [\e^{-4 q_n},\e^{4 q_n}] \} \Big) \, .
	\]
	Then $\bbP$-a.s.\ $\cG_n$ occurs only for finitely many $n$.
\end{lemma}
\noindent The proof of Lemma \ref{sogood2} can be found in Section \ref{proofsogood} of the Appendix.

\begin{lemma}\label{k-is-small}
	Let $\cK_n$ be the event that there exists a $k$-distant trap with $k\ge \frac{6}{\lambda} q_n$, \textit{i.e.}\
	\[
	\cK_n := \bigcup_{-n\leq x\leq n} \bigcup_{k\ge \frac{6}{\lambda} q_n } \cW_{x}^{(k)} \, .
	\]
	Then $\bbP$-a.s., $\cK_n$ occurs for only finitely many $n$.
\end{lemma}
\noindent The proof of Lemma \ref{k-is-small} can be found in Section \ref{proofk-is-small} of the Appendix.

The last lemma shows that, with high probability, there are no ways to have $\rho_x^{(k)}$ close to $d_n$ without having  both $c_{x-1}, \frac{1}{c_{x+k}} > \e^{q_n^2}$: it tells that the traps of depth of order $d_n$ can only be of well-and-wall type.

\begin{lemma}
	\label{annoying2}
Let $\gep_n =q_n^{-\gd}$ for some $\gd>0$, and let
	\[
	\cH_n 
	:= \bigcup_{-n\leq x \leq n}  \bigcup_{k\ge 0}   \{\rho_{x}^{(k)} > \gep_n  \e^{-\lambda k/2} d_n \}\cap (\cW_x^{(k)})^{c}\,.
	\]
	Then, if $\gd$ is sufficiently small, we have that $\bbP(\cH_n)  \to 0$ as $n\to+\infty$.
\end{lemma}

\noindent The proof of Lemma \ref{annoying2} can be found in Section \ref{proofannoying2} of the Appendix.

\subsection{Reduction to large traps}

Also under Assumption \ref{assumptionB} the result of Proposition \ref{blocchetti} is valid:
 with $\gep_n = q_n^{-\gd}$ defined in Lemma~\ref{annoying2}, we have that
\begin{align}
\label{reductiontolargetraps2}
P^\go\Big( \frac{1}{d_n} \big|T_n-\sum_{j\in \cJ_n} T(\cB_j)\big|  > \gep_n^{(1-\ga)/2} \Big) \leq  \gep_n^{(1-\ga)/2}\, ,
\end{align}
with $\bbP$-probability going to $1$.
The proof is identical to the one of Proposition \ref{blocchetti}, and relies on Lemma~\ref{annoying2} in place of Lemma~\ref{annoying}, in order to say that with high probability, all $\rho_z^{(k)}$ outside trapping blocks are smaller than $\gep_n \e^{-\lambda k/2} d_n$ (so that one gets \eqref{eq:expectT}).

\subsection{Crossing of large traps}

As in the case of simple traps, Lemma \ref{disgiunti2} guarantees that to each trapping block $\cB$ can be associated a unique site $x_{\cB}$ and a unique $k_{\cB} \in \{0,\ldots,  \tfrac{6}{\lambda}  q_n \}$ such that $x_{\cB}$ is a $k_{\cB}$-distant $n$-trap (Lemma~\ref{k-is-small} ensures that $k_{\cB} < \tfrac{6}{\lambda} q_n$).   As before we denote by $\rho_{\cB}:= \rho_{x_{\cB}}^{(k_{\cB})}$ the depth of the trap associated with~$x_\cB$. We approximate again $T(\cB)/\rho_{\cB}$ by  $\tau_\cB$, which in the case of well-and-wall traps becomes extremely simple: 
\begin{equation}
\label{def:taubk}
\tau_{\cB} 
= 2\, \mathbf{e}_{\cB} \, ,
\end{equation}
where $\mathbf{e}_\cB \sim \mathrm{Exp}(1)$.
Analogously to the case of simple traps define
\begin{itemize}
	\item[$p_{\cB}$]$:= P_{x_{\cB}+k_\cB+1}^\go(T_{x_{\cB}} > T_{x_{\cB} +k_{\cB} +C_n})$ the probability of escaping to the right of $x_\cB+k_\cB+1\,$;
	\item[$\theta_\cB$] $:=E^{\bar \go}[\Theta_{x_{\cB} + k_{\cB}}]$, where $\Theta_{x_{\cB}+k_{\cB}}$ equal to $1$ plus the time that it takes for $X_j$ to go from $x_\cB+k_\cB-1$ to $x_\cB+k_\cB$.
\end{itemize}
Since we have $c_{x_{\cB}+k_\cB}^{-1}> \e^{q_n^2}$, we get that $p_{\cB}\to 1$ as $n\to +\infty$. On the other hand, also $c_{x_{\cB}-1}>\e^{q_n^2}$, so that the main contribution to $\theta_\cB$  comes from the time spent on the edge $(x_{\cB}-1, x_{\cB})$, which is approximately $2$ if $k_{\cB}=0$ and  $2 c_{x_{\cB}-1}$ if $k_{\cB}>0$.
In \eqref{semplificare}, this roughly corresponds to having $p_\cB =1$ and $\theta_{\cB}=2$, and this gives a heuristic reason why \eqref{def:taubk} holds.

In analogy with Proposition \ref{trappolone}, we claim that there is a constant $c>0$ such that for any $\gd \in(0,1)$,
\begin{align}\label{paolo2}
\sup_{\go\in \bar \gG_n} 
P^\go\Big(\Big|\frac{T(\cB)}{\rho_\cB}-\tau_{\cB}\Big|>\gd  , A_n\Big)
	\leq c \e^{-q_n^2/2}+c \gd^{-2}\e^{-q_n^2/2}\e^{2 \gl k_{\cB}} \leq \gd^{-2} \e^{- q_n^{2}/4}  \, ,
\end{align}
the last inequality holding $\bbP$-a.s., for $n$ large enough, since $\lambda k_{\cB} \leq 6 q_n$ thanks to Lemma~\ref{k-is-small}.
We follow the three steps of the proof of  Proposition \ref{trappolone}. We use the same notations as in Section \ref{sectioncrossinglarge}, re-adapted to the case of $k$-traps. We call  $x=x_\cB$ and $k=k_\cB$. 

\smallskip

\noindent{\bf Step 1.}  We let $T^{(1)}:=T_{x+k+1}-T_{x+k}$ and $t_1=T_{x+k+1}$. The random variable $G:=1+\#\{\ell \, :\,X_{\ell -1}=x+k+1,\,X_\ell=x+k\}$ is a geometric r.v.\ of parameter $p_\cB$, representing the number of attempts to run away from $x+k+1$ and never come back into $x+k$.  When $G\geq 2$, we iteratively define, for $2\leq i\leq G$,
\begin{align*}
t_i
:=\inf\{\ell >t_{i-1}:\,X_{\ell-1}=x+k,\,X_\ell=x+k+1\}\quad\mbox{ and }\quad
T^{(i)}
:=t_i-t_{i-1}\, .
\end{align*}
As in \eqref{decomposino} we rewrite $T(\cB)$ as
$
T(\cB)=T_{x+k} +\sum_{i=1}^{G}T^{(i)} +\bar T(x+k+1,2C_n)\,,
$
where $\bar T(x+k+1,2C_n)$ is the time it takes for $(X_\ell)_{\ell\in\N}$ to go from $x+k+1$ to $2C_n$ conditioned on never returning to $x+k$. By the triangular inequality 
\begin{align}\label{ddd2}
P^{\bar\go}\Big(\Big|\frac{T(\cB)}{\rho_\cB}-\tau_{\cB}\Big|>\gd \Big)
\leq D_1+D_2+D_3\,,
\end{align}
with $D_1$, $D_2$ and $D_3$ that are the obvious homologous of the quantities appearing in \eqref{ddd}.

\smallskip

\noindent{\bf Step 2.} For $D_1$ we use Markov inequality. We notice that in the case $k=0$ we have $E^{\bar\go}[T_{x+k}]\leq c' C_n^2 \e^{8q_n}$, while for $k>0$ we get $E^{\bar\go}[T_{x+k}]\leq c' c_{x-1} C_n  \e^{8q_n}$ (cf.~\eqref{natale}, together with Definition~\ref{def:goodtrapping2}). Hence
\begin{align*}
D_1
\leq 
\begin{cases}
c\, \gd^{-1} \e^{11q_n}\rho_x^{-1} \quad&\mbox{if }k=0 \, ;\\
c\, \gd^{-1} \e^{11q_n} (c_{x-1})^{-1}\e^{\gl k}\leq  c\delta^{-1}\e^{-q_n^2/2} e^{\lambda k} \quad &\mbox{if }k>0 \, .
\end{cases}
\end{align*}

For the term $D_3$, we use the same idea as \eqref{termd3}. For well-and-wall traps we even have, as can be seen from \eqref{pibi2} below, $p_\cB>c>0$, so
\begin{align*}
D_3 
&\leq c' \gd^{-1} \rho_\cB^{-1} E^{\bar\go}[ \bar T_{x+k+1,2C_n}] 
\leq c''\, \gd^{-1} \rho_\cB^{-1} C_n \e^{8q_n}\, .
\end{align*}

\smallskip

\noindent{\bf Step 3.}
It remains to deal with the term $D_2$. Following the reasoning after \eqref{pibi} and remembering that $A_n$ happens a.s.~for $n$ large enough, we see that for well-and-wall traps, using \eqref{formulap}, since $c_{x+k} < \e^{-q_n^{2}}$ and all other conductances are in $[\e^{-4q_n}, \e^{4q_n}]$,
\begin{align}\label{pibi2}
p_\cB
=   P_{x+k+1}^\go(T_{x+k+C_n}<T_{x+k})
\geq \Big(1+c' \e^{-q_n^2/2}\Big)^{-1}\,.
\end{align}
In light of \eqref{pibi2}, it is sufficient to replace the bound  \eqref{denver} by 
\begin{align}
D_2
\leq P^{\bar\go}\Big(\Big|\frac{ T^{(1)}}{\rho_\cB} -2\ \mathbf{e}_{\cB}\Big|>\frac{\gd}{3}\Big)+P^{\bar\go}(G>1)\,.
\label{newdenver}
\end{align}
Note that the second term is equal to $1-p_\cB\leq c' \e^{-q_n^2/2}$.

In order to control the first term in \eqref{newdenver}, we use decomposition \eqref{coniglio} with $\gU_1=0$,   $\cG_1$  a geometric of parameter $1/(1+\rho_{x+k})$ (with $\mathbf{e}_{\cB}$ coupled with $\cG_1$, $\cG_1 =\lceil \mathbf{e}_{\cB} \log \frac{\rho_{x+k}}{1+\rho_{x+k}} \rceil$) and $\{\Theta_{x+k}(j)\}_{j\in\N}$ a collection of i.i.d.~copies of $\Theta_{x+k}$. 
We end up with
\begin{align}\label{ragnarock2}
P^{\bar\go}\Big(\Big|\frac{ T^{(1)}}{\rho_\cB} - 2 \mathbf{e}_{\cB}\Big|>\frac{\gd}{3}\Big)
=  P^{\bar\go}\Big(\Big|\frac{ 1}{\rho_\cB} \sum_{j=1}^{\cG_1-1}\Theta_{x+k}(j)- 2\, \mathbf{e}_{\cB}\Big|>\frac{\gd}{3}\Big)\,.
\end{align}
Then, the triangular inequality gives the following upper bound, analogously to \eqref{pony}, 
\begin{align}\label{pony2}
P^{\bar\go}\Big(\Big|\sum_{j=1}^{\cG_1-1}\Big(\Theta_{x+k}(j)\rho_\cB^{-1}-  2\rho_{x+k}^{-1} \Big)\Big|>\frac \delta{6}\Big) 
+ P^{\bar\go}\Big(2 \Big|\mathbf{e}_\cB-(\cG_1-1)\rho_{x+k}^{-1}\Big|>\frac \delta{6}\Big)\,.
\end{align}
For the second term in \eqref{pony2}, we use that $|\mathbf{e}_\cB - (\cG_1-1)/\rho_{x+k} | \leq (2+\mathbf{e}_\cB)/\rho_{x+k}$, so that the second term is bounded by 
$
P^{\bar\go}( \mathbf{e}_\cB> c' \delta \rho_{x+k} )\leq \exp\{- c' \gd \e^{q_n^2/2 } \}
$
for some constant $c>0$, since $\rho_{x+k}>\e^{q_n^2/2}$ (either $k=0$ and this is obvious, or $k>1$ and $c_{x+k-1} > \e^{-4q_n }$, $1/c_{x+k} >  \e^{q_n^2}$ by definition of a good  block).

We finally deal with the  first term in \eqref{pony2}. We first split it with the triangular inequality into (recall $\theta_{\cB} := E^{\bar \go}[\Theta_{x+k}]$)
\begin{align} \label{pony222}
P^{\bar\go}\Big(\Big|\sum_{j=1}^{\cG_1-1}\Big(\Theta_{x+k}(j)- \theta_{\cB}\Big)\Big|> \frac \delta{12}\rho_{\cB}\Big)
+ P^{\bar\go}\Big(\cG_1\Big|  \theta_{\cB} \rho_\cB^{-1}- 2 \rho_{x+k}^{-1}\Big|>\frac \delta{12}\Big)\,.
\end{align}
These two terms can be controlled thanks to the estimates on the mean and the variance of $\Theta_{x+k}$ given by the following lemma.
\begin{lemma}\label{meanvarlemma2}
	There exists a constant $c'>0$ such that, for $\go\in \bar \Gamma_n$:

	$\ast$  if $k= 0$,
		\begin{align}
	\theta_{\cB} := E^{\bar\go}[\Theta_{x+k}]	&\in \big[ 2 \, ,\, 2 + c' \e^{-q_n^2/2} \big]\label{expectationk1} \\
	Var(\Theta_x) &\leq c' C_n^2 \e^{16 q_n}
	\label{variancek1}\,.
	\end{align}
	
	$\ast$ if $k\geq 1$,
	\begin{align}
	\theta_{\cB} := E^{\bar\go}[\Theta_{x+k}]
	&\in\Big[ 2\frac{c_{x-1}}{c_{x+k-1}}\e^{-\gl k}\,,\, 2\frac{c_{x-1}}{c_{x+k-1}}\e^{-\gl k}+ c' \e^{4q_n}\Big]\label{expectationk2} \\
	Var(\Theta_{x+k})	&\leq	c' C_n^2  \e^{16 q_n} c_{x-1}^2 \label{variancek2}\,.
	\end{align}
\end{lemma}
Let us first bound \eqref{pony222} with the help of this lemma.
In the case $k=0$, we have that $\rho_{\cB}  = \rho_{x+k}$. Since $|\theta_{\cB} -2 | \leq c' \e^{- q_n^2/2}$, we get that the second term in \eqref{pony222} is bounded by $P^{\go} \big( \cG_1 \rho_{\cB}^{-1} > c' \gd \e^{q_n^2/2} \big) \leq \exp\{ - c'' \gd \e^{q_n^2/2} \} \leq c \gd^{-1} \e^{- q_n^2/2}$ (using that $\cG_1$ is a geometric r.v.\ with parameter $1/(1+\rho_{x+k}) > (2\rho_{\cB})^{-1}$).
For the first term in \eqref{pony222}, we get thanks to Chebychev inequality and the variance bound \eqref{variancek1}, as in \eqref{pony1}:
\begin{align*}
P^{\bar\go}\Big(\Big|\sum_{j=1}^{\cG_1-1}\Big(\Theta_{x+k}(j)-   \theta_{\cB} \Big)\Big|> \, \frac \delta{12}\rho_{\cB}\Big) &
 \leq P^{\bar\go}\big(\cG_1>\e^{q_n}\rho_{\cB}\big)+\e^{{q_n}}\rho_{\cB} \frac{Var(\Theta_x)}{\rho_{\cB}^2 (\gd/12)^2}	\\
&\leq \e^{- c\e^{q_n} }+c'\gd^{-2} \e^{20 q_n}  \rho_{\cB}^{-1} \, .
\end{align*}
All together, this gives that $D_2\leq c' \e^{- q_n^2/2}  + c''\gd^{-2} \e^{ - q_n^2/2}$.

In the case $k\geq 1$, recall that $\rho_{\cB}:= \frac{c_{x-1}}{c_{k+1}} \e^{-\lambda (k+1)}$: by \eqref{expectationk2}, we see that $| \theta_{\cB} \rho_\cB^{-1}- 2 \rho_{x+k}^{-1}|\leq c' \e^{4q_n}\rho_\cB^{-1}$. Hence, since $\cG_1$ is a geometric r.v.\ with parameter $1/(1+\rho_{x+k})$, the second term in \eqref{pony222} is bounded by $ \exp\{- c \gd \e^{- 4 q_n} \rho_{\cB} /\rho_{x+k}\} \leq \exp\{- c' \gd \e^{q_n^2/2 } e^{-\lambda k} \} \leq c'' \gd^{-1} \e^{-q_n^2/2} \e^{\lambda k}$, where we used that $\rho_{\cB}/\rho_{x+k} =  \e^{-\lambda k} c_{x-1}/c_{x+k-1}$, with $c_{x-1}> \e^{q_n^2}$ and $c_{x+k-1} \le \e^{4q_n}$.
It remains to bound the first term of \eqref{pony222}, again thanks to Chebychev inequality and  \eqref{variancek2} as above:
\begin{align*}
P^{\bar\go}\Big(\Big|\sum_{j=1}^{\cG_1-1}\Big(\Theta_{x+k}(j)-   \theta_{\cB} \Big)\Big|>\frac \delta{12}\rho_{\cB}\Big)
&\leq P^{\bar\go}\big(\cG_1>\e^{q_n}\rho_{x+k}\big)+\e^{{q_n}}\rho_{x+k}\frac{Var(\Theta_{x+k})}{\rho_{\cB}^2(\gd/12)^2}	\\
&\leq \e^{-\e^{q_n} /2}+c'\gd^{-2}\e^{-q_n^2/2}\e^{2 \gl k}, 
\end{align*}
where we used that $\rho_{x+k} Var(\Theta_{x+k})/\rho_{\cB}^2 \leq c' \e^{20 q_n} e^{2\lambda k} c_{x+k} c_{x+k-1}$, with $c_{x+k} < \e^{- q_n^2}$ and $c_{x+k-1}\leq \e^{4 q_n}$.
Collecting all the previous estimates, we get that 
$
D_2
\leq c\e^{-q_n^2/2} + c''\gd^{-2}\e^{-q_n^2/2}\e^{ 2 \gl k}
$,
which concludes the proof of \eqref{paolo2}.

\begin{proof}[Proof of Lemma \ref{meanvarlemma2}]  We follow the proof of Lemma \ref{meanvarlemma1}. We have, see~\eqref{formulatheta}
	\begin{align*}
	E^{\bar\go}[\Theta_{x+k}]
	=1+ E^{\bar\go}_{x+k-1}[T_{x+k}]
	=2+\frac{2}{c_{x+k-1}}\sum_{\ell=x+k -C_n}^{x+k-2}c_\ell \e^{- \gl(x+k-1 -\ell)}\,.
	\end{align*}  
In the case $k=0$, all $c_{\ell}$ in the sum are smaller than $\e^{4 q_n} \leq \e^{q_n^2/2}$, and $c_{x-1}> \e^{q_n^2}$, which gives \eqref{expectationk1}.
In the case $k\geq 1$, then we separate the term $\ell=x-1$ in the sum, and we use that all $c_\ell<\e^{4q_n}$ for $\ell\neq x-1 $ to obtain \eqref{expectationk2} (we also use that $C_n \leq \e^{q_n}$).

	For the second moment, we bound, analogously to \eqref{polipo}
	\begin{align}\label{polipo2}
	(\Theta_{x+k})^2\leq 2C_n \sum_{j\neq x-1,x} N_j^2
	\end{align}		
	with $N_j$ the random variable counting the number of visits to point $j$ before touching $x+k$, for $j\in\{-C_n,\dots,x+k-1\}$. 
As above, $N_j \leq Y(j)$, where $Y(j)$ is a geometric random variable with parameter $q(j)=P_{j}^{\bar \go}(T_{x+k}<T_j)$. 
Thanks to  \eqref{classicone1}, we get that if $k=0$ then $q(j) \geq c \e^{-8q_n}$ for all $j<x$, by definition of a good block; on the other hand, if $k>0$, we get that $q(j) \geq c \e^{- 8 q_n} / c_{x-1}$ for all $j< x+k$ (one could get a better bound in the case $j\neq x-1,x$, but we do not need it).
Then, using that $E^{\go}_{x+k-1}[N_j^2] \leq 1/q(j)^2$, 
equations \eqref{variancek1} and \eqref{variancek2}  follow by taking the expectation in \eqref{polipo2} (and using that $C_n \leq \e^{q_n}$).
\end{proof}


\subsection{Reduction to a finite number of traps}

First of all, using a union bound and then \eqref{paolo2} (with $\gd=  \frac{d_n \e^{-q_n}}{|\cJ_n| \rho_{\cB_j}} \leq |\cJ_n|^{-1} <1$), we get that for all a.e.\ $\go$, for $n$ large enough
\begin{align*}
P^{\go} \Big(  \Big| \sum_{j\in \cJ_n}T(\cB_j) -\sum_{j\in \cJ_n} 2 \rho_{\cB_j} \mathbf{e}_{\cB_j}  \Big| \geq d_n \e^{- q_n} \Big)
 	&\leq \sum_{j\in \cJ_n} P^{\go} \Big( \big| T(\cB_j) - 2 \rho_{\cB_j} \mathbf{e}_{\cB_j}  \big| \geq \frac{ d_n \e^{-q_n}  }{ |\cJ_n| } \Big) \\
	& \leq c  |\cJ_n|^3  \e^{- q_n^2/4}  \, .
\end{align*}
Then, by Lemma~\ref{disgiunti2}, a.s.\ $|\cJ_n|\leq \e^{5q_n}$ for $n$ large, so this goes to $0$. Hence, together with~\eqref{reductiontolargetraps2}, and using that $\gep_n^{(1-\ga)/2} \geq \e^{-q_n}$ for $n$ large, we have that with $\bbP\otimes P^{\go}$-probability going to $1$, as $n\to+\infty$,
\begin{equation}
\label{reductiontolargetraps3}
 \frac{1}{d_n} \big| T_n - 2 \sum_{j\in \cJ_n} \rho_{\cB_j} \mathbf{e}_{\cB_j} \big| \leq  2 \gep_n^{(1-\ga)/2} \to 0 \, .
\end{equation}

Now, the analogous of Proposition~\ref{prop:fewblocks} holds: for any $\eta,\eta'>0$, there exists $\gep = \gep(\eta, \eta')$ such that 
\eqref{eq:propreduction} holds also in the present case,  \textit{i.e.}\ the main contribution to $\sum_{j\in \cJ_n} \rho_{\cB_j} \mathbf{e}_{\cB_j}$ comes from the blocks with $\rho_{\cB}> \gep d_n$.
The proof  
follows the same scheme as that of Proposition~\ref{prop:fewblocks}, with fewer technicalities since $\tau_{\cB_j}$ is simply replaced by $2\mathbf{e}_{\cB_j}$ (we do not need Lemma~\ref{lem:limittauB}). In particular, Step~1 of the proof is the same (up to \eqref{endofstep1}), but Step~2 is much easier, since $\cH_{\ell}$ is now the sum of $N_{\ell}$ independent exponential random variables. We do not provide the details here, since they are straighforward.

In the end, \eqref{eq:propreduction} combined with \eqref{reductiontolargetraps3} shows that for any fixed $\eta,\eta'>0$, we can choose $\gep$ so  that \eqref{reductiontofew} holds with $\bbP\otimes P^{\go}$-probability at least $1-\eta'$.

\subsection{Convergence}
The analogous of  Proposition~\ref{prop:convergence} holds and is in fact much  simpler, since $\xi_{\cB} =2$ under Assumption~\ref{assumptionB}.
We  simply need to use that for any $\gep>0$, $\{ (\tfrac{i}{k_n} , \tfrac{1}{d_n} \rho_{\cB_i}) ; 1\leq  i\leq k_n, \rho_{\cB_i} > \gep d_n\}$ converges in distribution to a Poisson Point Process $\cP_{\gep}$ on $[0,1]\times \bbR_+$ with intensity $\dd x \ga w^{-(1+\ga)} \ind_{\{w>\gep\}} \dd w$---this is due to the tail behavior of $\rho_0$ and to the definition~\eqref{def:dn} of $d_n$, see the proof of Proposition~\ref{prop:convergence}.
This easily gives that, analogously to \eqref{convergencesumK}, we have the following  under $\bbP\otimes P^{\go}$
\[
2\sum_{i=1}^{k_n} \rho_{\cB_i} \mathbf{e}_{\cB_i} \ind_{\{ \rho_{\cB_i} > \gep d_n\}} \  \Longrightarrow  \  2 \sum_{(x, w, r) \in \bar \cP} w\,  r \, \ind_{\{x\leq 1\}}\ind_{\{ w > \gep \}} \, ,
\]
where $\bar \cP$ is a PPP on $(\bbR_+)^{3}$ with intensity $\dd x \, \ga w^{-(1+\ga)}  \dd w  \e^{-r }\dd r$.
All together, with~\eqref{reductiontolargetraps3}, letting $\gep\downarrow 0$, we get as in \eqref{eq:compoundPoisson} that
\begin{equation}
\label{eq:compoundPoisson2}
\frac{1}{d_n} T_n  \ \Longrightarrow \
  2 \Big(  \frac{\pi \ga}{\sin(\pi \ga)} \Big)^{1/\ga}\,   \sum_{(x, w) \in  \cP} w \ind_{\{x\leq 1\}} \, ,
\end{equation}
with $\cP$ a PPP on $(\bbR_+)^2$ of intensity $\dd x\, \frac{\ga}{\Gamma(1-\ga)} w^{-(1+\ga)} \, \dd w$.
The conclusion of the proof of Theorem~\ref{thm:passagetime}, \textit{i.e.}\ the convergence of the process, is identical to Section~\ref{sec:convproc} from that point on.

\begin{appendix}


\section{Some formulas for resistor networks}
\label{app:formulas}

In this section, we recall some classical formulas for resistor networks, which translate into properties for the hitting times of random walks among random conductances.
The first important identity is the following: for $ i<x<j$, we have
\begin{align}\label{classicone}
P_x^{\go}(T_i<T_j)=\frac{\Ceff(\{x\}\leftrightarrow \{i\})}{\Ceff(\{x\}\leftrightarrow \{i,j\})}\, .
\end{align}
The effective conductance $C_{\rm eff}$ is here (see \cite{LP16} for the general definition)
\begin{align}
C_{\rm eff}(\{x\}\leftrightarrow\{i\})&=S(i,x-1)^{-1}\label{effettivamente1}\\
C_{\rm eff}(\{x\}\leftrightarrow\{i, j\})&=S(i,x-1)^{-1}+S(x,j-1)^{-1}\,,\label{effettivamente2}
\end{align}
with (recall the definition \eqref{def:rhok} of $\rho_x^{(k)}:= e^{-\lambda(k+1)} c_{x-1}/c_{x+k}$)
\begin{equation}
\label{Sij}
S(i,j):=\sum_{\ell=i}^j\frac{1}{c_\ell^{\lambda}} = \frac{1}{c_{i-1}^{\lambda}}  \sum_{k=0}^{j-i} \rho_{i}^{(k)} . 
\end{equation}
For the first formula we have used that we have conductances in series, while for the second formula we have two sequences of conductances-in-series that are in parallel. Then, \eqref{effettivamente1} and \eqref{effettivamente2} together with \eqref{classicone} give
\begin{align}\label{classicone1}
P_x^{\go}(T_i<T_j)
	=\frac{S(x,j-1)}{S(i,j-1)}\,.
\end{align}

This apply for instance to the probability $p_{\cB}$
appearing in Section~\ref{sec:crossingontraps}:
we have that
\begin{equation}
\label{formulap}
p_x 
	:= P_{x+1}^{\go}(T_{x+ C_n} > T_{x})
	= \Big( 1+ c_{x} \sum_{j=1}^{C_n} \frac{1}{c_{x+j}}  \e^{-\lambda(j+1)}\Big)^{-1}  \, .
\end{equation}

Another important identity we use throughout the paper deals with the expectation of the hitting times.
We use the following representation, cf.~\cite[Eq.~(3.22)]{B06}:
\begin{align}\label{magicformula}
E_x^{\go} [T_y]&=\frac{1}{C_{\rm eff}(\{x\}\leftrightarrow\{y\})}\sum_{z<y}\pi(z)P^\go_z(T_x<T_y)
\end{align}
where $\pi(z):=c_{z-1}^\gl+c_z^\gl$  is a reversible measure for $(X_n)_{n\in\N}$.
We notice that in \eqref{magicformula} the quantity $P^\go_z(T_x<T_y)$ is equal to $1$ for $z\leq x$, while for $x<z<y$ we can use \eqref{classicone1}.
We therefore get that, after calculation, see \cite[Eq.~(2.10)]{BS17}, for $y>x$
\begin{align}\label{natale}
E_x^{\go}[T_y] &= \sum_{z\leq x} (c_{z}^{\lambda}+c_{z-1}^{\lambda}) S(x,y-1)  + \sum_{x<z<y} (c_{z}^{\lambda}+c_{z-1}^{\lambda}) S(z,y-1) \notag\\
& = 2\sum_{z \leq x-1} c_{z}^{\lambda} S(x,y-1) + \sum_{x\leq z <y} c_z^{\lambda} S(z,y-1) +  \sum_{x <  z < y} c_{z-1}^{\lambda} S(z,y-1)  \notag \\
& =(y-x) + 2  \sum_{z \leq x} c_{z-1}^{\lambda}  S(x,y-1) +  2 \sum_{x <  z < y} c_{z-1}^{\lambda} S(z,y-1) 
\end{align}
(we used that $c_z S(z,y-1) = 1+ c_z S(z+1,y)$, and that $S(y,y-1)=0$).
We can finally rewrite this as follows:
\begin{align}
E_x^{\go}[T_y] = (y-x) + 2 \sum_{z< x} \sum_{k=x-z}^{y-z} \rho_z^{(k)}  + 2 \sum_{ z=x}^y \sum_{k=0}^{y-z} \rho_z^{(k)} \, .
\label{realmagic}
\end{align}
This applies for instance to quantity $\theta_x = 1+ E^{\go}_{x-1}[T_x] $ defined in Section~\ref{sec:crossingontraps}: 
\begin{equation}
\label{formulatheta}
\theta_x 
	=  1+E^{\bar \go}_{x-1}[T_x] 
	= 2+ \frac{2}{c_{x-1}} \sum_{\ell =x-C_n}^{x-2} c_\ell \e^{-\lambda (x-1-\ell)} \, .
\end{equation}
We also need a formula for the expectation of hitting times for a random walk killed at some site $z$. It is not explicitly stated in \cite{B06}, but can be obtained in the same way as \eqref{magicformula} (as mentioned in the sentence before \cite[eq.~(3.21)]{B06}):
for $y<x<v$, we have
\begin{align}\label{magicformula2}
E_x^{\go} [T_{y} \1{T_y<T_v}]
	=\frac{1}{C_{\rm eff}(\{x\}\leftrightarrow \{y,v\})}\sum_{y < z< v}\pi(z) P^\go_z \big(T_x< T_{y} \wedge T_v \big) P^\go_z( T_y < T_v)\,.
\end{align}


\section{Estimates conditionally on having a large trap}
\label{sec:rho}

Let us recall Proposition~\ref{prop:rho0}, which derives from \cite[Corollary 5]{C86}. It gives the following sharp asymptotics:  under Assumption~\ref{assumptionA} we have as $t\to+\infty$
\begin{equation}
\label{rhosimple}
\bbP( \rho_0 >t) \e^{\lambda\ga} \sim \bbE[c_0^{\ga}]  L_0(t) t^{-\ga_0} \ind_{\{\bbE[c_0^{\ga}] <+\infty \}} +\bbE[1/c_0^{\ga}]  L_\infty(t) t^{-\ga_{\infty}} \ind_{\{\bbE[1/c_0^{\ga}] <+\infty \}} ;
\end{equation}
under Assumption~\ref{assumptionB} we have as $t\to+\infty$
 \begin{equation}
 \label{rhowellandwall}
\bbP( \rho_0 >t) \e^{\lambda \ga} \sim  \ga \frac{\Gamma (1+\gamma_0) \Gamma(1+\gamma_{\infty})}{ \Gamma(1+\gamma_0+\gamma_{\infty})}   \, (\log t)   L_0(t)L_{\infty}(t) t^{-\ga} \, .
\end{equation}

\subsection{Proof of Propositions~\ref{prop:conditional1}}


Recall that when $\bbE[c_{0}^{\ga}]<+\infty$, we define $\bar c_{-1}$ a random variable with c.d.f.\ $F_{\bar c_{-1}}(u) = \frac{1}{\bbE[c_{-1}^{\ga}]} \bbE[c_{-1}^{\ga} \ind_{c_{-1}\leq u}]$, and similarly for $1/\bar c_0$.
We consider two cases separately: 1) $\bbE[c_{0}^{\ga}]  < +\infty$, $\bbE[1/c_0^{\ga}]  = +\infty$ (the case $\bbE[c_{0}^{\ga}] = +\infty$, $\bbE[1/c_0^{\ga}] <+\infty$ is symmetric); 2) the case $\bbE[c_0^{\ga}] ,\bbE[1/c_0^{\ga}] =+\infty$.

{\bf 1)} Suppose that $\bbE[c_0^{\ga}]  < +\infty$, $\bbE[1/c_0^{\ga}]  = +\infty$, so the second term in \eqref{rhosimple} is equal to~$0$.
Then for any fixed $0<a<b$ and $v>0$ we have as $t\to+\infty$
\begin{align*}
\bbP( c_{-1} \in (a,b]  ,1/c_0>v , \rho_0>t) &= \bbE \big[ \bbP( c_{-1}/c_0 > \e^{\lambda} t \mid c_{-1}) \ind_{\{c_{-1} \in (a,b]\}} \big] \\
	& = (1+o(1)) L_{0}(t) \e^{-\lambda \ga_0} t^{-\ga_{0}} \bbE[c_{-1}^{\ga} \ind_{\{ c_{-1} \in (a,b]\}}] \, .
\end{align*}
Here we used that $\bbP(c_{-1}/c_0 > \e^{\lambda } t/x) = (1+o(1)) L_{0}(t) \e^{-\lambda \ga_0} (t/x)^{-\ga_{0}}$ as $t\to\infty$, uniformly for $x \in (a,b]$.
Letting $b\to \infty$, we  get that $\bbP( c_{-1} >a  , 1/c_0>v , \rho_0 >t) = (1+o(1)) L_{0}(t) \e^{-\lambda \ga_0} t^{-\ga_{0}} \bbE[c_{-1}^{\ga} \ind_{\{ c_{-1} >a \}}] $.

As a consequence, in view of \eqref{rhosimple}, we have that
\[ \bbP\big( c_{-1} >a , 1/c_0>v \mid  \rho_0>t\big) \xrightarrow{t\to+\infty}  \frac{1}{\bbE[c_{-1}^{\ga}]}\bbE[c_{-1}^{\ga} \ind_{\{c_{-1} > a\}}] = 1- F_{\bar c_{-1}}(a) \, , \]
giving that $(c_{-1}, 1/c_0)$, conditionally on $\rho_0>t$, converges in distribution to $ (\bar c_{-1}, +\infty)$ (\textit{i.e.}\ $B=0$ in \eqref{limitconditional}).

{\bf 2)} If $\bbE[c_0^{\ga}]<+\infty$, $\bbE[1/c_0^{\ga}] <+\infty$ (necessarily $\ga=\ga_0=\ga_{\infty}$), then both $\bar c_{-1}$ and $1/\bar c_0$ are well defined, and both terms in \eqref{rhosimple} are non-null.
The same reasoning as above is still valid: 
for any fixed $0<a<b$ and $v>0$, and any fixed $0<c<d$ and $v'>0$, we get as $t\to+\infty$
\begin{align*}
\bbP\big( c_{-1}\in (a,b] , 1/c_0>v , \rho_0>t\big) 
	& =(1+o(1)) L_0(t) \e^{-\lambda \ga_0} t^{-\ga_0} \bbE[c_{-1}^{\ga} \ind_{\{c_{-1}\in (a,b]\}}] \, ;\\
\bbP\big( c_{-1}>v' , 1/c_0 \in (c,d] , \rho_0>t\big) 
	& =(1+o(1)) L_{\infty}(t) \e^{-\lambda \ga_{\infty}}  t^{-\ga_{\infty}} \bbE[1/c_0^{\ga} \ind_{\{1/c_0\in (c,d]\}}]\, .
\end{align*}
As a consequence, we get that for any $a,c>0$, as $t\to+\infty$
\[
\bbP\big( c_{-1} >a , 1/c_0>c , \rho_0>t\big) \sim  L_0(t) \e^{-\lambda \ga} t^{-\ga} \bbE[c_{-1}^{\ga} \ind_{\{c_{-1}>a\}}]+ L_{\infty}(t) \e^{-\lambda \ga}  t^{-\ga} \bbE[1/c_0^{\ga} \ind_{\{1/c_0>c\}}]\, .
\]
In view of \eqref{rhosimple}, we get that
\begin{align*}
\bbP\big( c_{-1} >a , 1/c_0>c \mid \rho_0>t\big) \xrightarrow{t\to+\infty}  q \,(1-F_{\bar c_{-1}}(a)) + (1-q)\, (1-F_{1/\bar c_0}(c)) \, ,
\end{align*}
where $q = \lim_{t\to +\infty} \frac{\bbE[c_0^{\ga}]  L_0(t) }{\bbE[c_0^{\ga}]  L_0(t)+ \bbE[1/c_0^{\ga}]  L_\infty(t) } =\lim_{c\to +\infty} \lim_{t\to+\infty} \bbP( 1/c_0>c \mid \rho_0>t)$. This concludes the proof.

\subsection{Proof of Propositions~\ref{prop:conditional2}}
Under Assumption~\ref{assumptionB}, as above, we get that for any fixed $b,d>0$, as $t\to+\infty$
\begin{align*}
\bbP\big( c_{-1} \leq a , \rho_0>t\big) & =(1+o(1))  L_0(t)  \e^{-\lambda \ga_0} t^{-\ga_0} \bbE[c_{-1}^{\ga_0} \ind_{\{c_{-1} \leq b\}}] \, ;\\
\bbP\big(  1/c_0 \leq  d , \rho_0>t\big) & =(1+o(1))  L_{\infty}(t)  \e^{-\lambda \ga_\infty} t^{-\ga_{\infty}}  \bbE[1/c_0^{\ga_\infty} \ind_{\{1/c_0 \leq d\}}]\, .
\end{align*}
Since $L_0(t) = \gp_0(\log t) (\log t )^{\gamma_0}$,  $L_{\infty}(t) = \gp_{\infty}(\log t) (\log t)^{\gamma_{\infty}} $ with$\gamma_\infty,\gamma_0>-1$, we have that  $L_0(t), L_{\infty}(t)$ are both negligible compared to $(\log t) L_0(t) L_{\infty}(t)$.
In view of \eqref{rhowellandwall}, for any fixed $b,d>0$, we get
$
\bbP \big( c_{-1}\leq b \text{ or } 1/c_0\leq d \mid \rho_0>t\big) \to  0\, ,$
which concludes the proof.


\subsection{Distribution of $\tau_{\cB}$ conditionally on having $\rho_{\cB}$ large}
\label{sec:appb2}
In this section, we prove Lemma~\ref{lem:tauB} and \ref{lem:limittauB}: we consider the case of \ref{assumptionA}.
For simplicity, we assume that $\rho_{\cB} = \rho_0$ (\textit{i.e.}\ $x_{\cB}=0$). We recall from \eqref{app:tau} that $\tau_{\cB}$ is equal to
$
\tau_{\cB} := \xi_{\cB} \, \mathbf{e}_{\cB}=\frac{\theta_{\cB}}{p_{\cB}} \, \mathbf{e}_{\cB}
$
with $\mathbf{e}_{\cB}\sim \mathrm{Exp}(1)$, and with 
\begin{equation}
\label{p0}
p_{\cB} = p_0  =  P_{1}^{\go}(\tau_0>\tau_{C_n}) =  \Big( 1+  c_{0} V^{(n)} \Big)^{-1} \qquad \text{with }  V^{(n)}:= \sum_{j = 1}^{C_n} \frac{1}{c_j} \e^{-\lambda (j+1)} \, ,
\end{equation} 
\begin{equation}
\theta_{\cB} =\theta_0 = 1+E_{0}^{\bar \go}[ T_{1}] =  2\Big(  1+ \frac{1}{c_{-1}} W^{(n)}  \Big) \qquad \;\,\quad\text{with } W^{(n)} := \sum_{j = 2}^{C_n} c_{-j} \e^{- \lambda(j+1)} \, .
\label{theta0}
\end{equation}
We refer to \eqref{formulap}, \eqref{formulatheta} for the formulas.

Let us stress right away that $\bbP(V^{(n)}>t)$ is bounded by a constant times $\bbP(1/c_0 >t)$. Indeed, denoting $C_{\lambda}:= \sum_{j\geq 1} \e^{-\lambda(j+1)}$, a union bound gives  that $\bbP(V^{(n)}>t)$ is bounded by
\begin{align}
 \sum_{j\geq 1}  \bbP \Big( \frac{1}{c_j} \geq  \frac{\e^{\lambda(j+1)} t}{C_{\lambda}}  \Big)  \leq \sum_{j \geq 1} c'_{\lambda} \e^{-\frac{\ga}{2}\lambda(j+1)}  \bbP(1/c_j >t)
=c \bbP(1/c_{0} >t) \, ,
\label{tailV}
\end{align}
where we also used Potter's bound. Similarly, $\bbP(W^{(n)}>t) \leq c\,  \bbP(c_{-1} >t)$.

Before we prove Lemma~\ref{lem:tauB}, we prove the following result, which deals with the tail distribution of $p_{0},\theta_0$ conditionally on having $\rho_0$ large.

\begin{lemma}
\label{lem:ptheta}
For any $\tilde \gd , \bar \gd >0$ fixed small enough, there is a constant $c>0$ such that for any $1\leq t\leq n^{2 \bar \gd}$ and any $A_n \geq n^{\frac1\ga - \bar \gd}$,
\begin{align}
\label{forp}
\bbP \big(1/p_{0} > t \mid  \rho_0 \geq A_n \big) & \leq  c t^{-2\ga_{\infty} +\tilde \gd} + c   \bbP \big(1/c_0 >t  \big)  f_{\infty}(A_n) \, ,\\
\bbP\big(\theta_0 > t \mid  \rho_{0} \geq A_n \big) &\leq c t^{-2\ga_{0} + \tilde \gd} + c   \bbP \big(c_{-1} >t  \big) f_0(A_n) \, ,
\label{fortheta}
\end{align}
with $f_{\infty}(A_n):=  \bbP \big( c_{-1} >A_n \big)  /  \bbP(\rho_0 >A_n) $ and $f_0(A_n):=\bbP \big( 1/c_0 >A_n \big)  /  \bbP(\rho_0 >A_n) $.
\end{lemma}
Notice that, as a function of $t$, these bounds are regularly varying.

\begin{proof}
We only treat \eqref{forp}, the other bound \eqref{fortheta} being similar.
Using \eqref{p0}, we write  
\begin{align*}
 \bbP \big(  1/p_0 > t  \mid \rho_0>A_n  \big) = 
 	\bbP \big(  c_0 V^{(n)} > t/2  , \rho_0  >A_n  \big) \,/\, \bbP(\rho_0 >A_n)\, .
\end{align*}
We  split $\bbP (  c_0 V^{(n)} > t/2  , \rho_0 >A_n  )$ into four parts (recall that $\rho_0 = \e^{-\lambda} c_{-1}/c_0$)
\begin{align}
 \bbP  \big( & c_0 V^{(n)} > t/2  , \rho_0 >A_n  , c_0 <1/A_n \big) 
  + \bbP  \big(  c_0 V^{(n)} > t/2  , \rho_0 >A_n  , c_0 \in [A_n^{-1},1] \big) \notag \\
&+ \bbP  \big(  c_0 V^{(n)} > t/2  , \rho_0 >A_n  , c_0 \in [1,t] \big) 
 + \bbP  \big(  c_0 V^{(n)} > t/2  , \rho_0 >A_n  , c_0  > t \big)
 \label{fourterms}
\end{align}

\textbullet\ The first term is bounded by $\bbP ( V^{(n)}> t A_n/2) \bbP( 1/c_0 >A_n  )$, so that, recalling \eqref{tailV}, it is bounded by a constant times $A_n^{- 2\ga_0 + \tilde \gd}$, which is itself bounded by a constant times $A_n^{- \ga_0+ 2 \tilde \gd } \bbP(\rho_0 >A_n)$ (we have $\ga_0\geq \ga$).

\textbullet\  The last term in \eqref{fourterms} is bounded by $\bbP(c_{-1} >t A_n) \bbP(c_{0} >t) \leq c t^{-2\ga_{\infty}+\tilde \gd} \bbP(c_{-1}>A_n)$, the inequality coming from Potter's bound. Note also that   $ \bbP(\rho_0 >A_n) \geq c \bbP( c_{-1} > A_n) $.

\textbullet\  Using \eqref{tailV} and Potter's bound, we get that the second term in \eqref{fourterms} is bounded by
\begin{align}
\sum_{l=0}^{\log_2(A_n) -1}
\bbP  \big( V^{(n)}> 2^{l-1} t \big) &\bbP(c_{-1} >2^{-(l +1)} A_n) \bbP \big( c_0\in [2^{-(l+1)},2^{-l} ] \big) \notag\\
& \leq c\sum_{l=0}^{\log_2(A_n)} (2^{l})^{-2\ga_{0} +\ga_{\infty} + \tilde\gd} \bbP \big( 1/c_0 > t \big) \bbP(c_{-1} > A_n) \, .
\label{secondterminfourterms}
\end{align}

\textbullet\ Finally, the third term  in \eqref{fourterms} is bounded by
\begin{align*}
\sum_{l=0}^{\log_2(t)-1} \bbP  \big( V^{(n)}> 2^{-l-2} t \big)& \bbP(c_{-1} >2^l A_n) \bbP  \big (c_0\in [2^l,2^{l+1}] \big ) \notag \\
& \leq c\sum_{k=0}^{\log_2(t)} (2^{l})^{\ga_{0} -2\ga_{\infty} + \tilde \gd } \bbP \big( 1/c_0 > t \big) \bbP(c_{-1} > A_n) \, .
\end{align*}

Now, if $\ga_{\infty} < 2 \ga_0 -\tilde \gd$ and $\ga_0< 2\ga_{\infty}-\tilde \gd$ (which is the case for instance if $\ga_0=\ga_{\infty} =\ga$  and $\tilde \gd$ small), then both sums over $k$ are finite: we get that the second and third term in \eqref{fourterms} are bounded by $\bbP \big( 1/c_0 > t \big) \bbP(c_{-1} > A_n)$.
If on the other hand $\ga_{\infty} \geq  2 \ga_0 - \tilde \gd >\ga_0$ (in which case $\ga_0< 2\ga_{\infty}- \tilde \gd$), then we get that $\sum_{l=0}^{\log_2(A_n)} (2^{l})^{ -2\ga_{0}+\ga_{\infty} +\tilde \gd} \leq  c A_n^{ \ga_{\infty} - 2\ga_0 + \tilde \gd }$, and \eqref{secondterminfourterms} is bounded by a constant times $A_n^{ -2\ga_0+2\tilde \gd} \leq c A_n^{-\ga_0 +3\tilde \gd} \bbP(\rho_0 >A_n)$ (recall that if  $\ga_0<\ga_{\infty}$ then  $\bbP(\rho_0 >A_n) \sim c' \bbP(1/c_0>A_n)$).
Finally, if $\ga_0 \geq 2\ga_{\infty} - \tilde \gd$, we have that  $\sum_{l=0}^{\log_2(t)} (2^{l})^{\ga_{0} -2\ga_{\infty} +\tilde \gd} \leq  c t^{\ga_0 - 2\ga_{\infty}+\tilde \gd}$, and then the third term in \eqref{fourterms} is bounded by a constant times $t^{-2\ga_{\infty} +2\tilde \gd} \bbP(c_{-1} >A_n)$, with $\bbP(c_{-1}>A_n) \sim c \bbP(\rho_0 >A_n)$, since $\ga_0>\ga_{\infty}$.

All together, we have bounded the four terms in \eqref{fourterms}, so that
\begin{equation}
 \bbP \big(  1/p_0 > t  \mid \rho_0>A_n  \big)  \leq c A_n^{-\ga_0 +3\tilde \gd}  + c t^{- 2\ga_{\infty} +2\tilde \gd} + \bbP \big( 1/c_0 > t \big) \frac{\bbP(c_{-1} > A_n) }{ \bbP(\rho_0 >A_n)}\, .
 \label{conclusionB1}
\end{equation}
We then get \eqref{forp} by bounding $A_n^{-\ga_0 +3\tilde \gd}$ by $n^{-1+ 4\tilde \gd}$ (recall $A_n \geq n^{\frac{1}{\ga} -\bar \gd}$), and using that $t\leq n^{2 \bar \gd}$: we get that the first term in \eqref{conclusionB1} is negligible.
\end{proof}

\begin{proof}[Proof of Lemma~\ref{lem:tauB}]
Let us write $\bbP_{A_n}(\cdot)$ for $\bbP(\cdot \mid \rho_0 > A_n)$ for simplicity, and $\xi_{0} = \theta_0/p_0$.
Recall \eqref{app:tau}. We decompose the probability according to whether $\theta_0 \leq 4$ or $1/p_0 \leq 4$ (or neither), and we write
\begin{align*}
\bbP_{A_n} \Big( \xi_0  > t \Big)  \leq 
\bbP_{A_n} \Big( \tfrac{1}{p_0}  > t /4 \Big) 
+ \bbP_{A_n} \Big( \theta_0  > t/4 \Big) 
+  \bbP_{A_n} \Big( \theta_0 >4, \tfrac{1}{p_0}>4 \Big)\, .
\end{align*}
For the last term, recalling formulas~\eqref{p0}-\eqref{theta0}, we get that
\begin{align*}
\bbP_{A_n}\big(  \theta_0 >4, 1/p_0 < 4 \big) 
&=  \bbP_{A_n}\big(  \frac{1}{c_{-1}}  W^{(n)}>1 , c_0 V >3 \big)
\leq  \bbP_{A_n} \big(  V^{(n)} W^{(n)} > 3 \rho_0 \big)\\
& \leq \bbP\big(  V^{(n)} W^{(n)} >  A_n \big) \, .
\end{align*}
Now, one can easily get that $\bbP(V^{(n)} >t) =  t^{-\ga_0+o(1)}$ and $\bbP(W^{(n)}>t) =   t^{-\ga_{\infty}+o(1)}$ as $t\to +\infty$ (uniformly in $n$, see \eqref{tailV}), and hence that $\bbP(V^{(n)}W^{(n)} >t) = t^{-\ga+o(1)}$ (see \cite[Lemma~1.3]{BS17}). 
Therefore, we get that $\bbP (  V^{(n)} W^{(n)} >  A_n )\leq n^{-1+2 \bar \gd}$, for $n$ large enough and $\bar \gd$ small enough (recall $\ga<1$ and $A_n\geq n^{\frac1\ga -\tilde\gd}$). Hence
\begin{equation*}
\bbP_{A_n} \big( \xi_0 > t \big)\leq   \bbP_{A_n}(1/p_0 > t /4 ) + \bbP_{A_n}(\theta_0 > t /4 ) +  n^{-1+2 \bar \gd} \, ,
\end{equation*}
which together with Lemma~\ref{lem:ptheta} concludes the proof of the first part of the lemma (note that $n^{-1 +2\bar \gd}$ is negligible compared to $t^{-2\ga + \tilde \gd}$, uniformly over $t\le n^{\bar \gd}$, provided $\bar \gd$ is small).

For the second part of the lemma, recall that $\tau_{\cB} := \xi_{\cB} \, \mathbf{e}_{\cB}$, with $\mathbf{e}_{\cB} \sim \mathrm{Exp}(1)$ independent of $\xi_{\cB}$: we therefore get, conditioning first on $\xi_{\cB}$ 
\begin{align*}
\bbP \otimes P^{\go}( \xi_{\cB}\, \mathbf{e}_{\cB} >t \mid \rho_{\cB} > A_n )  = \bbE \big[ \e^{- t/\xi_{\cB}} \mid  \rho_{\cB} > A_n \big]
\leq \e^{-t} + \bbP_{A_n} \big( \xi_0 > t \big),
\end{align*}
which concludes the proof.
\end{proof}

\begin{proof}[Proof of Lemma~\ref{lem:limittauB}]
The proof follows essentially from Proposition~\ref{prop:conditional1}. Again, for simplicity of notations, we reduce to the case where $\rho_{\cB}=\rho_0$, \textit{i.e.}\ $x_{\cB} =0$.
Recall the definition of $\xi_{\cB}$, and the formulas \eqref{p0}-\eqref{theta0} for $p_{\cB}, \theta_{\cB}$,
\begin{equation}
\label{eq:writingtauB}
\xi_{\cB}:=\frac{\theta_{\cB}}{ p_{\cB}}= 2  \Big( 1+ c_0 V^{(n)}\Big) \Big(1+ \frac{1}{c_{-1}} W^{(n)}\Big)   \, .
\end{equation}
Notice that $V^{(n)}, W^{(n)}$ are independent of $c_{-1}, c_0$, hence of $\rho_0$. We clearly have that $V^{(n)}, W^{(n)} \to V, W$ as $n \to +\infty$, by monotone convergence.
We  now simply use Proposition~\ref{prop:conditional1}: it gives that conditionally on $\rho_0>t$, $( \frac{1}{c_{-1}}, c_0 )$ converges in distribution as $t\to+\infty$ to
\begin{enumerate}
\item[1.] $(\frac{1}{\bar c_{-1}}, 0)$ if $\bbE[c_{-1}^{\ga}]<+\infty$, $\bbE[1/c_0^{\ga}] =+\infty$;
\item[2.] $(0, \bar c_0)$ if $\bbE[c_{-1}^{\ga}]=+\infty$, $\bbE[1/c_0^{\ga}] <+\infty$;
\item[3.] $((1-B) \frac{1}{\bar c_{-1}} , B\, \bar c_0)$ if $\bbE[c_{-1}^{\ga}]<+\infty$, $\bbE[1/c_0^{\ga}] <+\infty$.
\end{enumerate}
The distributions of $\bar c_{-1}, \bar c_0$ and $B\sim \mathrm{Bern}(q)$ are those given in the statement of Lemma~\ref{lem:limittauB}.
Hence, in view of \eqref{eq:writingtauB}, we get that conditionally on $\rho_0>n$, $\xi_{\cB}$ converges in distribution as $n\to+\infty$ to 
\[
\zeta:= 2 \Big( 1+ B\, \bar c_0 V\Big) \Big(1+ (1-B)\tfrac{1}{\bar c_{-1}} W \Big)  = 2 \Big(1+ B\, \bar c_0 V+ (1-B)\tfrac{1}{\bar c_{-1}} W \Big) \, ,
\]
with $B=1$ if $\bbE[c_{-1}^{\ga}]<+\infty$, $\bbE[1/c_0^{\ga}] =+\infty$ and $B=0$ if $\bbE[c_{-1}^{\ga}]=+\infty$, $\bbE[1/c_0^{\ga}] <+\infty$.

Moreover, in view of Lemma~\ref{lem:tauB} (see in particular \eqref{limitingtaub}), we get that there is a constant $c>0$ such that for any $t>1$,
\begin{equation*}
\bbP( \zeta  >t) \leq c L_0(t) t^{-\ga_0} \1{\bbE[1/c_0^{\ga}]<+\infty} + c L_{\infty}(y)t^{-\ga_{\infty}} \1{\bbE[c_{-1}^{\ga}]<+\infty} \, .
\end{equation*}
This implies in particular that $\bbE[(\zeta)^{\ga}]<+\infty$.
\end{proof}


\section{Trap properties: proofs of the lemmas}\label{dimostrazionitrappole}
In this section we collect the proofs of some technical properties of both simple and $k$-distant traps. Often the proofs deal with both cases at once.
We recall that $\rho_x^{(k)} = \e^{-\lambda(k+1)} \frac{c_{x-1}}{c_{x+k}} =   \rho_{x} \cdots \rho_{x+k}\,$, and $\rho_x = \rho_x^{(0)}$.

\subsection{Traps are isolated: Proof of Lemmas	\ref{disgiunti} and \ref{disgiunti2}}\label{proofdisgiunti}

	The case $k,k'=0$ will include both the case of simple traps of Lemma~\ref{disgiunti} and the case of $0$-distant well-and-wall traps of Lemma \ref{disgiunti2}.
We start by estimating the probability $\bbP(\cW_x^{(k)} \cap \cW_y^{(k')}   )$, for all $ - n \leq x <y \leq n$ and $ k,k' \geq 0$. 
Note that $\cW_x^{(k)} \subset \{\rho_x^{(k)} > d_n \e^{-q_n}\}$.
	
	\textbullet\ First, if $y -1\neq x+k$ and $y +k'\neq x+k$, then $\rho_x^{(k)}$ and $\rho_y^{(k')}$ are independent, so that by \eqref{asymprhok} (recall also the definition~\eqref{def:dn} of $d_n$)
	\begin{align*}
	\bbP\big( \rho_x^{(k)}, \rho_y^{(k')} > d_n \e^{-q_n}\big) 
	& \leq  \bbP( \rho_x^{(k)} > d_n \e^{-q_n} ) \bbP(\rho_y^{(k')} > d_n \e^{-q_n})  \\
	& \leq  c \,\e^{- \lambda \ga k/2} \e^{- \lambda \ga  k'/2}\psi(d_n)^2 d_n^{-2\ga} \e^{ 4 \ga q_n} \leq  \frac{c'}{n^2} \e^{4  q_n }  \e^{- \lambda \ga (k+k')/2} \, .
	\end{align*}
Here, we used Potter's bound to get that $\psi(d_n \e^{-q_n}) \leq c \psi(d_n) \e^{\ga q_n}$ (and the fact that $\ga\leq 1$).

	\textbullet\ If $y=x+k+1$, then we have that $\rho_x^{(k)} \rho_y^{(k')} = \rho_x^{(k+k'+1)}$, and
	\begin{align*}
	\bbP\big( \rho_x^{(k)}, \rho_y^{(k')} > d_n e^{-q_n} \big) 
	& \leq \bbP\big(  \rho_x^{(k+k'+1)} > d_n^2 \e^{- 2q_n} \big) \\
	& \leq  c\, \e^{-\lambda \ga (k+k'+1)/2} \psi(d_n^2 \e^{- 2q_n }) d_n^{-2\ga} \e^{2\ga q_n} \leq  c'\, \e^{-\lambda \ga (k+k')/2}  n^{-3/2} \, .
	\end{align*}
	The last inequality uses once again Potter's bound (the expression is regularly varying in $n$, with index $-2$).

	\textbullet\ If $ y +k' = x+k$, then necessarily $k>0$. Note that $c_{y-1}$ is independent of $\rho_x^{(k)}$, so recalling Definition~\ref{ktraps},  $\bbP(\cW_x^{(k)} \cap \cW_y^{(k')}   )$ is bounded by
	\begin{align*}
	\bbP\big( \rho_x^{(k)} > d_n e^{-q_n}  ; c_{y-1} > \e^{q_n^2} \big) &\leq c \e^{- \lambda k/2} \psi(d_n \e^{-q_n}) d_n^{-\ga} \e^{\ga q_n} e^{- \ga_{\infty} q_n^{2} /2} \\
	& \leq \frac{c}{n} \e^{- \ga_{\infty} q_n^2/4}  \e^{ - \lambda k/2}  \, .
	\end{align*}
For the last inequality, we used Potter's bound and the definition \eqref{def:dn} of $d_n$, and then took $n$ large enough.

	\smallskip
	Therefore, by a union bound, we have that 
	\begin{align*}
	\bbP  (\cD_n) 
	\leq  C n^2 \e^{- 5 q_n} &\sum_{k,k'\geq 0} \frac{1}{n^2} \e^{4  q_n }  \e^{- \lambda \ga (k+k')/2}\\
	& + c'  n \sum_{k,k'=0}^{+\infty} \e^{-\lambda \ga (k+k')/2}  n^{-3/2} 
+c n \sum_{k=1}^{+\infty} \frac{c}{n}  \e^{- \ga_{\infty} q_n^2/4}  \e^{ - \lambda k/2}  \, .
	\end{align*}
	All together, we obtain that  $\bbP(\cD_n)$ is bounded by a constant times $\e^{-q_n}$ and therefore goes to $0$ as $n\to +\infty$.
	To upgrade this to an almost sure statement, we use a monotonicity trick. We need to introduce
	\begin{align}
	\label{tildeDn}
	\tilde \cD_{\ell} &:= \bigcup_{ \substack{-2\ell \leq x<y\leq 2\ell \\ |x-y|\leq 2 \ell  \e^{- 5q_{\ell}} }}  \bigcup_{k\geq 0}  \bigcup_{k'\geq 0}  \tilde \cW^{(k)}_x \cap \tilde \cW_y^{(k')} \, ,
	\end{align}
	with $\tilde \cW_x^{(k)}:= \{ \rho_x > \tilde c\,  d_\ell \e^{-5q_{\ell}}\} \cap \{c_{x-1}, \frac{1}{c_{x+k}} > \tilde c \e^{q_{\ell}^2}\}$ for $k\geq 0$ (for simple traps the second condition is absent).
The constant $\tilde c$ is chosen large enough so that for all $n \in \{\ell, \ldots, 2\ell \}$ we have $ \tilde c\,  d_\ell \e^{-5 q_{\ell} } \leq  d_{n} \e^{-q_n}$, $ \tilde c\,  \e^{q_{\ell}^2} \leq  d_{n} \e^{- q_n^2}$.
	With this definition, we have that $\cD_n \subset \tilde \cD_\ell $ for all $\ell\leq n\leq 2\ell$.
	Then, as above, we obtain that $\bbP(\tilde \cD_\ell) \leq c \,\e^{- 5 q_{\ell}}$. Setting $n_\ell = \exp( (2\log \ell)^4 )$ so that $q_{n_{\ell}} = 2\log \ell$, we therefore  have that $\sum_\ell \bbP(\tilde \cD_{n_\ell}) <+\infty$, and by Borel-Cantelli there is some $\ell_0$ such that $\tilde \cD_{n_\ell}$ does not occur for $\ell\geq \ell_0$.
	Then, we realize that $n_{\ell+1} \sim n_{\ell}$ as $\ell \to +\infty$, so  $n_{\ell+1} \leq 2 n_{\ell}$ for all  $\ell$ large enough, say $\ell \geq \ell_1$. Therefore, for any $\ell\geq  \tilde \ell:=\max (\ell_0,\ell_1)$,  $\cD_{n_\ell}$ does not occur, and additionally  $\cD_n \subset \tilde \cD_{n_\ell}$ for any $n_{\ell} \leq n \leq 2 n_{\ell} \leq n_{\ell+1}$: we conclude that $\cD_{n}$ does not occur for any  $n\geq n_{\tilde \ell}$.

\subsection{Traps are not too deep: Proof of Lemmas	\ref{lem:maxrho} and \ref{lem:maxrho2}}\label{prooflem:maxrho}

	The case $k=0$ will include both the case of simple traps of Lemma \ref{lem:maxrho} and the case of $0$-distant well-and-wall traps of Lemma \ref{lem:maxrho2}.
	First of all, by \eqref{asymprhok} (and using Potter's bound), we have that
	\[
	\bbP\big( \rho_x^{(k)} > d_n \e^{q_n/2} \big)
	\leq c\,  \e^{-\ga \lambda  k/2}\psi(d_n) d_n^{-\ga} \e^{-\frac\ga 4 q_n} 
	\leq \frac{c'}{n} \e^{-\lambda \ga k/2}  \e^{-  \ga q_n /4}  \, .
	\]
	We therefore get thanks to a union bound that 
	$\bbP (  M_n > d_n \e^{q_n} ) \leq c'' \e^{- \ga q_n/4}$.
	Hence, setting $n_{\ell}:= \exp\{ (\tfrac{8}{\ga}\log \ell)^{4} \}$ (so that $q_{n_\ell} = \frac{8}{\ga} \log n$), we get by Borel-Cantelli that if $\ell$ is large enough, $M_{n_{\ell}} \leq d_{n_{\ell}} \e^{q_{n_{\ell}}/2}$. Then, since $M_n$ is increasing, we get that for all $ n_{\ell} \leq n\leq n_{\ell+1}$ we have $M_n \leq M_{n_{\ell+1}} \leq  2  d_{n} \e^{q_n/2}$ for $\ell$ large enough, since $n_{\ell+1}\sim n_{\ell}$ as $\ell\to+\infty$. We get Lemma~\ref{lem:maxrho2} by using that $2\leq \e^{q_n/2}$ for $n$ large enough.

\subsection{All triblocks are good: Proof of Lemmas	\ref{sogood} and \ref{sogood2}}\label{proofsogood}

The case $k=0$ will include both the case of simple traps of Lemma \ref{lem:maxrho} and the case of $0$-distant well-and-wall traps of Lemma \ref{lem:maxrho2}.
Let us estimate the probability, for any $x$, $k\geq 0$ and $y\neq x-1, x+k$, 
	\begin{align*}
	\bbP\big(  \rho_x^{(k)} > d_n\, \e^{-q_n} , c_{y} \notin [\e^{-4q_n},\e^{4 q_n}] \big)  
	& \leq \bbP \big(  \rho_x^{(k)} > d_n\, \e^{-q_n} \big) \big( \bbP(c_y > \e^{4 q_n} ) + \bbP(\tfrac{1}{c_y} > \e^{4 q_n}) \big)\\
	& \leq  c \e^{-\lambda \ga k/2} \psi(d_n) d_n^{-\ga} \e^{ 2\ga q_n}  \e^{- 3\ga q_n} \, .
	\end{align*}
	Here, we used that $c_y$ and $\rho_x^{(k)}$ are independent, and Potter's bound to get that $\psi(d_n \e^{q_n})\leq c\psi(d_n) \e^{\ga q_n}$ and that $\bbP(c_y > \e^{4 q_n} )\leq c \e^{- 3\ga q_n}$ (and similarly for $\bbP( \tfrac{1}{c_y} > \e^{4 q_n} )$).
For Lemma~\ref{sogood}, we additionally have to bound
\begin{align*}
\bbP\big(  \rho_x  > d_n\, \e^{-q_n} , c_{x} > \e^{4q_n}\big)   &\leq 
\bbP\big( c_{x-1}  > d_n \e^{3 q_n} \,  , c_{x} > \e^{4q_n}\big) \\
& \leq c L_\infty(d_n) d_n^{-\ga_{\infty}} \e^{-\ga q_n} \leq c \psi(d_n) d_n^{-\ga} \e^{-\ga q_n}\, .
\end{align*}
A similar bound holds for $\bbP\big(  \rho_x  > d_n\, \e^{-q_n} , c_{x-1} < \e^{-4q_n}\big) $.
	
Therefore, since by definition of $d_n$ we have $ \psi(d_n) d_n^{-\ga}\leq c'/n$, we get by a union bound
\begin{align*}
\bbP(\cG_n ) \leq c' C_n \sum_{k\geq 0}   \e^{-\lambda \ga k/2} \e^{- \ga q_n} \leq c'' \e^{- \ga q_n/2}\, ,
\end{align*}
so $\bbP(\cG_n )\to 0$ as $n\to+\infty$.

We can easily upgrade this to an a.s.\ statement, in the same manner as in Section~\ref{proofdisgiunti}, by introducing some appropriate event $\tilde \cG_\ell$ (analogously to \eqref{tildeDn}, by using the events $\{\rho_x^{(k)} > \tilde c d_n\e^{-q_n} \}$ and $\{ c_{y} \notin [ \tilde c \e^{-4q_n}, \tilde c \e^{4 q_n}]\}$ for some appropriate constant $\tilde c$), in such a way that $\cG_n \subset \tilde \cG_{\ell}$ for all $\ell \leq n \leq 2 \ell$.
Then, $\bbP$-a.s.,  $\tilde \cG_\ell$ occurs  finitely many times along the subsequence $n_\ell = \exp ( (\tfrac{4}{\ga} \log \ell)^{4})$ (so $q_{n_{\ell}}= \frac{4}{\ga} \log n$), and we get the conclusion of the lemma since $n_{\ell+1}\leq 2 n_{\ell}$ for $\ell$ large enough).

\subsection{$k$ is small: Proof of Lemma \ref{k-is-small}}\label{proofk-is-small}

Let us notice that for any $x$, and any $k \geq 0$, recalling \eqref{asymprhok}
\begin{align*}
\bbP\big( \cW_x^{(k)} \big) \leq \bbP\big( \rho_x^{(k)} > d_n \e^{-q_n} \big) \leq c\e^{- \lambda \ga k/2} \psi(d_n ) d_n^{-\ga}  \e^{2\ga q_n} \, ,
\end{align*}
where we used Potter's bound to get that $\psi(d_n \e^{-q_n}) \leq c \psi(d_n )\e^{\ga q_n}$.
Since $\psi(d_n ) d_n^{-\ga}\leq c/n$ by definition of $d_n$, we get by a union bound that
\begin{align*}
\bbP\big( \cK_n \big) \leq c'  \e^{2\ga q_n}\sum_{k\geq \tfrac6\lambda q_n} \e^{- \lambda \ga k/2} \leq c'' \e^{2\ga q_n} \e^{-3\ga q_n } = c'' \e^{-\ga q_n}\, ,
\end{align*}
so $\bbP(\cK_n )\to 0$ as $n\to+\infty$.
This is easily upgraded  to an a.s.\ statement, as in Section~\ref{proofdisgiunti}.

\subsection{Traps are the only annoying parts of the environment (\ref{assumptionA}): Proof of Lemma~\ref{annoying}}\label{proofannoying}

In this subsection, we work with Assumption~\ref{assumptionA}.
Let us estimate, for any $x\in \bbZ$ and $k\geq 1$, the probability 
\begin{align*}
\bbP\big( \rho_x^{(k)} > \gep_n \e^{-\lambda \ga k/2} d_n \,   ; \,\rho_x, \rho_{x+k} \leq  d_n \e^{- q_n}  \big)\,.
\end{align*}
We assume that we are in the case where $\bbP( \rho_0 > d_n ) \sim  c L_0(d_n) d_n^{-\ga_0}$: in particular $\ga_{\infty} \geq \ga_0 =\ga$ (and $\gamma_{\infty}\leq \gamma_0$ if $\ga_{\infty} = \ga_0$), and we have $\bbE[c_0^{\ga}]<+\infty$. The case $\bbP( \rho_0 >t ) \sim c L_{\infty} (t) t^{-\ga_{\infty}}$ is symmetric. 

Since $\{\rho_x<d_n\e^{-q_n}\}\subset \big\{\{c_x>\e^{\frac{q_n}{2}-\lambda}\}\cup\{c_{x-1}<d_n\e^{-q_n/2}\}\big\}$, it is enough to control two contributions.
The first contribution is
\begin{align}
\bbP\big( \rho_x^{(k)} > \gep_n \e^{-\lambda k/2}  d_n ;\,  c_{x} > & \e^{\frac{q_n}{2}-\lambda}\text{ or } \tfrac{1}{c_{x+k-1}} > \e^{\frac{q_n}{2}-\lambda}  \big)
	\leq 2 \bbP\big( \tfrac{c_{x-1}}{c_{x+k}} > \gep_n \e^{\lambda k/2}  d_n  \big) \e^{- \ga q_n /4} \notag\\
	& \leq  c  \e^{ - \lambda \ga k/4} \psi(d_n) d_n^{-\ga}  \gep_n^{-2\ga}  \e^{- \ga q_n /4} \leq \frac{c'}{n}  \e^{ - \lambda \ga k/4} \e^{- \ga q_n /5}\, .
\label{daje1}
\end{align}
Here, we used that $c_x, c_{x+k-1}$ are independent of $\rho_{x}^{(k)}$, together with Potter's bound and the definition of $d_n$. We also used the fact that $\gep_n^{-2\ga} = \e^{o(q_n)}$ from our choice of $\gep_n$.

The second contribution we need to control is, using $c_{x-1} \leq  \rho_x c_x$ and $\tfrac{1}{c_{x+k}} \leq \rho_{x+k} \tfrac{1}{c_{x+k-1}}$ (recall that in $\cH_n$ we have that $\rho_x,\rho_{x+k} \leq d_n \e^{-q_n}$),
\begin{align}
\label{daje2}
\bbP\Big(  \rho_x^{(k)} \geq  \gep_n  \e^{-\lambda k/2}  d_n   &\, ;\, c_{x-1} , \tfrac{1}{c_{x+k}} \, \leq d_n \e^{- \frac12 q_n}  \Big) 
\notag \\
&\leq \bbP  \Big( \tfrac{1}{c_{x+k}}  >  \tfrac{1}{c_{x-1}} \e^{ \lambda k/2}  \gep_n d_n   ; \e^{ \frac14 q_n} \leq c_{x-1} \leq d_n \e^{- \frac12 q_n}   \Big)
\, .
\end{align}
Indeed, we used that $c_{x-1} = c_{x+k} \e^{\lambda (k+1)}   \rho_x^{(k)}$, so that $c_{x-1} \geq \e^{\lambda k/2} \gep_n \e^{q_n/2} \geq \e^{q_n/4}$ (for $n$ large enough). Then, the idea is similar to the proof of Proposition~\ref{prop:conditional1}, but here we need a more quantitative estimate, which bring several technicalities. Conditioning with respect to $c_{x-1}$, we get that \eqref{daje2} is bounded by a constant times
\begin{align}
 \bbE\Big[ L_0& \Big( \e^{ \lambda k/2} d_n \gep_n /c_{x-1} \Big)   \e^{ - \lambda \ga k/2} d_n^{-\ga} \gep_n^{-\ga} c_{x-1}^{\ga} \ind_{\{ \e^{q_n/4} \leq c_{x-1} \leq d_n \e^{-q_n/2}   \}} \Big] \notag \\
& \leq c'  \e^{ - \lambda \ga k/4} \gep_n^{-2\ga}  d_n^{-\ga} 
 \bbE\Big[ L_0\Big( d_n /c_{x-1} \Big)  c_{x-1}^{\ga} \ind_{\{ \e^{q_n/4} \leq c_{x-1} \leq d_n \e^{-q_n/2}   \}} \Big]\, ,
\label{daje3}
\end{align}
where we used Potter's bound. 

\textbullet\ If $\ga= \ga_{0} <\ga_\infty$, we get by Potter's bound that for any $\eta>0$ (small enough)
\begin{align*}
\bbE\Big[ L_0 \big( d_n /c_{0} \big)  c_{0}^{\ga} \ind_{\{ e^{q_n/4} \leq c_{0} \leq d_n \e^{-q_n/2}   \}} \Big]
& \leq c L_0(d_n) \bbE\big[   (c_{0})^{\ga +\eta} \ind_{\{ c_{0} \geq  \e^{q_n/4}   \}} \big]\\
& \leq  c L_0(d_n) \e^{ - \eta' q_n} \, ,
\end{align*}
with $\eta' = (\ga -\ga_{\infty} +2\eta)/4>0$ (in particular $\ga+\eta < \ga_{\infty}$).
Plugged into \eqref{daje3}, we therefore get that 
\begin{equation}
\bbP\Big(  \rho_x^{(k)} \geq  \gep_n d_n  \, ;\, c_{x-1} , \tfrac{1}{c_{x+k}} \leq d_n \e^{-q_n/2}  \Big)  \leq  \frac{c}{n} \gep_n^{-2\ga} \e^{-\lambda\ga k/4} \e^{ - \eta' q_n} \, ,
\label{daje4}
\end{equation}
where we used that $L_0(d_n) d_n^{-\ga} \leq c/n$ (recall that $\bbP( \rho_0 >d_n ) \sim c L_0(d_n) d_n^{-\ga_0}$).

\textbullet\ If $\ga_0=\ga_{\infty}$, then we split
\begin{align}
\bbE\Big[ L_0 \big( d_n /c_{0} \big)  c_{0}^{\ga} \ind_{\{ \e^{q_n/4} \leq c_{0} \leq d_n \e^{-q_n/2}   \}} \Big] 
= &\, \bbE\Big[ L_0 \big( d_n /c_{0} \big)  c_{0}^{\ga} \ind_{\{ \e^{q_n/4} \leq c_{0} \leq \sqrt{d_n }   \}} \Big]  \notag\\
&+ \bbE\Big[ L_0 \big( d_n /c_{0} \big)  c_{0}^{\ga} \ind_{\{ \sqrt{d_n}\leq c_{0} \leq d_n \e^{-q_n/2}   \}} \Big] \, .
\label{daje5a}
\end{align}
For the first term, we use that $L_0(\e^t)$ is regularly varying to get that $L_0(d_n/c_0) \leq c L_0(d_n)$ for $ \e^{q_n/4} \leq c_0 \leq \sqrt{d_n}$ (since $ \frac12 \log d_n \leq  \log(d_n/c_0) \leq  \log d_n$).
Therefore, the first term in \eqref{daje5a} is bounded by a constant times $L_0 (d_n) \bbE\Big[ c_{0}^{\ga} \ind_{\{ c_{0}> \e^{q_n/4}  \}} \Big]$,
with
\begin{align}
\bbE\big[ c_{0}^{\ga} \ind_{\{ c_{0}> \e^{q_n/4}  \}} \big] & = \e^{\ga q_n/4} \bbP(c_0^{\ga} > \e^{\ga q_n/4}) +  \int_{\e^{\ga q_n/4}}^{+\infty} \bbP(  c_0^{\ga} >t ) \dd t \notag \\
& \sim L_{\infty}(\e^{q_n/4})   + \ga \int_{q_n/4}^{+\infty} L_{\infty}(\e^u) \dd u \, ,
\label{expectation}
\end{align}
where for the second line we used that $ \bbP(  c_0^{\ga} >t )\sim L_0(t^{1/\ga}) t^{-1}$ and then  a change of variable $t= \e^{\ga u}$.
Now, we use that $L_{\infty}(\e^u) = \gp_{\infty}(u) u^{\gamma_{\infty}}$ with $\gamma_{\infty} >-1$ to get that the last integral is bounded by a constant times $ \gp_{\infty}(q_n) q_n^{1+\gamma_{\infty}}$ (and so is the first term).

For the second term in \eqref{daje5a}, we write
\begin{align}
\bbE\Big[ L_0 \big( d_n /c_{0} \big)  c_{0}^{\ga} \ind_{\{ \sqrt{d_n} \leq c_{0} \leq d_n \e^{-q_n}  \}} \Big] 
	&\leq\sum_{j= q_n }^{\frac12 \log d_n} L_0 (\e^j)  (\e^{-j} d_n)^{\ga} \bbP(c_0 \in  [\e^{-(j+1)}d_n, \e^{-j} d_n ]) \notag\\
	& \leq c L_{\infty}(d_n) \sum_{j= q_n }^{\frac12 \log d_n} L_0 (\e^j)  \, .
\end{align}
For the last inequality, we used that $\bbP(c_0 \in  [\e^{-(j+1)}d_n, \e^{-j} d_n ]) \leq c L_{\infty}(\e^{-j}d_n) (\e^{-j} d_n)^{-\ga}$, together with the fact that $ L_{\infty}(\e^{-j}d_n) \leq c L_{\infty}(d_n)$ for the range of $j$ considered (since $L_{\infty}(\e^t)$ is regularly varying).
Then, since $L_0(\e^{j})\sim \gp_0(j) j^{\gamma_0}$ with $\gamma_0 \neq 1$,  we get that  $\sum_{j= q_n }^{\frac12 \log d_n} L_0 (\e^j)$ is bounded by a constant times $\gp_0(\log d_n) (\log d_n)^{1+\gamma_0} \leq c (\log n) L_0(d_n)$ if $\gamma_0>-1$; by $\gp_0(q_n)  q_n^{1+\gamma_0}$ if $\gamma_0 <-1$.

All together, we get that \eqref{daje5a} is bounded by a constant times
 \begin{align*}
 L_0(d_n) \gp_{\infty}(q_n) q_n^{1+\gamma_{\infty}}  + 
 \begin{cases}
L_0(d_n){\infty} \times (\log n) L_\infty(d_n) & \quad \text{if } \gamma_0> -1 \, ; \\
L_{\infty}(d_n) \times \gp_0(q_n) q_n^{1+\gamma_0} & \quad \text{if } \gamma_0<-1 \, . 
 \end{cases}
 \end{align*}
In the case $\gamma_0 >-$ ($\bbE[1/c_0^{\infty}] =+\infty$), then necessarily $\gamma_{\infty}< -1$ (we need to have $\bbE[c_0^{\infty}] <+\infty$): we may bound   $(\log n) L_{\infty}(d_n) \leq   c \gp_{\infty} (\log n) (\log n)^{1+\gamma_{\infty}} \leq c' gp_{\infty} (q_n) q_n^{1+\gamma_{\infty}}$ (recall $q_n =(\log n)^1/4$).
In the case $\gamma_0 <-1$, since we assume that $\bbP(\rho_0>d_n) \sim c L_0(d_n) d_n^{-\ga}$ it means in particular that $L_{\infty}(d_n) \leq L_{0}(d_n)$.
Overall, using again that $L_0(d_n) d_n^{-\ga} \leq c/n$, we get from \eqref{daje3} that
\begin{equation}
\bbP\Big(  \rho_x^{(k)} \geq  \gep_n d_n  \, ; \, c_{x-1} , \tfrac{1}{c_{x+k}} \leq d_n \e^{-q_n/2}  \Big)  \leq  \frac{c}{n} \gep_n^{-2\ga} \e^{-\lambda\ga k/4}  q_n^{-  c} \, ,
\label{daje5}
\end{equation}
for some constant $c>0$;  $c=-(1+\gamma_{\infty})/2$ if  $\gamma_0 >-1$ and $c = - (1+\gamma_0)/2 $ if $\gamma_0<-1$ (recall $\gamma_{\infty}\leq \gamma_0$).

Finally, by a union bound we get that,
thanks to \eqref{daje1} and \eqref{daje4}-\eqref{daje5}
\begin{align*}
\bbP\big( \cH_n \big) \leq c \gep_n^{-2\ga} 
\e^{- \eta' q_n} \quad \text{ if } \ga_0<\ga_{\infty} \, ; \qquad 
\bbP\big( \cH_n \big) \leq c \gep_n^{-2\ga} 
q_n^{-c} &\quad \text{ if } \ga_0=\ga_{\infty}\, .
\end{align*}
This concludes the proof if we had chosen $\gep_n = q_n^{- \gd}$  with $\gd>0$ small enough.

\subsection{Traps are the only annoying parts of the environment (\ref{assumptionB}): Proof of Lemma~\ref{annoying2}} \label{proofannoying2}
We work under Assumption \ref{assumptionB}. In this case, because of \eqref{rhowellandwall} and of the definition\eqref{def:dn} of $d_n$, we have that (recall Assumption~\ref{mainassumption})
\begin{align}\label{giovanni}
\gp_0(\log n)\gp_\infty(\log n)\,(\log n)^{1+\gamma_0+\gamma_\infty}d_n^{-\ga}\sim c\,n^{-1}\, ,
\end{align}
 with $\gamma_0, \gamma_{\infty}>-1$.
We want to estimate 
$
\bbP\big( \rho_x^{(k)} > \varepsilon_n\e^{-\lambda k/2}  d_n \, ;  c_{x-1} <\e^{q_n^2} \mbox{ or }   c_{x+k}^{-1} <\e^{q_n^2} \big)\,.
$

We will deal with $\bbP( \rho_x^{(k)} > \varepsilon_n\e^{-\lambda k/2}  d_n \, ;  c_{x+k}^{-1} <\e^{q_n^2})$, the remaining part being similar.  First of all, using \eqref{giovanni}, we see that, using Potter's bound
\begin{align*}
\bbP( \tfrac{c_{x-1}}{c_{x+k}} > \varepsilon_n\e^{\gl k/2} d_n\,; \; c_{x+k}^{-1} <1)
	& \leq \bbP( {c_{x-1}} > \varepsilon_n\e^{\gl k/2} d_n)\\
	&\leq \frac{c}{n} \gp_0(\log n)^{-1} (\log n)^{-(1+\gamma_0)} \gep_n^{-2\ga} \e^{-\ga\gl k/4} \, .
\end{align*}
Provided that $\gd$ has been fixed small enough, summing this quantity over $x\in[-n,n]$ and over $k\geq 0$ we obtain something that goes to $0$ as $n\to\infty$ (recall $\gep_n = q_n^{-\gd} = (\log n)^{-\gd /4}$).
We can therefore restrict to the event  $c_{x+k}^{-1}>1$. We bound
\begin{align}\label{spallotto}
\bbP( \tfrac{c_{x-1}}{c_{x+k}} > \varepsilon_n \e^{\gl k/2}d_n;& \; 1<c_{x+k}^{-1} <\e^{q_n^2})
= \E\Big[\P\big(\tfrac{c_{x-1}}{c_{x+k}} > \varepsilon_n\e^{\gl k/2} d_n\,\big|\,c_{x+k}\big)\ind_{\{1<c_{x+k}^{-1}<\e^{q_n^2}\}}\Big]\notag\\
&\leq d_n^{-\ga}\varepsilon_n^{-2\ga} \e^{-\ga\gl k/4} \E\Big[L_\infty( d_nc_{x+k})c_{x+k}^{-\ga} \ind_{\{1<c_{x+k}^{-1}<\e^{q_n^2}\}}\Big]\,.
\end{align}
Since $L_\infty(d_n c_{x+k})=\gp_\infty\big(\log(d_n c_{x+k})\big)\big(\log (d_n c_{x+k})\big)^{\gamma_\infty}\leq c'\gp_\infty(\log n )(\log n )^{\gamma_\infty}$ under the condition that $1<c_{x+k}^{-1}<\e^{q_n^2}$, we are left to control the last expectation in~\eqref{spallotto}. Analogously to~\eqref{expectation}:
\begin{align*}
\E\Big[c_{x+k}^{-\ga} \ind_{\{1<c_{x+k}^{-1}<\e^{q_n^2}\}}\Big]
	&\leq 1 + \int_{1}^{\e^{\ga q_n^2}} \P(c_{x+k}^{-\ga}>t) \,{\rm d}t\\
	&\leq  1+ c \int_{0}^{q_n^2}L_0 (\e^u) \dd u \leq  c' \gp_0 (q_n^2) q_n^{2(1+\gamma_0)}\, .
\end{align*}
Putting this back into \eqref{spallotto} (and using Potter's bound), we obtain
\begin{align*}
\bbP( \tfrac{c_{x-1}}{c_{x+k}} > \varepsilon_n \e^{\gl k/2}d_n\, ; 1<c_{x+k}^{-1} <\e^{q_n^2})
	&\leq c''d_n^{-\ga} \gep_n^{-2\ga} \e^{-\ga\gl k/4}\gp_\infty(\log n )(\log n )^{\gamma_\infty} q_n^{2(1+\gamma_0)} \gp_0 (q_n^2) \\
	&\leq \frac{c'''}{n} \e^{-\ga\gl k/4} (\log n)^{- \frac12  (1+\gamma_0)}  \frac{ \gp_0(q_n^2)}{\gp_0(\log n )} \gep_n^{- 2\ga}  ,
\end{align*}
where we have used \eqref{giovanni} and the fact that $q_n =(\log n)^{1/4} $.
Provided that $\gd$ has been fixed small enough (recall $\gep_n = q_n^{-\gd}$),  summing the last quantity over $x\in[-n,n]$ and over $k\geq 0$ we obtain again something that goes to $0$ as $n\to\infty$, concluding the proof.

\end{appendix}

\bibliography{bibliografia}
\bibliographystyle{amsalpha}

\end{document}